\documentclass[12pt]{amsart}

\usepackage{custom_preamble}

\begin{document}

\title[Bounds in the idempotent theorem]{Bounds in Cohen's idempotent theorem}

\author{\tsname}
\address{\tsaddress}
\email{\tsemail}

\begin{abstract} 
Suppose that $G$ is a finite Abelian group and write $\mathcal{W}(G)$ for the set of cosets of subgroups of $G$.  We show that if $f:G \rightarrow \Z$ has $\|f\|_{A(G)} \leq M$ then there is some $z:\mathcal{W}(G) \rightarrow \Z$ such that
\begin{equation*}
f=\sum_{W \in \mathcal{W}(G)}{z(W)1_W} \text{ and } \|z\|_{\ell_1(\mathcal{W}(G))} =\exp(M^{4+o(1)}).
\end{equation*}
\end{abstract}

\maketitle

\section{Introduction}\label{sec.intro}

This paper is about quantitative aspects of Cohen's idempotent theorem \cite[Theorem 3]{coh::} (stated here as Theorem \ref{thm.cidem}).  To state our results precisely we shall need some notation and basic results.

Suppose that $G$ is a finite Abelian group.  We write $\wh{G}$ for its dual group, that is the finite Abelian group of homomorphisms $G \rightarrow S^1$ where $S^1:=\{z\in \C: |z|=1\}$.  We regard $G$ as endowed with a Haar probability measure $m_G$ (this is simply the measure assigning mass $|G|^{-1}$ to each element of $G$) so that we can then define the Fourier transform of a function $f \in L_1(m_G)$ to be
\begin{equation*}
\wh{f}:\wh{G}\rightarrow \C; \gamma \mapsto \int{f(x)\overline{\gamma(x)}dm_G(x)}.
\end{equation*}
We shall be interested in the Fourier algebra norm of functions, and this is defined by
\begin{equation*}
\|f\|_{A(G)}:=\|\wh{f}\|_{\ell_1(\wh{G})}=\sum_\gamma{|\wh{f}(\gamma)|}.
\end{equation*}
It is an easy calculation to see that if $H \leq G$ then
\begin{equation*}
\wh{1_H}(\gamma)=\begin{cases} m_G(H) & \text{ if } \gamma(h)=1 \text{ for all } h \in H\\ 0 & \text{ otherwise,}\end{cases}
\end{equation*}
and it follows from this and Parseval's theorem (see (\ref{eqn.pars}) in \S\ref{sec.fa} if unfamiliar) that
\begin{equation}\label{eqn.calc}
\|1_H\|_{A(G)}=\sum_{\gamma \in \wh{G}}{|\wh{1_H}(\gamma)|} = \frac{1}{m_G(H)}\sum_{\gamma \in \wh{G}}{|\wh{1_H}(\gamma)|^2} =\frac{1}{m_G(H)}\int{1_H^2dm_G}= 1.
\end{equation}
Write $\mathcal{W}(G):=\bigcup_{H \leq G}{G/H}$ and suppose that $z:\mathcal{W}(G)\rightarrow \Z$.  Then
\begin{equation*}
f:=\sum_{W \in \mathcal{W}(G)}{z(W)1_W}
\end{equation*}
is integer-valued and has
\begin{equation*}
\Im f \subset \Z \text{ and } \|f\|_{A(G)} \leq \|z\|_{\ell_1(\mathcal{W}(G))}.
\end{equation*}
Our main result is the following weak converse.
\begin{theorem}\label{thm.main}
Suppose that $M \geq 1$.  Then for all finite Abelian groups $G$ and functions $f:G \rightarrow \Z$ with $\|f\|_{A(G)} \leq M$ there is some $z:\mathcal{W}(G) \rightarrow \Z$ such that
\begin{equation*}
f=\sum_{W \in \mathcal{W}(G)}{z(W)1_W} \text{ and } \|z\|_{\ell_1(\mathcal{W}(G))} \leq \exp\left(M^{4+o(1)}\right).
\end{equation*}
\end{theorem}
This may be compared with \cite[Theorem 1.3]{gresan::0} which gives a bound of $\exp(\exp(O(M^4)))$.  On the other hand long arithmetic progressions show that we cannot do much better:\footnote{\textbf{Notational warning}: here and elsewhere we follow Knuth's definition \cite[p19]{knu::} of $\Omega$ rather than Hardy and Littlewood's \cite[p225]{harlit::1}.  Specifically, for us $f=\Omega(g)$ is equivalent to $g=O(f)$.}
\begin{proposition}\label{prop.egg}
Suppose that $M \geq 1$.  Then there is a finite Abelian group $G$ and a function $f:G \rightarrow \Z$ with $\|f\|_{A(G)} \leq M$ such that if $z:\mathcal{W}(G) \rightarrow \Z$ has
\begin{equation*}
f=\sum_{W \in \mathcal{W}(G)}{z(W)1_W} \text{ then } \|z\|_{\ell_1(\mathcal{W}(G))} =\Omega\left(\exp\left(\frac{\pi^2}{4}M\right)\right).
\end{equation*}
\end{proposition}
\begin{proof}
The characters on $G=\Z/N\Z$ are exactly the functions of the form $x \mapsto \exp(2\pi i jx/N)$ for $1 \leq j \leq N$ and so for $N,n>1$ writing $I_N:=\{m+N\Z: -n \leq m \leq n\}$ and inserting the computation of the Lebesgue constants due to Fej{\'e}r \cite[(16.)]{fej::0} we have
\begin{align*}
\lim_{N \rightarrow \infty}{\|1_{I_N}\|_{A(\Z/N\Z)}}&=\lim_{N \rightarrow \infty}{\frac{1}{N}\sum_{j=1}^N{\left|\sum_{m=-n}^n{\exp\left(2\pi i\frac{mj}{N} \right)}\right|}}\\
&= \int_0^1{\left|\sum_{m=-n}^n{\exp(2\pi im\theta )}\right|d\theta}= \frac{4}{\pi^2}\log n + O(1).
\end{align*}
Since there are infinitely many primes it follows that for all $n\in \N$ there is some prime $N \geq 4n+2$ such that $G:=\Z/N\Z$ contains a set $A$ of size $2n+1$ with $\|1_A\|_{A(G)} \leq \frac{4}{\pi^2}\log n + O(1)$.  Since $N$ is prime we see that any representation of $1_A$ in terms of a function $z$ of the required type must have $\|z\|_{\ell_1(\mathcal{W}(G))} \geq |A|$ from which we get the result.
\end{proof}
In fact Fej{\'e}r's calculation in \cite[(17.)]{fej::0} includes a determination of the $O(1)$ term in the form $c_0+\frac{c_1}{n} + o_{n \rightarrow \infty}(n^{-1})$ so that the constant behind the $\Omega$ can be computed rather accurately if desired, and Watson in \cite{wat::} went even further with the asymptotic expansion using Szeg{\H o}'s beautiful formula for the Lebesgue constants in \cite{sze::4}.  

Proving our main result in the setting of general finite Abelian groups rather than Abelian groups of bounded exponent adds a number of difficulties.  To help understand the overarching method we have presented Theorem \ref{thm.main} in the case when $G$ is a group of exponent $2$ in \cite{san::18}, where the simplifications also lead to a better bound.  We state this result explicitly in \S\ref{sec.rel} along with some results from other classes of group where more can be said.

\subsection{Applications and connections}

Although some similarity may already be clear at this stage, we explicitly connect our work to Cohen's idempotent theorem in \S\ref{sec.con}.  One of the applications of this is to describe the algebra homomorphisms $L_1(G_1) \rightarrow M(G_2)$ where $G_1$ and $G_2$ are locally compact Abelian groups.  The rough idea is to note that such a map must arise as the pullback of a function between the dual groups whose graph has small algebra norm.  The details may be found in \cite[\S4.1.3]{rud::1}.

Wojciechowski \cite{woj::0}, and then Czuron and Wojciechowski \cite{czuwoj::}, made use of quantitative information from the idempotent theorem to strengthen consequences of the results above about non-existence of algebra homomorphisms into `local' results about the norms of maps between finite dimensional subspaces.  Stronger quantitative information in the present paper can be inserted directly to give stronger information there.

As a last connection to other work we mentioned that there is a quantitative connection between the coset ring (defined just before Theorem \ref{thm.cidem}) and the stability ring of Terry and Wolf \cite{terwol::0,terwol::1}.

\subsection{Outline of the paper}

Before moving on to the rest of the paper we should discuss the structure and notation, and a little about the contribution.  The overarching structure is the same as that of \cite{gresan::0}.  In \S\ref{sec.cn}, \S\ref{sec.bs}, \S\ref{sec.ain}, \S\ref{sec.aa} and \S\ref{sec.fa}, we set up the basic background theory we shall need which is for much the same purpose as in \cite{gresan::0}.  Notation and definitions are set up and made as needed. In particular, the two different types of covering number we use are defined in \S\ref{sec.cn}; Bohr sets and their various types of dimension are defined in \S\ref{sec.bs}; notation for measures and convolutions at the start of \S\ref{sec.ain}; and approximate annihilators at the start of \S\ref{sec.aa}.

There were three main parts to the argument in \cite{gresan::0}, and essentially the first two of them introduce a need for a doubly (rather than singly) exponential bound in \cite[Theorem 1.3]{gresan::0}.  The main contribution of this paper is to note how these can be removed.

The first part of the argument in \cite{gresan::0} was a sort of quantitative continuity result developed from the work of Green and Konyagin in \cite{grekon::}.  Our analogue of this is in \S\ref{sec.qc} and is closely related to their work, although here we make use of an advance due to Croot, Sisask and {\L}aba \cite{croabasis::} to get a sort of $L_p$ version.

The second ingredient was a Freiman-type theorem.  Freiman's theorem has been improved since then to have quasi-polynomial dependencies and our work simply takes advantage of this.  We record a suitable Freiman-type theorem in \S\ref{sec.ft}.

The third ingredient is the concept of arithmetic connectivity. We refine this in \S\ref{sec.ac}, but the improvement it leads to is polynomial rather than exponential.  (Without any change to the notion of arithmetic connectivity from \cite{gresan::0} our arguments lead to Theorem \ref{thm.main} with the $4+o(1)$ replaced by some larger constant.)

These three main ingredients are combined in the argument in \S\ref{sec.main} to give Theorem \ref{thm.main2} which has Theorem \ref{thm.main} as a special case.

\subsection{Limitations of the argument}

As with the argument in \cite{gresan::0}, though for different reasons, the argument for Theorem \ref{thm.main} has two separate points, both of which force bounds of the shape we get.  The first point is in Proposition \ref{prop.screl}, the core of which goes back to Green and Konyagin \cite{grekon::}.  Whilst we improve one dependency, the other dependencies have not been touched since their work.

The second point is in Proposition \ref{prop.f}.  Here there is a well-known conjectural improvement -- the polynomial Freiman-Ruzsa conjecture -- although it doesn't seem like such an improvement is altogether necessary.  In particular, it seems quite realistic to hope to improve Lemma \ref{lem.s} directly.

\section{Covering numbers}\label{sec.cn}

Given two sets $S,T \subset G$ with $T$ non-empty, the \textbf{covering number} of $S$ by $T$ is
\begin{equation*}
\mathcal{C}_G(S;T):=\min\left\{|X| : S \subset X+T\right\}.
\end{equation*}
We often omit the subscript if the underlying group is clear.

Since $T$ is non-empty and $G$ is finite this minimum is well-defined.  Moreover, if $S$ is also non-empty then $\mathcal{C}(S;T) \geq 1$ whatever the set $T$.  

Covering numbers enjoy the following simple properties.
\begin{lemma}[Behaviour of covering numbers]\label{lem.bpcn}
Suppose that $G$ and $H$ are Abelian groups.
\begin{enumerate}
\item\label{pt.bpcn.1} \emph{(Restrictions and extensions)} For all $U \supset S$ and $T \supset V \neq \emptyset$ we have
\begin{equation*}
\mathcal{C}(S;T) \leq \mathcal{C}(U;V).
\end{equation*}
\item\label{pt.prod} \emph{(Products)} For all $S,T \subset G$ and $U,V \subset H$ with $T, V \neq \emptyset$ we have
\begin{equation*}
\mathcal{C}_{G \times H}(S\times U;T \times V) \leq \mathcal{C}_G(S;T)\mathcal{C}_H(U;V).
\end{equation*}
\item\label{pt.extn} \emph{(Compositions)} For all $S,T,U$ with $T,U \neq \emptyset$ we have
\begin{equation*}
\mathcal{C}(S;U) \leq \mathcal{C}(S;T)\mathcal{C}(T;U).
\end{equation*}
\item\label{pt.pullback} \emph{(Pullbacks)} For all $U,V \subset H$ with $V \neq \emptyset$ and homomorphisms $\phi:G \rightarrow H$ we have
\begin{equation*}
\mathcal{C}_G(\phi^{-1}(U);\phi^{-1}(V-V)) \leq \mathcal{C}_H(U;V).
\end{equation*}
\end{enumerate}
\end{lemma}
\begin{proof}
First, if $U \subset X+V$ and $U \supset S$ and $T \supset V$ then certainly $S \subset X+T$ from which (\ref{pt.bpcn.1}) follows.

Secondly, if $S \subset X+T$ and $U \subset Y+V$ then $S \times U \subset X \times Y + T \times V$ and (\ref{pt.prod}) follows.

Thirdly, if $S \subset X+T$ and $T \subset Y+U$ then $S \subset X+Y+U$ and hence $\mathcal{C}(S,U) \leq |X+Y| \leq |X||Y|$ from which (\ref{pt.extn}) follows.

Finally, if $U \subset X+V$ then write $X'$ for the set of $x \in X$ such that $(x+V)\cap \phi(G) \neq \emptyset$ and let $z:X' \rightarrow G$ be a choice function such that $\phi(z(x)) \in x+V$.  Put $Z:=\{z(x):x \in X'\}$.  If $y \in \phi^{-1}(U)$ then
\begin{equation*}
\phi(y) \in (X+V)\cap \phi(G) \subset X' + V \subset \phi(Z)-V+V.
\end{equation*}
It follows that $y \in Z + \phi^{-1}(V-V)$ and we have (\ref{pt.pullback}) since $|Z| \leq |X'| \leq |X|$.
\end{proof}
Covering numbers are closely related to doubling as the following lemma captures.
\begin{lemma}\label{lem.covsum}
Suppose that $A,B,S,T \subset G$ with $B,T \neq \emptyset$.  Then
\begin{equation*}
m_G(A+S) \leq \mathcal{C}(A;B)\mathcal{C}(S;T)m_G(B+T).
\end{equation*}
\end{lemma}
\begin{proof}
Let $X$ be such that $A \subset X+B$ and $|X| = \mathcal{C}(A,B)$, and $Y$ be such that $S \subset Y+T$ and $|Y|=\mathcal{C}(S,T)$.  Then $A+S \subset X+Y+B+T$ and hence
\begin{equation*}
m_G(A+S)\leq m_G(X+Y+B+T) \leq |X||Y|m_G(B+T) \leq \mathcal{C}(A,B)\mathcal{C}(S,T)m_G(B+T),
\end{equation*}
and the lemma is proved.
\end{proof}
Conversely we have Ruzsa's covering lemma.
\begin{lemma}[Ruzsa's covering lemma]\label{lem.rc}
Suppose that $A,B \subset G$ for some $B \neq \emptyset$.  Then
\begin{equation*}
\mathcal{C}(A;B-B) \leq \frac{m_G(A+B)}{m_G(B)}.
\end{equation*}
\end{lemma}
\begin{proof}
Suppose that $X \subset A$ is maximal such that for every distinct $x,x' \in X$ we have $(x+B) \cap (x'+B) = \emptyset$.  It then follows that if $x \in A\setminus X$, there is some $x' \in X$ such that $(x+B) \cap (x'+B) \neq \emptyset$, and hence $A\setminus X \subset X+B-B$.  Of course, since $0_G \in B-B$ we certainly have $X \subset X+ B-B$ and so $A \subset X+B-B$.  On the other hand, the sets $\{x+B: x \in X\}$ are disjoint subsets of $A+B$ and there are $|X|$ of them.  The lemma follows.
\end{proof}

In the light of Lemma \ref{lem.bpcn} part (\ref{pt.pullback}) above, for sets $S,T \subset G$ with $0_G \in T$ it is natural to define the \textbf{difference covering number} of $S$ by $T$ to be
\begin{equation*}
\mathcal{C}_G^\Delta(S;T) :=\min\left\{\mathcal{C}_H(U;V):\begin{array}{l}H \in \textbf{Ab}, H \text{ finite}, \phi \in \Hom(G,H),\\ S \subset \phi^{-1}(U), \phi^{-1}(V-V) \subset T\end{array}\right\},
\end{equation*}
where $\textbf{Ab}$ denotes the category of Abelian groups and $\Hom(G,H)$ is the set of homomorphisms between $G$ and $H$.  As before we often omit the subscript if the underlying group is clear.

Again, since $0_G \in T$ the minimum above is well-defined, and if $S$ is non-empty then $\mathcal{C}_G^\Delta(S;T)  \geq 1$.

For our purposes difference covering numbers turn out to behave slightly better than covering numbers.
\begin{lemma}[Behaviour of difference covering numbers]\label{lem.dfc}\
\begin{enumerate}
\item\label{p4.dfc} \emph{(Restrictions and extensions)} For all $S' \supset S$ and $T \supset T' \ni 0_G$ we have
\begin{equation*}
\mathcal{C}^\Delta(S;T) \leq \mathcal{C}^\Delta(S';T').
\end{equation*}
\item\label{pt2.difc} \emph{(Intersections)} For all $S,S',T,T'$ with $T,T' \ni 0_G$ we have
\begin{equation*}
\mathcal{C}^\Delta\left(S\cap S';T \cap T'\right) \leq \mathcal{C}^\Delta\left(S;T\right)\mathcal{C}^\Delta\left(S';T'\right).
\end{equation*}
\item\label{pt3.difc} \emph{(Domination by coverings numbers)} For all $S,T$ we have
\begin{equation*}
\mathcal{C}^\Delta(S;T-T) \leq\mathcal{C}(S;T).
\end{equation*}
\item\label{pt4.difc} \emph{(Domination of coverings numbers)} For all $S,T$ with $T \ni 0_G$ we have
\begin{equation*}
\mathcal{C}(S;T) \leq \mathcal{C}^\Delta(S;T).
\end{equation*}
\end{enumerate}
\end{lemma}
\begin{proof}
First, (\ref{p4.dfc}) follows immediately from the definition of the difference covering number.

Secondly, suppose that $\phi \in \Hom(G,H)$ and $\psi \in \Hom(G,H')$, and $U,V \subset H$ have $\mathcal{C}_H(U;V)=\mathcal{C}_G^\Delta(S;T)$ and $U',V' \subset H$ have $\mathcal{C}_{H'}(U';V')=\mathcal{C}_G^\Delta(S';T')$, are all such that
\begin{equation*}
S \subset \phi^{-1}(U), \phi^{-1}(V-V) \subset T, S' \subset \psi^{-1}(U'), \text{ and } \psi^{-1}(V'-V') \subset T'.
\end{equation*}
The map $\phi\times \psi$ is a group homomorphism $G \rightarrow H \times H'$ (defined by $x\mapsto (\phi(x),\psi(x))$).  Moreover,
\begin{equation*}
S \cap S' \subset \phi^{-1}(U) \cap \psi^{-1}(U') = (\phi \times \psi)^{-1}(U \times U')
\end{equation*}
and
\begin{align*}
(\phi\times \psi)^{-1}(V \times V' - V \times V') & = (\phi\times \psi)^{-1}((V-V) \times (V'-V'))\\ & = \phi^{-1}(V-V) \cap \psi^{-1}(V'-V') \subset T \cap T'.
\end{align*}
By the definition of the difference covering number and Lemma \ref{lem.bpcn} (\ref{pt.prod}) we have that
\begin{align*}
\mathcal{C}^\Delta_G(S\cap S';T\cap T') & \leq \mathcal{C}_{H \times H'}(U \times U'; V \times V')\\
& \leq \mathcal{C}_H(U;V)\mathcal{C}_{H'}(U';V') = \mathcal{C}_G^\Delta(S;T)\mathcal{C}_G^\Delta(S';T').
\end{align*}
Part (\ref{pt2.difc}) is proved.

Thirdly, let $\phi:G\rightarrow G$ be the identity homomorphism, $U:=S$ and $V:=T$ so that $S \subset \phi^{-1}(U)$ and $\phi^{-1}(V-V) \subset T-T$.  It follows that
\begin{equation*}
\mathcal{C}^\Delta(S;T-T) \leq \mathcal{C}_G(U;V) = \mathcal{C}_G(S;T)
\end{equation*}
and (\ref{pt3.difc}) is proved.

Finally, let $\phi \in \Hom (G,H)$ and $U,V \subset H$ be such that $S \subset \phi^{-1}(U)$ and $\phi^{-1}(V-V) \subset T$ and $\mathcal{C}_H(U;V)=\mathcal{C}^\Delta_G(S;T)$.  Then by Lemma \ref{lem.bpcn} (\ref{pt.bpcn.1}) and (\ref{pt.pullback}) we see that
\begin{equation*}
\mathcal{C}_G(S;T) \leq \mathcal{C}_G(\phi^{-1}(U);\phi^{-1}(V-V)) \leq \mathcal{C}_H(U;V)=\mathcal{C}^\Delta_G(S;T).
\end{equation*}
This gives (\ref{pt4.difc}).  
\end{proof}
It will also be useful to have a version of Ruzsa's covering lemma for difference covering numbers.
\begin{lemma}[Ruzsa's covering lemma, revisited]\label{lem.rcr}
Suppose that $A,B,X \subset G$ with both $X\neq \emptyset$ and $0_G \in B$.  Then
\begin{equation*}
\mathcal{C}^\Delta(A;B) \leq \frac{m_G(A+X)}{m_G(X)}\mathcal{C}^\Delta(X-X;B).
\end{equation*}
\end{lemma}
\begin{proof}
Let $H$ be an Abelian group, $\phi \in \Hom (G,H)$ and $U,V \subset H$ be such that $\phi^{-1}(U) \supset X-X$ and $\phi^{-1}(V-V) \subset B$.  By Ruzsa's covering lemma (Lemma \ref{lem.rc}) we see that there is some set $T$ with
\begin{equation*}
|T| \leq \frac{m_G(A+X)}{m_G(X)} \text{ and } A \subset T+X-X.
\end{equation*}
Let $U':=\phi(T)+U$ so that $\mathcal{C}_H(U';V) \leq |T|\mathcal{C}_H(U;V)$.  On the other hand $\phi^{-1}(U') \supset T+X-X \supset A$ and the result follows.
\end{proof}

\section{Bohr systems}\label{sec.bs}

Bohr sets interact particularly well with covering numbers and difference covering numbers.  We write $\|\cdot\|$ for the map $S^1 \rightarrow [0,\frac{1}{2}]$ defined by
\begin{equation*}
\|z\|:=\min\{|\theta|: z=\exp(2\pi i \theta)\}.
\end{equation*}
It is easy to check that this is well-defined and that the map $(z,w) \mapsto \|zw^{-1}\|$ is a translation-invariant metric on $S^1$.  Given a set of characters $\Gamma$ on $G$, and a function $\delta:\Gamma \rightarrow \R_{>0}$, then we write
\begin{equation*}
\Bohr(\Gamma,\delta):=\left\{x \in G: \| \gamma(x)\|<\delta(\gamma) \text{ for all } \gamma \in \Gamma\right\},
\end{equation*}
and call such a set a (generalised\footnote{We call these \emph{generalised} Bohr sets because usually (\emph{e.g.} \cite[Definition 4.6]{taovu::}) Bohr sets are defined using only the constant functions; we use this more general definition to ensure that the intersection of two Bohr sets is a Bohr set, but quite apart from being a natural extension this is by no means the first time this has been done (see \emph{e.g.} \cite[(0.11)]{bou::1} and \cite[Definition 5.1]{ruz::03}).}) \textbf{Bohr set}. 

In fact we shall not so much be interested in Bohr sets as \emph{families} of Bohr sets.  A \textbf{Bohr system} is a vector $B=(B_\eta)_{\eta \in (0,1]}$ for which there is a set of characters $\Gamma$ and a function $\delta:\Gamma \rightarrow \R_{>0}$ such that
\begin{equation*}
B_\eta=\Bohr(\Gamma,\eta\delta) \text{ for each } \eta \in (0,1].
\end{equation*}
We say that $B$ is \textbf{generated} by $(\Gamma,\delta)$ and, of course, the same Bohr system may be generated by different pairs.

This definition is motivated by that of Bourgain systems \cite[Definition 4.1]{gresan::0}, although it is in some sense `smoother'.  (In this paper what we mean by this is captured by Lemma \ref{lem.grow} which does not hold for Bourgain systems.)

We first record some trivial properties of Bohr systems; their proof is left to the reader.
\begin{lemma}[Properties of Bohr systems]\label{lem.pbs}
Suppose that $B$ is a Bohr system.  Then
\begin{enumerate}
\item \emph{(Identity)} $0_G \in B_\eta$ for all $\eta \in (0,1]$;
\item \emph{(Symmetry)} $B_\eta = -B_\eta$ for all $\eta \in (0,1]$;
\item \emph{(Nesting)} $B_{\eta}\subset B_{\eta'} $ whenever $0<\eta\leq \eta'\leq 1$;
\item \emph{(Sub-additivity)} $B_{\eta} + B_{\eta'} \subset B_{\eta + \eta'}$ for all $\eta,\eta'\in (0,1]$ with $\eta+\eta' \leq 1$.
\end{enumerate}
\end{lemma}
\cite[Definition 4.1]{gresan::0} took the approach of axiomatising these properties along with something called dimension.  In that vein we define the \textbf{doubling dimension} of a Bohr system $B$ to be
\begin{equation*}
\dim^* B = \sup\left\{\log_2\mathcal{C}\left(B_\eta;B_{\frac{1}{2}\eta}\right): \eta \in (0,1]\right\}.
\end{equation*}
It may be instructive to consider two examples.
\begin{lemma}[Bohr systems of very low doubling dimension]\
\begin{enumerate}
\item Suppose that $B$ is a Bohr system with $\dim^* B<1$.  Then there is a subgroup $H \leq G$ such that $B_\eta=H$ for all $\eta \in (0,1]$.
\item Conversely, suppose that $H \leq G$.  Then the constant vector $B$ with $B_\eta=H$ for all $\eta \in (0,1]$ is a Bohr system and $\dim^*B=0$.
\end{enumerate}
\end{lemma}
\begin{proof}
First, since $\dim^*B<1$ we see that for each $\eta \in (0,1]$ there is a set $X_\eta$ with $|X_\eta|<2^1=2$ such that $B_\eta \subset X_\eta + B_{\frac{1}{2}\eta}$.  Since $B_\eta$ is non-empty we see that $0<|X_\eta| <2$ and so $|X_\eta|=1$.  Write $X_1=\{x_1\}$.  Then
\begin{equation*}
B_1 - B_1 \subset \left(x_1+B_{\frac{1}{2}}\right) - \left(x_1 + B_{\frac{1}{2}}\right) = B_{\frac{1}{2}}-B_{\frac{1}{2}} \subset B_1,
\end{equation*}
and so for all $x,y \in B_1$ we have $x-y \in B_1$ and so there is some subgroup $H \leq G$ such that $B_1=H$.  We show by induction that for each $i \in \N_0$ the set $B_{2^{-i}}$ contains a translate of $H$, from which the result follows since $0_G \in B_{2^{-i}}$.

Turning to the induction: the base case of $i=0$ holds trivially.  Suppose that $B_{2^{-i}}$ contains a translate of $H$.  Then there is some set $X_{2^{-i}}=\{x_{2^{-i}}\}$ such that $B_{2^{-i}} \subset x_{2^{-i}} +B_{2^{-(i+1)}}$, whence $B_{2^{-(i+1)}}$ contains a translate of $H$ as required and the first result is proved.

In the other direction, simply let $\Gamma:=\{\gamma : \gamma(x)=1 \text{ for all }x \in H\}$ and let $\delta$ be the constant function $1/|G|$.  Writing $B$ for the Bohr system generated by $\Gamma$ and $\delta$ we see that $H \subset B_\eta$ for all $\eta \in (0,1]$.  On the other hand if $x \in B_1$ then $|G|\|\gamma(x)\|<1$ and
\begin{align*}
\cos (2\pi |G|\|\gamma(x)\|) & =\frac{1}{2}\left(\exp(2\pi i |G|\|\gamma(x)\|) +\exp(-2\pi i |G|\|\gamma(x)\|)\right)\\ &= \frac{1}{2}\left(\gamma(x)^{|G|}+\overline{\gamma(x)^{|G|}}\right) = 1.
\end{align*}
It follows that $2\pi|G|\|\gamma(x)\| \in 2\pi \Z$ and hence $|G|\|\gamma(x)\| \in \Z$.  We conclude that $\|\gamma(x)\|=0$ and hence $\gamma(x)=1$ for all $x \in B_1$ and $\gamma \in \Gamma$.  It follows that $B_1=H$ and hence $B$ is a constant vector by nesting.  It remains to note that $\mathcal{C}(H;H) =1$ and so $\dim^* B=\log_21=0$ as claimed.
\end{proof}
We say that a Bohr system $B$ has \textbf{rank} $k$ if it can be generated by a pair $(\Gamma,\delta)$ with $|\Gamma| = k$.
\begin{lemma}[Rank $1$ Bohr systems]\label{lem.b1}
Suppose that $B$ is a rank $1$ Bohr system. Then $\dim^* B \leq \log_23$.
\end{lemma}
\begin{proof}
Let $(\Gamma,\delta)$ generate $B$ where $\Gamma=\{\gamma\}$ and write $\delta=\delta(\gamma)$. Suppose that $\eta \in \left(0,1\right]$. We shall show that there is some $x \in G$ such that 
\begin{equation}\label{eqn.inc}
B_{\eta} \subset \{-x,0,x\}+B_{\frac{1}{2}\eta}.
\end{equation}
If $B_{\eta}=B_{\frac{1}{2}\eta}$ then we may take $x=0_G$ and be done; if not let $x \in B_{\eta}\setminus B_{\frac{1}{2}\eta}$ be such that $\|\gamma(x)\|$ is minimal.  Let $\psi \in \left(-\frac{1}{2},\frac{1}{2}\right]$ be such that $\gamma(x)=\exp(2\pi i \psi)$; note that $\|\gamma(x)\|=|\psi|$.

Suppose that $y \in B_{\eta}\setminus B_{\frac{1}{2}\eta}$ and let $\theta \in \left(-\frac{1}{2},\frac{1}{2}\right]$ be such that $\gamma(y)=\exp(2\pi i \theta)$; note that $\| \gamma(y)\|=|\theta|$.  Since $x \not \in B_{\frac{1}{2}\eta}$, $|\psi|$ is minimal, and $y \in B_{\eta}$ we have
\begin{equation*}
\frac{1}{2}\eta\delta \leq \|\gamma(x)\| = |\psi| \leq |\theta|=\|\gamma(y)\| < \eta\delta.
\end{equation*}
Thus if $\psi$ and $\theta$ have the same sign then
\begin{equation*}
|\theta - \psi| = ||\theta| - |\psi|| = |\theta| - |\psi| <\eta\delta  -\frac{1}{2}\eta\delta = \frac{1}{2}\eta\delta,
\end{equation*}
and hence $\|\gamma(y-x)\| < \frac{1}{2}\eta\delta$ (since $\gamma(y-x)=\exp(2\pi i (\theta-\psi))$), so $y \in x+B_{\frac{1}{2}\eta}$.  Similarly if $\psi$ and $\theta$ have opposite signs then $|\theta + \psi| < \frac{1}{2}\eta\delta$ and $\|\gamma(y+x)\| < \frac{1}{2}\eta\delta$, and so $y \in -x+B_{\frac{1}{2}\eta}$.  The claimed inclusion (\ref{eqn.inc}) follows and the result is proved.
\end{proof}
We define the \textbf{width} of a Bohr system $B$ to be
\begin{equation*}
w(B):=\inf\left\{\|\delta\|_{\ell_\infty(\Gamma)}: (\Gamma,\delta) \text{ generates } B\right\}.
\end{equation*}
\begin{lemma}\label{lem.grow}
Suppose that $B$ is a Bohr system and $w(B) < \frac{1}{4}$.  Then
\begin{equation*}
\dim^* B \leq \log_2 \mathcal{C}\left(B_1;B_{\frac{1}{8}}\right) \leq 3\dim^* B
\end{equation*}
\end{lemma}
To prove this we shall use the following trivial observation.
\begin{observation*}
Suppose that $\gamma$ is a character, $x \in G$, and $n \in \N$.  Then
\begin{equation*}
\|\gamma(nx)\|=n\|\gamma(x)\| \text{ provided } \|\gamma(x)\| < \frac{1}{2n}.
\end{equation*}
\end{observation*}
\begin{proof}
Let $\theta,\psi$ be such that $\|\gamma(x)\|=|\theta|$, $\|\gamma(nx)\|=|\psi|$, $\gamma(x)=\exp(2\pi i \theta)$, and $\gamma(nx)=\exp(2\pi i \psi)$.  Since $\gamma$ is a homomorphism, $\gamma(nx)=\gamma(x)^n = \exp(2\pi i \theta n)$, and so $n\theta - \psi \in \Z$.  However, $|n\theta - \psi | < n|\theta| + |\psi | < 1$ (since $|\theta| < \frac{1}{2n}$ and $|\psi|\leq \frac{1}{2}$) and so $\psi = n\theta$ and the result is proved.
\end{proof}
\begin{proof}[Proof of Lemma \ref{lem.grow}]
The right hand inequality is easy from Lemma \ref{lem.bpcn} part (\ref{pt.extn}) and the definition of doubling dimension:
\begin{equation*}
\log_2 \mathcal{C}\left(B_1;B_{\frac{1}{8}}\right) \leq \log_2 \mathcal{C}\left(B_1;B_{\frac{1}{2}}\right)+\log_2 \mathcal{C}\left(B_{\frac{1}{2}};B_{\frac{1}{4}}\right)+\log_2 \mathcal{C}\left(B_{\frac{1}{4}};B_{\frac{1}{8}}\right) \leq 3\dim^* B.
\end{equation*}
In the other direction, since $w(B) < \frac{1}{4}$ there is a pair $(\Gamma,\delta)$ generating $B$ such that $\|\delta\|_{\ell_\infty(\Gamma)}<\frac{1}{4}$.

Suppose that $\eta \in (0,1]$ and let $X \subset B_{\eta}$ be $B_{\frac{1}{2}\eta}$-separated \emph{i.e.} if $x,y \in X$ have $x-y \in B_{\frac{1}{2}\eta}$ then $x=y$.  Let $k \in \N$ be a natural number such that $\frac{1}{2} \leq \eta k \leq 1$ (the reason for which choice will become clear).  Then by nesting of Bohr sets and Lemma \ref{lem.bpcn} part (\ref{pt.bpcn.1}) we have
\begin{equation*}
\mathcal{C}\left(B_{\eta k};B_{\frac{1}{4} \eta k}\right) \leq \mathcal{C}\left(B_{1};B_{\frac{1}{8}}\right)
\end{equation*}
and so there is a set $Z$ such that $B_{\eta k} \subset Z+B_{\frac{1}{4} \eta k}$ and $|Z| \leq \mathcal{C}\left(B_{1};B_{\frac{1}{8}}\right)$.

Since $\eta k \leq 1$ and each $x \in X$ has $x \in B_{\eta}$ we conclude (by sub-additivity) that $kx \in B_{\eta k}$, and hence there is some $z(x) \in Z$ such that $kx \in z(x)+B_{\frac{1}{4}\eta k}$.  Suppose that $z(x)=z(y)$ for $x,y \in X$.  By sub-additivity and nesting we have
\begin{equation*}
x-y \in B_{2\eta} \subset B_{\frac{2}{k}} \text{ and } k(x-y) \in B_{\frac{1}{4} \eta k}-B_{\frac{1}{4} \eta k} \subset B_{\frac{1}{2}\eta k}.
\end{equation*}
Suppose that $\gamma \in \Gamma$.  Then we have just seen that $\|\gamma(x-y)\| < \frac{2}{k}\delta(\gamma) < \frac{1}{2k}$ (since $\delta(\gamma) <\frac{1}{4}$) and so by the Observation we see that
\begin{equation*}
k\|\gamma(x-y)\| = \|\gamma(k(x-y))\|  < \frac{1}{2}\eta k \delta(\gamma).
\end{equation*}
Dividing by $k$ and noting that $\gamma$ was an arbitrary element of $\Gamma$ it follows that $x-y \in B_{\frac{1}{2}\eta}$ and hence $x=y$.  We conclude that the function $z$ is injective and hence $|X| \leq |Z| \leq \mathcal{C}(B_{1};B_{\frac{1}{8}})$.  

Finally, if $X$ is maximal with the given property then for any $y \in B_{\eta}$ either $y \in X$ and so $y \in X+B_{\frac{1}{2}\eta}$ or else there is some $x \in X$ such that $y \in x+B_{\frac{1}{2}\eta}$.  It follows that
\begin{equation*}
B_{\eta} \subset X+B_{\frac{1}{2}\eta},
\end{equation*}
and the left hand inequality is proved given the upper bound on $|X|$.
\end{proof}

We can make new Bohr systems from old by taking intersections: given Bohr systems $B$ and $B'$ we define their \textbf{intersection} to be
\begin{equation*}
B \wedge B':=(B_\eta\cap B'_\eta)_{\eta \in (0,1]}.
\end{equation*}
Writing $\mathcal{B}(G)$ for the set of Bohr systems on $G$ we then have a lattice structure as captured by the following trivial lemma.
\begin{lemma}[Lattice structure]
The pair $(\mathcal{B}(G),\wedge)$ is a meet-semilattice, meaning that is satisfies
\begin{enumerate}
\item \emph{(Closure)} $B \wedge B' \in \mathcal{B}(G)$ for all $B,B' \in \mathcal{B}(G)$;
\item \emph{(Associativity)} $(B \wedge B') \wedge B'' = B \wedge (B' \wedge B'')$ for all $B,B',B'' \in \mathcal{B}(G)$;
\item \emph{(Commutativity)} $B \wedge B'= B' \wedge B$ for all $B,B' \in \mathcal{B}(G)$;
\item \emph{(Idempotence)} $B \wedge B = B$ for all $B \in \mathcal{B}(G)$.
\end{enumerate}
\end{lemma}
\begin{proof}
The only property with any content is the first, the truth of which is dependent on the slightly more general definition of Bohr set we made.  Suppose that $B$ is generated by $(\Gamma,\delta)$ and $B'$ is generated by $(\Gamma',\delta')$.  Then consider the Bohr system $B''$ generated by $(\Gamma \cup \Gamma',\delta\wedge\delta')$ where
\begin{equation*}
\delta\wedge \delta':\Gamma \cup \Gamma' \rightarrow \R_{>0}; \gamma \mapsto \begin{cases} \delta(\gamma) & \text{ if } \gamma \in \Gamma \setminus \Gamma'\\
\delta'(\gamma) & \text{ if } \gamma \in \Gamma' \setminus \Gamma\\
\min\{\delta(\gamma),\delta'(\gamma)\} & \text{ if } \gamma \in \Gamma \cap \Gamma'\end{cases}.
\end{equation*}
It is easy to check that $B''=B \wedge B'$ and hence $B \wedge B' \in \mathcal{B}(G)$.  The remaining properties are inherited pointwise from the meet-semilattice $(\mathcal{P}(G),\cap)$, where $\mathcal{P}(G)$ is the power-set of $G$, that is the set of all subsets of $G$.
\end{proof}

As usual this structure gives rise to a partial order on $\mathcal{B}(G)$ where we write $B' \leq B$ if $B'\wedge B = B'$.

Another way we can produce new Bohr systems is via dilation: given a Bohr system $B$ and a parameter $\lambda \in (0,1]$, we write $\lambda B$ for the \textbf{$\lambda$-dilate} of $B$, and define it to be the vector
\begin{equation*}
\lambda B=(B_{\eta\lambda})_{\eta \in (0,1]}.
\end{equation*}
We then have the following trivial properties.
\begin{lemma}[Basic properties of dilation]\ 
\begin{enumerate}
\item \emph{(Order-preserving action)} The map
\begin{equation*}
(0,1] \times \mathcal{B}(G) \rightarrow \mathcal{B}(G); (\lambda,B) \mapsto \lambda B
\end{equation*}
is a well-defined order-preserving action of the monoid $((0,1],\times)$ on the set of Bohr systems.
\item \emph{(Distribution over meet)} We have 
\begin{equation*}
\lambda (B \wedge B') = (\lambda B) \wedge (\lambda B') \text{ for all }B,B' \in \mathcal{B}(G), \lambda \in (0,1].
\end{equation*}
\end{enumerate}
\end{lemma}
The doubling dimension interacts fairly well with intersection and dilation and it can be shown that
\begin{equation*}
\dim^* \lambda B \leq \dim^* B \text{ and }\dim^* B \wedge B' =O(\dim^* B + \dim^*B')
\end{equation*}
for Bohr systems $B,B'$ and $\lambda \in (0,1]$.  (The first of these is trivial; the second requires a little more work.)  

The big-$O$ here is inconvenient in applications and to deal with this we define a variant which is equivalent, but which behaves a little better under intersection.  The \textbf{dimension} of a Bohr system $B$ is defined to be
\begin{equation*}
\dim B = \sup\left\{\log_2\mathcal{C}^\Delta\left(B_\eta;B_{\frac{1}{2}\eta}\right): \eta \in (0,1]\right\}.
\end{equation*}
\begin{lemma}[Basic properties of dimension]\label{lem.smi}\
\begin{enumerate}
\item \label{lem.smi.1} \emph{(Sub-additivity of dimension w.r.t. intersection)} For all $B,B' \in \mathcal{B}(G)$ we have
\begin{equation*}
\dim B \wedge B' \leq \dim B + \dim B'.
\end{equation*}
\item \label{lem.smi.2} \emph{(Monotonicity of dimension w.r.t. dilation)} For all $B \in \mathcal{B}(G)$ and $\lambda \in (0,1]$ we have
\begin{equation*}
\dim \lambda B \leq \dim B.
\end{equation*}
\item \label{lem.smi.3} \emph{(Equivalence of dimension and doubling dimension)} For all $B \in \mathcal{B}(G)$ we have
\begin{equation*}
\dim^* B \leq \dim B \leq 2\dim^* B.
\end{equation*}
\end{enumerate}
\end{lemma}
\begin{proof}
First, from Lemma \ref{lem.dfc}, part (\ref{pt2.difc}) we have
\begin{align*}
\mathcal{C}^\Delta\left((B\wedge B')_\eta;(B\wedge B')_{\frac{1}{2}\eta}\right)& =\mathcal{C}^\Delta\left(B_\eta \cap B'_\eta;B_{\frac{1}{2}\eta}\cap B'_{\frac{1}{2}\eta}\right)\\ & \leq \mathcal{C}^\Delta\left(B_\eta;B_{\frac{1}{2}\eta}\right)\mathcal{C}^\Delta\left(B_\eta';B_{\frac{1}{2}\eta}'\right)
\end{align*}
for all $\eta \in (0,1]$.  Taking $\log$s the sub-additivity of dimension follows since suprema are sub-linear.

Secondly, monotonicity follows immediately since
\begin{align*}
\dim \lambda B & = \sup\left\{\log_2\mathcal{C}^\Delta\left((\lambda B)_\eta;(\lambda B)_{\frac{1}{2}\eta}\right): \eta \in (0,1]\right\}\\
& = \sup\left\{\log_2\mathcal{C}^\Delta\left(B_{\eta};B_{\frac{1}{2}\eta}\right): \eta \in (0,\lambda]\right\} \leq \dim B.
\end{align*}

Finally, it follows from Lemma \ref{lem.dfc} part (\ref{pt4.difc}) that $\dim^* B \leq \dim B$.  On the other hand from the sub-additivity and symmetry of Bohr sets we have $B_{\frac{1}{2}\eta} \supset B_{\frac{1}{4}\eta}-B_{\frac{1}{4}\eta}$, and so by Lemma \ref{lem.dfc} parts (\ref{p4.dfc}) and (\ref{pt3.difc}) we get
\begin{equation*}
\mathcal{C}^\Delta\left(B_\eta;B_{\frac{1}{2}\eta}\right) \leq \mathcal{C}^\Delta\left(B_\eta;B_{\frac{1}{4}\eta}-B_{\frac{1}{4}\eta}\right)\leq \mathcal{C}\left(B_\eta;B_{\frac{1}{4}\eta}\right).
\end{equation*}
Hence by Lemma \ref{lem.bpcn} part (\ref{pt.extn}) and the definition of doubling dimension we have
\begin{equation*}
\mathcal{C}\left(B_\eta;B_{\frac{1}{4}\eta}\right) \leq \mathcal{C}\left(B_\eta;B_{\frac{1}{2}\eta}\right)\mathcal{C}\left(B_{\frac{1}{2}\eta};B_{\frac{1}{4}\eta}\right) \leq 2^{2\dim^* B},
\end{equation*}
and so $\dim B \leq 2\dim^* B$ as claimed.
\end{proof}

As well as the various notion of dimension, Bohr systems also have a notion of size relative to some `reference' set.  Very roughly we think of the `size' of a Bohr system $B$ relative to some reference set $A$ as being $\mathcal{C}^\Delta(A;B_1)$. This quantity is then governed by the following lemma.
\begin{lemma}[Size of Bohr systems]\label{lem.bss}
Suppose that $B$ is a Bohr system and $A \subset G$.  Then the following hold.
\begin{enumerate}
\item \label{lem.bss.1} \emph{(Size of dilates)}  For all $\lambda \in (0,1]$ we have
\begin{equation*}
\mathcal{C}^\Delta(A;(\lambda B)_1) \leq  \mathcal{C}^\Delta\left(A;B_{1}\right)(4\lambda^{-1})^{\dim B}.
\end{equation*}
\item \label{lem.bss.2} \emph{(Size and non-triviality)}  If $\mathcal{C}^\Delta(A;B_1)<|A|$ then there is some $x \in B_1$ with $x \neq 0_G$.
\end{enumerate}
\end{lemma}
\begin{proof}
By symmetry and sub-additivity of Bohr sets we see that $(\lambda B)_1 \supset B_{\frac{1}{2}\lambda} - B_{\frac{1}{2}\lambda}$ and so by Lemma \ref{lem.dfc} parts (\ref{p4.dfc}) and (\ref{pt3.difc}) we have
\begin{equation*}
\mathcal{C}^\Delta(A;(\lambda B)_1) \leq \mathcal{C}^\Delta\left(A;B_{\frac{1}{2}\lambda} - B_{\frac{1}{2}\lambda}\right) \leq \mathcal{C}\left(A;B_{\frac{1}{2}\lambda}\right).
\end{equation*}
Write $r$ for the largest natural number such that $2^r\lambda \leq 1$.  By Lemma \ref{lem.bpcn} part (\ref{pt.extn}) we see that
\begin{align*}
 \mathcal{C}\left(A;B_{\frac{1}{2}\lambda}\right) & \leq \mathcal{C}\left(A;B_{2^r\lambda}\right)\prod_{i=0}^{r}{ \mathcal{C}\left(B_{2^{i}\lambda};B_{2^{i-1}\lambda}\right) }\\
 & \leq \mathcal{C}\left(A;B_{1}\right) \mathcal{C}\left(B_1;B_{2^r\lambda}\right)\prod_{i=0}^{r}{ \mathcal{C}\left(B_{2^{i}\lambda};B_{2^{i-1}\lambda}\right) }\\
 & \leq \mathcal{C}\left(A;B_{1}\right) 2^{(r+2)\dim^* B}\leq \mathcal{C}^\Delta\left(A;B_{1}\right) 2^{(r+2)\dim B},
\end{align*}
where the last inequality is by Lemma \ref{lem.dfc} part (\ref{pt4.difc}) and the first inequality in Lemma \ref{lem.smi} part (\ref{lem.smi.3}).  The first part follows.

By Lemma \ref{lem.dfc} part (\ref{pt4.difc}) we then see that $\mathcal{C}(A;B_1) \leq \mathcal{C}^\Delta(A;B_1)<|A|$.  It follows that there is some set $X$ with $|X|<|A|$ such that $A \subset X+B_1$ whence $|A| \leq |X||B_1| < |A||B_1|$ which implies that $|B_1|>1$ and hence contains a non-trivial element establishing the second part.
\end{proof}

\section{Measures, convolution and approximate invariance}\label{sec.ain}

Given a finite set $X$ we write $C(X)$ for the complex-valued functions on $X$.  (We think of $X$ as a discrete topological space and these functions as continuous with an eye to \S\ref{sec.con}, hence the notation.) Further, given a probability measure $\mu$ on $X$ and a set $S$ with $\mu(S)>0$, we write $\mu_S$ for the probability measure induced by
\begin{equation*}
C(X) \rightarrow \C; f\mapsto \frac{1}{\mu(S)}\int{f1_Sd\mu}.
\end{equation*}
Moreover, if $S$ is a non-empty subset of $G$ then we write $m_S$ for $(m_G)_S$.  (Note that this notation is consistent since $m_G=(m_G)_G$.)

Below we shall define various notation for functions and for measures.  Since $G$ is finite we can associate to any measure $\mu$ on $G$ a function $y \mapsto \mu(\{y\})$.  The notational choices we make are designed to be compatible between these two different ways of thinking about measures hence the slightly unusual choice in (\ref{eqn.laterone}).

Given $f \in C(G)$ and an element $x \in G$ we define
\begin{equation*}
\tau_x(f)(y):=f(y-x) \text{ for all }y \in G.
\end{equation*}
We write $M(G)$ for the space of complex-valued measures on $G$ and to each $\mu \in M(G)$ associate the linear functional
\begin{equation}\label{eqn.laterone}
C(G) \rightarrow \C; f\mapsto \langle f,\mu\rangle :=\overline{\int{\overline{f(x)}d\mu(x)}}.
\end{equation}
The functionals defined above are all linear functionals by the Riesz Representation Theorem \cite[E4]{rud::1}, though of course it is rather simple in our setting of finite $G$.

Given $\mu \in M(G)$ we define $\tau_x(\mu)$ to be the measure induced by,
\begin{equation*}
C(G) \rightarrow \C; f\mapsto \int{\tau_{-x}(f)d\mu}.
\end{equation*}
and $\tilde{\mu}$ to be the measure induced by
\begin{equation*}
C(G) \rightarrow \C; f\mapsto \overline{\int{\overline{f(-x)}d\mu(x)}}.
\end{equation*}
Given $f \in L_\infty(G)$ and $\mu \in M(G)$ we define
\begin{equation*}
\mu \ast f(x)= f \ast \mu(x)=\int{f(y)d\mu(x-y)},
\end{equation*}
and for a further measure $\nu \in M(G)$ we define the \textbf{convolution} of $\mu$ and $\nu$, denoted $\mu \ast \nu$, to be the measure induced by
\begin{equation*}
C(G) \rightarrow \C; f \mapsto \int{f(x+y)d\mu(x)d\nu(y)}
\end{equation*}
This operation makes $M(G)$ into a commutative Banach algebra with unit; for details see \cite[\S1.3.1]{rud::1}.

This notation all extends in the expected way to functions so that if $f \in L_1(m_G)$ then $\tilde{f}$ is defined point-wise by
\begin{equation*}
\wt{f}(x):=\overline{f(-x)} \text{ for all }x \in G,
\end{equation*}
and given a further $g \in L_1(m_G)$ we define the \textbf{convolution} of $f$ and $g$ to be $f\ast g$ which is determined point-wise by
\begin{equation*}
f \ast g(x)= \int{f(y)g(x-y)dm_G(y)} \text{ for all }x \in G.
\end{equation*}
This can be written slightly differently using the inner product on $L_2(m_G)$.  If $g,f \in L_2(m_G)$ then
\begin{equation*}
\langle f,g\rangle_{L_2(m_G)} = \int{f(x)\overline{g(x)}dm_G(x)},
\end{equation*}
and
\begin{equation*}
f \ast g(x)=\langle f,\tau_x(\tilde{g})\rangle_{L_2(m_G)} \text{ for all }x \in G.
\end{equation*}

Given a Bohr system $B$ we say that a probability measure $\mu$ on $G$ is \textbf{$B$-approximately invariant} if for every $\eta \in (0,1]$ there are probability measures $\mu_\eta^+$ and $\mu_\eta^-$ such that
\begin{equation*}
(1-\eta)\mu_{\eta}^- \leq \tau_x(\mu) \leq (1+\eta)\mu_{\eta}^+  \text{ for all }x \in B_{\eta}.
\end{equation*}
It may be worth remembering at that for two measures $\nu$ and $\kappa$ we say $\nu \geq \kappa$ if and only if $\nu - \kappa$ is non-negative.

To motivate the name in this definition we have the following lemma where we recall that $\|\mu\|:=\int{d|\mu|}$.
\begin{lemma}\label{lem.inv}
Suppose that $B$ is a Bohr system and $\mu$ is $B$-approximately invariant.  Then for all $\eta \in (0,1]$ we have
\begin{equation*}
\|\mu - \tau_x(\mu)\| \leq \eta \text{ for all }x \in B_{\frac{1}{2}\eta}.
\end{equation*}
\end{lemma}
\begin{proof}
Suppose that $x \in B_{\frac{1}{2}\eta}$.  Then
\begin{equation*}
\left(1-{\frac{1}{2}\eta}\right)\mu_{\frac{1}{2}\eta}^- \leq \tau_x(\mu) \leq \left(1+\frac{\eta}{2}\right)\mu_{\frac{1}{2}\eta}^+ \text{ and } \left(1-{\frac{1}{2}\eta}\right)\mu_{\frac{1}{2}\eta}^- \leq \mu \leq \left(1+\frac{\eta}{2}\right)\mu_{\frac{1}{2}\eta}^+.
\end{equation*}
It follows that
\begin{equation*}
\tau_x(\mu) - \mu \leq \left(1+{\frac{1}{2}\eta}\right)\mu^+_{\frac{1}{2}\eta} - \left(1-{\frac{1}{2}\eta}\right)\mu^-_{\frac{1}{2}\eta},
\end{equation*}
and
\begin{equation*}
\tau_x(\mu) - \mu \geq \left(1-{\frac{1}{2}\eta}\right)\mu^-_{\frac{1}{2}\eta} - \left(1+{\frac{1}{2}\eta}\right)\mu^+_{\frac{1}{2}\eta}.
\end{equation*}
The Jordan decomposition theorem tells us that there are two measurable sets $P$ and $N$ (which together form a partition of $G$) such that $\tau_x(\mu) - \mu$ is a non-negative measure on $P$ and a non-positive measure on $N$.  We conclude that
\begin{align*}
\|\tau_x(\mu) - \mu\| & = (\tau_x(\mu) - \mu)(P) - (\tau_x(\mu) - \mu)(N)\\
& \leq \left(1+{\frac{1}{2}\eta}\right)\mu^+_{\frac{1}{2}\eta}(P) -  \left(1-{\frac{1}{2}\eta}\right)\mu^-_{\frac{1}{2}\eta}(P)\\
& \qquad + \left(1+{\frac{1}{2}\eta}\right)\mu^+_{\frac{1}{2}\eta}(N) -\left(1-{\frac{1}{2}\eta}\right)\mu^-_{\frac{1}{2}\eta}(N)\\
& = \left(1+{\frac{1}{2}\eta}\right) - \left(1-{\frac{1}{2}\eta}\right) = \eta,
\end{align*}
since $\mu^+_\eta$ and $\mu^-_\eta$ are probability measures and $N\sqcup P=G$.  The result is proved.
\end{proof}
This can be slightly generalised in the following convenient way.
\begin{lemma}\label{lem.cc}
Suppose that $B$ is a Bohr system and $\mu$ is a $B$-approximately invariant probability measure.  Then
\begin{equation*}
\|\tau_x(f \ast \mu) - f \ast \mu\|_{L_\infty(G)} \leq \eta \|f\|_{L_\infty(G)} \text{ for all }x \in B_{\frac{1}{2}\eta}.
\end{equation*}
\end{lemma}
\begin{proof}
Simply note that
\begin{equation*}
|f\ast \mu(y-x) - f \ast \mu(y)| \leq \int{|f(z)|d|\tau_x(\mu) - \mu|(z)} \leq \eta \|f\|_{L_\infty(G)}
\end{equation*}
by the triangle inequality and Lemma \ref{lem.inv}.
\end{proof}
Approximately invariant probability measures are closed under convolution with probability measures.
\begin{lemma}\label{lem.closbm}
Suppose that $B$ is a Bohr system, $\mu$ is a $B$-approximately invariant probability measure, and $\nu$ is a probability measure.  Then $\mu \ast \nu$ is a $B$-approximately invariant probability measure.
\end{lemma}
\begin{proof}
Since $\mu$ is $B$-approximately invariant there are probability measures $(\mu_\eta^-)_{\eta \in (0,1]}$ and $(\mu_\eta^+)_{\eta \in (0,1]}$ such that
\begin{equation*}
(1-\eta)\mu_{\eta}^- \leq \tau_x(\mu) \leq (1+\eta)\mu_\eta^+ \text{ for all }x \in B_\eta.
\end{equation*}
Since $\nu$ is a probability measure we can integrate the above inequalities to get
\begin{equation*}
(1-\eta)\mu_{\eta}^-\ast \nu \leq \tau_x(\mu)\ast \nu \leq (1+\eta)\mu_\eta^+\ast \nu \text{ for all }x \in B_\eta.
\end{equation*}
But then since $\tau_x(\mu) \ast \nu = \tau_x(\mu \ast \nu)$ we can put $(\mu \ast \nu)_\eta^-:=\mu_{\eta}^-\ast \nu$ and $(\mu \ast \nu)_\eta^+:=\mu_{\eta}^+\ast \nu$ to get the required family of measures for $\mu \ast \nu$.
\end{proof}

The last result of this section is essentially \cite[Lemma 3.0]{bou::5} and ensures a plentiful supply of approximately invariant probability measures.
\begin{proposition}\label{prop.ubreg}
Suppose that $B$ is a Bohr system, and $X$ is a non-empty set with $m_G(X+B_1) \leq Km_G(X)$. Then there is a $\lambda B$-approximately invariant probability measure with support contained in $X+B_1$ for some $1 \geq \lambda \geq \frac{1}{24\log 2K}$.
\end{proposition}
\begin{proof}
Let $C:=24$ and $\lambda:=1/C\log 2K$.  Note that $K \geq 1$ and so $\lambda<1/4$.  Suppose that for all $\kappa \in \left[\frac{1}{4},\frac{3}{4}\right]$ there is some $\delta_\kappa \in (0,\lambda]$ such that
\begin{equation*}
\frac{m_G(X+B_{\kappa+\delta_\kappa})}{m_G(X+B_{\kappa-\delta_\kappa})} > \exp\left(\frac{1}{2}\lambda^{-1}\delta_\kappa\right).
\end{equation*}
Write $I_\kappa:=[\kappa-\delta_\kappa,\kappa+\delta_\kappa]$, and note that $\bigcup_{\kappa}{I_\kappa} \supset \left[\frac{1}{4},\frac{3}{4}\right]$.  By the Vitali covering lemma\footnote{One can also proceed directly here.} we conclude that there is a sequence $\kappa_1<\dots<\kappa_m$ such that the intervals $(I_{\kappa_i})_{i=1}^m$ are disjoint and
\begin{equation*}
\sum_{i=1}^m{2\delta_{\kappa_i}}=\sum_{i=1}^m{\mu(I_{\delta_{\kappa_i}})}\geq \frac{1}{3}\mu\left(\left[\frac{1}{4},\frac{3}{4}\right]\right) = \frac{1}{6}.
\end{equation*}
Since the intervals $(I_{\kappa_i})_{i=1}^m$ are disjoint, $(\kappa_i)_{i=1}^m$ is an increasing sequence, and $\delta_{\kappa_1},\delta_{\kappa_m}\leq \lambda <\frac{1}{4}$ we see that
\begin{equation*}
0<\kappa_1-\delta_{\kappa_1} <\kappa_1+\delta_{\kappa_1}< \dots < \kappa_i + \delta_{\kappa_i} < \kappa_{i+1}-\delta_{\kappa_{i+1}}< \dots < \kappa_m+\delta_{\kappa_m}<1,
\end{equation*}
and hence
\begin{align*}
K=\exp(1/24\lambda) \leq \exp\left(\frac{1}{4}\lambda^{-1}\sum_{i=1}^m{2\delta_{\kappa_i}}\right)
& = \prod_{i=1}^m{\exp\left(\frac{1}{2}\lambda^{-1}\delta_{\kappa_i}\right)}\\
& < \prod_{i=1}^m{\frac{m_G(X+B_{\kappa_i+\delta_{\kappa_i}})}{m_G(X+B_{\kappa_i-\delta_{\kappa_i}})}}\\
& = \frac{m_G(X+B_{\kappa_m+\delta_{\kappa_m}})}{m_G(X+B_{\kappa_1-\delta_{\kappa_1}})} \cdot \prod_{i=1}^{m-1}{\frac{m_G(X+B_{\kappa_i+\delta_{\kappa_i}})}{m_G(X+B_{\kappa_{i+1}-\delta_{\kappa_{i+1}})}}}\\
& \leq \frac{m_G(X+B_{1})}{m_G(X)} \leq K. 
\end{align*}
This is a contradiction and so there is some $\kappa \in \left[\frac{1}{4},\frac{3}{4}\right]$ such that
\begin{equation*}
\frac{m_G(X+B_{\kappa+\delta})}{m_G(X+B_{\kappa-\delta})} \leq \exp\left(\frac{1}{2}\lambda^{-1} \delta\right) \text{ for all } \delta \in \left(0,\lambda\right].
\end{equation*}
Let $\mu$ be the uniform probability measure on $X+B_{\kappa}$, and for each $\eta \in (0,1]$ let $\mu_{\eta}^-$ be the uniform probability measure on $X+B_{\kappa-\lambda\eta}$ and $\mu_{\eta}^+$ be the uniform probability measure on $X+B_{\kappa+\lambda \eta}$.  If $x \in (\lambda B)_\eta$ then $x \in B_{\lambda \eta}$ and so 
\begin{equation*}
\tau_x(\mu) \leq \frac{m_G(X+B_{\kappa+\lambda\eta})}{m_G(X+B_\kappa)}\mu_{\eta}^+ \leq \exp\left(\frac{1}{2}\lambda^{-1} \lambda\eta\right) \mu_{\eta}^+ \leq (1+\eta)\mu_{\eta}^+,
\end{equation*}
since $1+x \geq \exp(x/2)$ whenever $0 \leq x\leq 1$. Similarly
\begin{equation*}
\tau_x(\mu) \geq \frac{m_G(X+B_{\kappa-\lambda\eta})}{m_G(X+B_\kappa)}\mu_{\eta}^- \geq \exp\left(-\frac{1}{2}\lambda^{-1} \lambda\eta\right) \mu_{\eta}^- \geq (1-\eta)\mu_{\eta}^-,
\end{equation*}
since $1-x \leq \exp(-x/2)$ whenever $0 \leq x\leq 1$.  The result is proved.
\end{proof}
For applications it will often be useful to have the following corollary.
\begin{corollary}\label{cor.ubreg}
Suppose that $B$ is a Bohr system with $\dim B \leq d$ for some parameter $d \geq 1$.  Then there is some $\lambda \in (\Omega(d^{-1}),1]$ and a $\lambda B$-approximately invariant probability measure $\mu$ supported on $B_1$.
\end{corollary}
\begin{proof}
Put $X:=B_{\frac{1}{2}}$ and $B':=\frac{1}{2}B$. By Lemma \ref{lem.covsum} we know that
\begin{equation*}
m_G(X+B_1') = m_G(B_{\frac{1}{2}} + B_{\frac{1}{2}}) \leq \mathcal{C}\left(B_{\frac{1}{2}};B_{\frac{1}{4}}\right)^2m_G(B_{\frac{1}{4}} + B_{\frac{1}{4}}).
\end{equation*}
However, by sub-additivity of Bohr sets $B_{\frac{1}{4}} + B_{\frac{1}{4}} \subset B_{\frac{1}{2}}=X$.  Thus given the definition of doubling dimension and the first inequality in Lemma \ref{lem.smi} part (\ref{lem.smi.3}) we see that
\begin{equation*}
\frac{m_G(X+B_1') }{m_G(X) } \leq \mathcal{C}\left(B_{\frac{1}{2}};B_{\frac{1}{4}}\right)^2 \leq 2^{2\dim^* B} \leq 2^{2d}.
\end{equation*}
By Proposition \ref{prop.ubreg} applied to $X$ and $B'$ there is a $\lambda B'$-approximately invariant probability measure $\mu$ with support in $X+B_1' =B_{\frac{1}{2}}+B_{\frac{1}{2}} \subset B_1$.  The result follows since $\lambda \geq 1/24\log 2^{2d+1}$ and $\lambda B' = \frac{\lambda}{2}B$.
\end{proof}

\section{Approximate annihilators}\label{sec.aa}

We shall understand the dual group of $G$ through what we call `approximate annihilators', though this nomenclature is non-standard.

Given a set $S \subset G$ and a parameter $\rho>0$ we define the \textbf{$\rho$-approximate annihilator} of $S$ to be the set
\begin{equation*}
N(S,\rho):=\{\gamma \in \wh{G}: |1-\gamma(x)|<\rho \text{ for all }x \in S\}.
\end{equation*}

Approximate annihilators enjoy many of the same properties as Bohr sets as we record in the following trivial lemma (an analogue of Lemma \ref{lem.pbs}).
\begin{lemma}[Properties of approximate annihilators]\
Suppose that $S$ is a set.  Then
\begin{enumerate}
\item \emph{(Identity)} $0_{\wh{G}} \in N(S,\rho)$ for all $\rho>0$;
\item \emph{(Symmetry)} $N(S,\rho) = -N(S,\rho)$ for all $\rho>0$;
\item \emph{(Nesting)} $N(S,\rho)\subset N(S,\rho') $ whenever $0<\rho\leq \rho'$;
\item \emph{(Sub-additivity)} $N(S,\rho) + N(S,\rho') \subset N(S,\rho+\rho')$ for all $\rho,\rho'>0$.
\end{enumerate}
\end{lemma}
Approximate annihilators and approximately invariant measures interact rather well as is captured by the following version of \cite[Lemma 3.6]{grekon::}.  To state it we require the Fourier transform extended to measures: for $\mu \in M(G)$ we define
\begin{equation*}
\wh{\mu}:\wh{G} \rightarrow \C; \gamma \mapsto \int{\overline{\gamma(x)}d\mu(x)}.
\end{equation*}
\begin{lemma}[Majorising annihilators]\label{lem.nest}
Suppose that $B$ is a Bohr system with $\mu$ a $B$-approximately invariant probability measure, and $\kappa,\eta \in (0,1]$ are parameters.  Then
\begin{equation*}
\{\gamma \in \wh{G}:|\wh{\mu}(\gamma)| \geq \kappa\} \subset N(B_{\frac{1}{2}\kappa\eta},\eta).
\end{equation*}
\end{lemma}
\begin{proof}
Suppose that $|\wh{\mu}(\gamma)| \geq \kappa$ and $y \in B_{\frac{1}{2}\kappa\eta}$.  Then $-y \in B_{\frac{1}{2}\kappa\eta}$ by symmetry and so by Lemma \ref{lem.inv} we have
\begin{equation*}
|1-\gamma(y)|\kappa < \left|\int{\gamma(x)d\mu(x)} - \int{\gamma(x+y)d\mu(x)}\right|\leq \|\mu - \tau_{-y}(\mu)\| \leq \eta \kappa.
\end{equation*}
The result follows on dividing by $\kappa$.
\end{proof}

In the more general topological setting where $G$ is not assumed finite, approximate annihilators form a base for the topology of the dual group \cite[Theorem 1.2.6]{rud::1}.  \cite[Theorem 1.2.6]{rud::1} also captures the natural duality between our approximate annihilators and sets of the form
\begin{equation}\label{eqn.abg}
\{x \in G: |1-\gamma(x)|<\rho \text{ for all }\gamma \in \Gamma\} \text{ for }\Gamma \subset \wh{G}.
\end{equation}
A number of elements of this paper would be neater if our Bohr sets were replaced by (a suitable generalisation of) sets of the form given in (\ref{eqn.abg}).  The only benefit we know of arising from our choice is that the proof of Lemma \ref{lem.grow} is slightly easier for vectors of Bohr sets.

For us the duality in \cite[Theorem 1.2.6]{rud::1} is captured in the following lemma.
\begin{lemma}[Duality of Bohr sets and approximate annihilators]\label{lem.bsann}\
\begin{enumerate}
\item If $X$ is a non-empty subset of $G$ and $\epsilon \in (0,1]$ then
\begin{equation*}
X \subset \Bohr\left(N(X,\epsilon),\delta\right) \text{ where } \delta:=\frac{\epsilon}{4}\cdot 1_{N(X,\epsilon)};
\end{equation*}
\item if $\Gamma$ is a non-empty set of characters of $G$ and $\delta:\Gamma \rightarrow \R_{>0}$ then
\begin{equation*}
\Gamma \subset N\left(\Bohr\left(\Gamma,\delta\right),\epsilon \right) \text{ where } \epsilon=2\pi \|\delta\|_{\ell_\infty(\Gamma)}.
\end{equation*}
\end{enumerate}
\end{lemma}
\begin{proof}
First note that
\begin{equation*}
1-\frac{\theta^2}{2} \leq \cos \theta \leq 1-\frac{2\theta^2}{\pi^2} \text{ whenever }|\theta| \leq \pi.
\end{equation*}
On the other hand $\|z\| \leq \frac{1}{2}$ for all $z \in S^1$ and
\begin{equation*}
\sqrt{2-2\cos 2\pi \|z\|} = |z-1|.
\end{equation*}
It follows that
\begin{equation*}
4\|\gamma(x)\| \leq |\gamma(x) -1| \leq 2\pi \|\gamma(x)\| \text{ for all }x \in G, \gamma \in \wh{G}.
\end{equation*}
The result is proved once we disentangle the meaning of the two claims.
\end{proof}
The following is \cite[Proposition 4.39]{taovu::} extended to two sets.  The proof is the same.
\begin{lemma}\label{lem.tvspec}
Suppose that $S,T$ are non-empty sets such that $m_G(S+T) \leq Km_G(S)$ and $\epsilon \in (0,1]$ is a parameter.  Then
\begin{equation*}
\{ \gamma \in \wh{G}: |\wh{1_{S+T}}(\gamma)| > (1-\epsilon)m_G(S+T)\} \subset N(T-T,2\sqrt{2\epsilon K}).
\end{equation*}
\end{lemma}
\begin{proof}
For each $\gamma \in \wh{G}$ let $\omega_\gamma \in S^1$ be such that $\overline{\omega_\gamma} \wh{1_{S+T}}(\gamma) = |\wh{1_{S+T}}(\gamma)|$.  For all $t,t' \in T$ we then have
\begin{align*}
|\gamma(t)-\gamma(t')|^2m_G(S) & = \int_{S}{|\gamma(t+s)-\gamma(t'+s)|^2dm_G(s)}\\ & \leq  2\left(\int_{S}{|\gamma(t+s)-\omega_\gamma|^2dm_G(s)}\right.\\ & \qquad \qquad \left. +\int_{S}{|\gamma(t'+s)-\omega_\gamma|^2dm_G(s)}\right)\\ & \leq 4\int_{S+T}{|\gamma(x)-\omega_\gamma|^2dm_G(x)} = 8(m_G(S+T)-|\wh{1_{S+T}}(\gamma)|).
\end{align*}
It follows that if $|\wh{1_{S+T}}(\gamma)| > (1-\epsilon)m_G(S+T)$ then
\begin{equation*}
|\gamma(t-t')-1| = |\gamma(t) - \gamma(t')| < 2\sqrt{2\epsilon K},
\end{equation*}
and the result is proved.
\end{proof}

\section{Fourier analysis}\label{sec.fa}

In this section we turn our attention to the Fourier transform itself.  First we have the Fourier inversion formula \cite[Theorem 1.5.1]{rud::1}: if $f \in A(G)$ then
\begin{equation*}
f(x)=\sum_{\gamma \in \wh{G}}{\wh{f}(\gamma)\gamma(x)} \text{ for all }x \in G.
\end{equation*} 
Since $G$ is finite this is a purely algebraic statement which can be easily checked.  It can be used to prove Parseval's theorem \cite[Theorem 1.6.2]{rud::1} that if $f,g \in L_2(m_G)$ then
\begin{equation}\label{eqn.pars}
\langle f,g\rangle_{L_2(m_G)} = \langle \wh{f},\wh{g}\rangle_{\ell_2(\wh{G})} = \sum_{\gamma \in \wh{G}}{\wh{f}(\gamma)\overline{\wh{g}(\gamma)}}.
\end{equation}

One of the key uses of Bohr sets is as approximate invariant sets for functions.
\begin{lemma}\label{lem.bohrinv}
Suppose that $\Gamma$ is a set of $k$ characters.  Then there is a Bohr system $B$ with $\mathcal{C}^\Delta(G;B_1)\leq 1$ and $\dim B=O(k)$, such that for every $f \in A(G)$ with $\supp \wh{f} \subset \Gamma$ we have
\begin{equation*}
\|\tau_x(f) - f\|_{L_\infty(G)} \leq \epsilon\|f\|_{A(G)} \text{ whenever } x\in B_{\frac{1}{\pi}\epsilon}.
\end{equation*}
\end{lemma}
\begin{proof}
For each $\gamma \in \Gamma$ let $B^{(\gamma)}$ be the Bohr system with frequency set $\{\gamma\}$ and width function the constant function $\frac{1}{2}$ and put $B:=\bigwedge_{\gamma \in \Gamma}{B^{(\gamma)}}$.  (Equivalently, let $B$ be the Bohr system with frequency set $\Gamma$ and width function the constant function $\frac{1}{2}$.)

Since $\|z\|\leq \frac{1}{2}$ for all $z \in S^1$ we see that $B_1 = G$.  It follows from Lemma \ref{lem.dfc} part (\ref{pt3.difc}) that $\mathcal{C}^\Delta(G;B_1) = \mathcal{C}^\Delta(G;G-G) \leq \mathcal{C}(G;G)$.  On the other hand $G \subset \{0_G\}+G$ and so $\mathcal{C}(G;G)\leq 1$ as claimed.

By Lemma \ref{lem.b1} (and the second inequality in Lemma \ref{lem.smi} part (\ref{lem.smi.3})) we have $\dim B^{(\gamma)}=O(1)$ and by Lemma \ref{lem.smi} part (\ref{lem.smi.1}) we conclude that $\dim B=O(k)$.

Now, suppose that $f$ is of the given form, meaning $\supp \wh{f} \subset \Gamma$ and $f \in A(G)$.  Then by Fourier inversion we have
\begin{equation*}
|\tau_x(f)(y)-f(y)| = \left|\sum_{\gamma \in \Gamma}{\wh{f}(\gamma)(\gamma(x+y)-\gamma(y))}\right| \leq \|f\|_{A(G)} \sup\{|\gamma(x)-1|: \gamma \in \Gamma\}.
\end{equation*}
On the other hand the second part of Lemma \ref{lem.bsann} tells us that this supremum is at most $\epsilon$ when $x \in B_{\frac{1}{\pi}\epsilon}$ and the result is proved.
\end{proof}

The next result is a variant of \cite[Lemma 3.2]{croabasis::} proved using their beautiful method.
\begin{lemma}\label{lem.chang}
Suppose that $B$ is a Bohr system, $\mu$ is $B$-approximately invariant, $g \in A(G)$, and $p \in [1,\infty)$ and $\epsilon \in (0,1]$ are parameters.  Then there is a Bohr system $B' \leq B$ such that for any $A \subset G$ we have
\begin{equation*}
\mathcal{C}^\Delta(A;B_1') \leq (2\epsilon^{-1})^{O(p\epsilon^{-2})}\mathcal{C}^\Delta(A;B_1) \text{ and }\dim B' \leq \dim B + O(p\epsilon^{-2}),
\end{equation*}
and
\begin{equation*}
\left\|\tau_x(g)-g\right\|_{L_p(\mu)} \leq \epsilon \|g\|_{A(G)} \text{ for all } x\in B_{1}'.
\end{equation*}
\end{lemma}
\begin{proof}
We may certainly suppose that $g \not \equiv 0$ so that $\|g\|_{A(G)}>0$ (or else simply take $B':=B$ and we are trivially done).  Consider independent identically distributed random variables $X_1,\dots,X_l$ taking values in $L_\infty(G)$ with
\begin{equation*}
\P\left(X_i=\frac{\wh{g}(\gamma)}{|\wh{g}(\gamma)|} \gamma \right)=\frac{1}{\|g\|_{A(G)}}|\wh{g}(\gamma)| \text{ for all } \gamma \in \wh{G} \text{ such that } |\wh{g}(\gamma)| \neq 0.
\end{equation*}
Note that this is well-defined since $0 < \|g\|_{A(G)} <\infty$.  Moreover, by the Fourier inversion formula, we have
\begin{equation*}
\E{X_i(x)} = \sum_{\gamma \in \wh{G}}{\gamma(x)\frac{\wh{g}(\gamma)}{|\wh{g}(\gamma)|} \cdot \frac{|\wh{g}(\gamma)|}{\|g\|_{A(G)}}} = \frac{g(x)}{\|g\|_{A(G)}} \text{ for all }x \in G.
\end{equation*}
Regarding the variables $X_i(x)-g(x)\|g\|_{A(G)}^{-1}$ as elements of $L_p(\P^l)$ and noting, further, that
\begin{equation*}
\left\|X_i(x)-\frac{g(x)}{\|g\|_{A(G)}}\right\|_{L_\infty(\P^l)} \leq \|X_i(x)\|_{L_\infty(\P^l)} + \left\|\frac{g(x)}{\|g\|_{A(G)}}\right\|_{L_\infty(\P^l)} = 1+\frac{|g(x)|}{\|g\|_{A(G)}} \leq 2,
\end{equation*}
we can apply the Marcinkiewicz-Zygmund inequality (see \emph{e.g.} \cite[Lemma 3.1]{croabasis::}) to get
\begin{equation*}
\E{\left|\sum_{i=1}^l{X_i(x)} - \frac{g(x)l}{\|g\|_{A(G)}}\right|^p} = O\left(pl\right)^{p/2}.
\end{equation*}
We integrate the above against $\mu^+_1$ (recall this is one of the family of measures provided by the hypothesis that $\mu$ is $B$-approximately invariant) and rearrange so that
\begin{equation*}
\E{\left\|\|g\|_{A(G)}\frac{1}{l}\sum_{i=1}^l{X_i} - g\right\|_{L_p(\mu^+_1)}^p} = O\left(pl^{-1}\|g\|_{A(G)}^2\right)^{p/2}.
\end{equation*}
Now, take $l = O(\epsilon^{-2}p)$ such that the right hand side rescaled is at most $\left(\frac{\epsilon \|g\|_{A(G)}}{8}\right)^p$.  It follows that there are characters $\gamma_1,\dots,\gamma_l$ such that
\begin{equation*}
\left\|f - g\right\|_{L_p(\mu^+_1)} \leq \frac{\epsilon \|g\|_{A(G)}}{8} \text{ where } f:=\|g\|_{A(G)}\cdot \frac{1}{l}\sum_{i=1}^l{\frac{\wh{g}(\gamma_i)}{|\wh{g}(\gamma_i)|} \gamma_i}.
\end{equation*}
Since $\|f\|_{A(G)} \leq \|g\|_{A(G)}$ (by the triangle inequality) we may apply Lemma \ref{lem.bohrinv} to the set of character $\{\gamma_1,\dots,\gamma_l\}$ to get a Bohr system $B''$ with $\mathcal{C}^\Delta(G;B_1'')\leq 1$ and $\dim B''=O(l)=O(\epsilon^{-2}p)$ such that
\begin{equation*}
\|\tau_x(f)-f\|_{L_\infty(G)} \leq \frac{ \epsilon \|g\|_{A(G)}}{2} \text{ for all } x \in B'_{\frac{1}{2\pi}\epsilon}.
\end{equation*}
If $x \in B_1$ then by the approximate invariance of $\mu$ we have $\tau_x(\mu) \leq 2\mu_1^+$ and $\mu \leq 2\mu_1^+$, and so by the triangle inequality we have
\begin{align*}
\|\tau_x(g) - g\|_{L_p(\mu)} & \leq \|\tau_x(g)-\tau_x(f)\|_{L_p(\mu)} + \|\tau_x(f) - f\|_{L_p(\mu)} + \|f - g\|_{L_p(\mu)}\\
& = \|g-f\|_{L_p(\tau_{-x}(\mu))} + \|\tau_x(f)-f\|_{L_p(\mu)} + \|f - g\|_{L_p(\mu)}\\
& \leq 2\cdot 2^{\frac{1}{p}}\|g-f\|_{L_p(\mu^+_1)} + \|\tau_x(f)-f\|_{L_\infty(G)}.
\end{align*}
We conclude that
\begin{align*}
\|\tau_x(g) - g\|_{L_p(\mu)} \leq 2\cdot 2^{1/p}\cdot \frac{\epsilon  \|g\|_{A(G)}}{8} + \frac{\epsilon  \|g\|_{A(G)}}{2}\leq  \epsilon  \|g\|_{A(G)} \text{ whenever } x \in B_1 \cap B''_{\frac{1}{2\pi}\epsilon}.
\end{align*}
Put $B':=B \wedge ((\frac{1}{2\pi}\epsilon) B'')$ and note by Lemma \ref{lem.smi} parts (\ref{lem.smi.1}) and (\ref{lem.smi.2}), and the earlier bound on $\dim B''$ that
\begin{equation*}
\dim B' \leq \dim B + \dim \left(\left(\frac{1}{2\pi}\epsilon\right) B''\right) \leq \dim B + \dim B'' = \dim B + O(p\epsilon^{-2});
\end{equation*}
and by Lemma \ref{lem.dfc} part (\ref{pt2.difc}) and Lemma \ref{lem.bss} part (\ref{lem.bss.1}) and the bounds on $B''$ we have
\begin{align*}
\mathcal{C}^\Delta(A;B'_1) & = \mathcal{C}^\Delta\left(A\cap G;B_1\cap \left(\left(\frac{1}{2\pi}\epsilon\right) B''\right)_1\right)\\
& \leq \mathcal{C}^\Delta\left(A;B_1\right) \mathcal{C}^\Delta\left(G;\left(\left(\frac{1}{2\pi}\epsilon\right) B''\right)_1\right)\\
& \leq \mathcal{C}^\Delta\left(A;B_1\right) (8\pi \epsilon^{-1})^{\dim B''} \mathcal{C}^\Delta\left(G;B_1''\right) \leq \mathcal{C}^\Delta\left(A;B_1\right) (2\epsilon^{-1})^{O(p\epsilon^{-2})}.
\end{align*}
The result is proved.
\end{proof}

\section{Quantitative continuity}\label{sec.qc}

It is well known that if $G$ is a locally compact Abelian group and $f \in A(G)$ then $f$ is uniformly continuous.  If $G$ is finite then this statement has no content -- every function on $G$ is uniformly continuous -- but in the paper \cite{grekon::}, Konyagin and Green proved a statement which can be thought of as a quantitative version of this fact which still has content for finite Abelian groups.  The main purpose of this section is to prove the following result of this type using essentially their method.
\begin{proposition}\label{prop.screl}
Suppose that $B$ is a Bohr system of dimension at most $d$ (for some $d \geq 1$), $f \in A(G)$, and $\delta,\kappa \in (0,1]$ and $p \geq 1$ are parameters.  Then there is a Bohr system $B'\leq B$ such that for any $A \subset G$ we have
\begin{equation*}
\mathcal{C}^\Delta(A;B_1') \leq \exp(O(\delta^{-1}d\log 2\kappa^{-1}d + p\delta^{-3}\log^32p\kappa^{-1}\delta^{-1})) \mathcal{C}^\Delta(A;B_1)
\end{equation*}
and
\begin{equation*}
\dim B' \leq d + O(p\delta^{-2}\log^22\delta^{-1}),
\end{equation*}
and a $B'$-approximately invariant probability measure $\mu$ and a probability measure $\nu$ supported on $B'_\kappa$ such that
\begin{equation*}
\sup_{x \in G}{\|f - f\ast \mu\|_{L_p(\tau_x(\nu))}} \leq \delta \|f\|_{A(G)}.
\end{equation*}
\end{proposition}
We shall prove Proposition \ref{prop.screl} iteratively using the following lemma (which is, itself, proved iteratively).
\begin{lemma}\label{lem.it}
Suppose that $B$ is a Bohr system of dimension at most $d$ (for some $d\geq 1$), $\nu$ is a $B$-approximately invariant probability measure, $\mu$ is a probability measure supported on a set $X$, $f \in A(G)$ and $\delta,\eta \in (0,1]$ and $p \geq 1$ are parameters. Then at least one of the following is true:
\begin{enumerate}
\item we have
\begin{equation*}
\sup_{x \in G}{\|f - f\ast \mu\|_{L_p(\tau_x(\nu))}} \leq \delta \|f\|_{A(G)};
\end{equation*}
\item there is some $1\geq \rho =\Omega\left(\delta\right)$ and a Bohr system $B' \leq B$ such that for any $A\subset G$ we have
\begin{equation*}
\mathcal{C}^\Delta(A;B_1') \leq \exp(O(p\rho^2\delta^{-2}\log^32p\delta^{-1}+d\log d))\mathcal{C}^\Delta(A;B_1)
\end{equation*}
and
\begin{equation*}
 \dim B' \leq \dim B + O(p\rho^2\delta^{-2}\log^22\delta^{-1}),
\end{equation*}
such that
\begin{equation*}
\sum_{\gamma \in N\left(B'_{2^{-7}\delta \eta},\eta\right)\setminus N\left(X, 2^{-5}\delta\right)}{|\wh{f}(\gamma)|} \geq \rho\|f\|_{A(G)} .
\end{equation*}
\end{enumerate}
\end{lemma}
\begin{proof}
Since the hypotheses and conclusions are invariant under translation by $x$ it suffices to prove that if
\begin{equation}\label{eqn.ass}
\|f - f\ast \mu\|_{L_p(\nu)} > \delta \|f\|_{A(G)},
\end{equation}
then we are in the second case of the lemma.

Let $\kappa := \lceil \log_2 8\delta^{-1}\rceil^{-1}$ for reasons which will become clear later; at this stage it suffices to note that $\kappa \in (0,1/2]$.  Define $\delta_i:=(1-\kappa)^i\delta$ for integers $i$ with $0 \leq i \leq \kappa^{-1}$ and put $g_0:=f-f\ast \mu$.  Suppose that we have defined a function $g_i$ such that
\begin{equation*}
\|g_i\|_{L_p(\nu)} > \delta_i\|f\|_{A(G)}, \|g_i\|_{A(G)} \leq 2^{1-i}\|f\|_{A(G)} \text{ and } g_i = g_0 \ast \mu_i
\end{equation*}
for some probability measure $\mu_i$.  By taking $\mu_0$ to be the delta probability measure assigning mass $1$ to $0_G$, we see from (\ref{eqn.ass}) that $g_i$ satisfies these hypotheses for $i=0$.

By Lemma \ref{lem.chang} applied to the function $g_i$, the Bohr system $B$ and measure $\nu$ with parameters $p$ and $\epsilon_i:=\kappa\|g_i\|_{L_p(\nu)}\|g_i\|_{A(G)}^{-1}$, there is a Bohr system $B^{(i)}$ with
\begin{equation}\label{eqn.covf}
\mathcal{C}^\Delta(A;B_1^{(i)}) \leq \exp(O(p\epsilon_i^{-2}\log 2\epsilon_i^{-1}))\mathcal{C}^\Delta(A;B_1) \text{ for any }A \subset G
\end{equation}
and
\begin{equation}\label{eqn.ds}
\dim B^{(i)} \leq \dim B + O(p\epsilon_i^{-2})
\end{equation}
such that
\begin{equation*}
\|\tau_x(g_i)-g_i\|_{L_p(\nu)} \leq \kappa\|g_i\|_{L_p(\nu)} \text{ for all }x \in B^{(i)}_{1}.
\end{equation*}
By Corollary \ref{cor.ubreg} applied to $B^{(i)}$ there is some $1 \geq \lambda_i=\Omega((1+\dim B^{(i)})^{-1})$ and a $\lambda_iB^{(i)}$-approximately invariant probability measure $\nu^{(i)}$ supported on $B^{(i)}_{1}$.  Integrating (and applying the integral triangle inequality) we conclude that
\begin{equation*}
\|g_i-g_i\ast \nu^{(i)}\|_{L_p(\nu)} \leq \kappa\|g_i\|_{L_p(\nu)},
\end{equation*}
and so by the triangle inequality and hypothesis on $g_i$ we have
\begin{equation*}
\|g_i \ast \nu^{(i)}\|_{L_p(\nu)} \geq \|g_i\|_{L_p(\nu)}-\kappa\|g_i\|_{L_p(\nu)}>\delta_{i+1}\|f\|_{A(G)}.
\end{equation*}
Put $g_{i+1}:= g_i \ast \nu^{(i)}$ and $\mu_{i+1}=\mu_i \ast \nu^{(i)}$.  If $\|g_{i+1}\|_{A(G)} \leq 2^{1-(i+1)}\|f\|_{A(G)}$ then repeat; otherwise terminate the iteration.  Since $\kappa \leq \frac{1}{2}$ and $x \mapsto (1-x)^{x^{-1}}$ is monotonically decreasing for all $x \in (0,1]$ we see that if $i\leq \kappa^{-1}$ then
\begin{align}\label{eqn.er}
\frac{1}{4}\delta \|f\|_{A(G)} \leq (1-\kappa)^{\kappa^{-1}}\delta \|f\|_{A(G)}&  \leq (1-\kappa)^{i}\delta \|f\|_{A(G)}\\ \nonumber & \leq \delta_{i}\|f\|_{A(G)} < \|g_i\|_{L_p(\nu)} \leq \|g_i\|_{A(G)}.
\end{align}
Given our choice of $\kappa$ we see that $2^{1-\kappa^{-1}}\|f\|_{A(G)} \leq \frac{1}{4}\delta \|f\|_{A(G)}$ and so it follows from (\ref{eqn.er}) that there is some minimal $i \leq \kappa^{-1}$ such that $\|g_{i}\|_{A(G)} > 2^{1-i}\|f\|_{A(G)}$.  In particular $2^{-i} \geq 2^{-4}\delta$.

By choice of $i$, construction of $\mu_i$, and definition of $g_0$ we have (where we use the fact that $\wh{f\ast \mu}(\gamma)=\wh{f}(\gamma)\wh{\mu}(\gamma)$)
\begin{align*}
2^{1-i}\|f\|_{A(G)} \leq \|g_i\|_{A(G)} & =\|g_0 \ast \mu_{i-1} \ast \nu^{(i-1)}\|_{A(G)}\\
& =\sum_{\gamma \in \wh{G}}{|\wh{f}(\gamma)||1-\wh{\mu}(\gamma)||\wh{\nu^{(i-1)}}(\gamma)||\wh{\mu_{i-1}}(\gamma)|}\\
& \leq \sum_{\gamma \in \wh{G}}{|\wh{f}(\gamma)||1-\wh{\mu}(\gamma)||\wh{\nu^{(i-1)}}(\gamma)|} .
\end{align*}
Hence
\begin{align*}
& \sum_{ \substack{ |\wh{\nu^{(i-1)}}(\gamma)|> 2^{-6}\delta\\ |1-\gamma(x)| \geq 2^{-5}\delta \text{ for some }x \in X}}{|\wh{f}(\gamma)||1-\wh{\mu}(\gamma)||\wh{\nu^{(i-1)}}(\gamma)|}\\
& \qquad \qquad + \sum_{ \substack{|1-\gamma(x)| <2^{-5}\delta \text{ for all }x \in X}}{|\wh{f}(\gamma)||1-\wh{\mu}(\gamma)||\wh{\nu^{(i-1)}}(\gamma)|}\\
& \qquad \qquad + \sum_{ \substack{|\wh{\nu^{(i-1)}}(\gamma)|\leq 2^{-6}\delta}}{|\wh{f}(\gamma)||1-\wh{\mu}(\gamma)||\wh{\nu^{(i-1)}}(\gamma)|}\\
& \qquad \qquad \qquad \qquad \geq  \sum_{\gamma \in \wh{G}}{|\wh{f}(\gamma)||1-\wh{\mu}(\gamma)||\wh{\nu^{(i-1)}}(\gamma)|}.
\end{align*}
If $\gamma\in \wh{G}$ is such that $|1-\gamma(x)| < 2^{-5}\delta$ for all $x\in X$, then by the triangle inequality $|1-\wh{\mu}(\gamma)| \leq 2^{-5}\delta$, and hence the second sum on the left is at most $2^{-5}\delta \|f\|_{A(G)}$.  Since $|1-\wh{\mu}(\gamma)| \leq 2$ by the triangle inequality, the third sum on the left is at most $2\|f\|_{A(G)}\cdot 2^{-6}\delta$, and so by the triangle inequality we have
\begin{align*}
\sum_{ \substack{|\wh{\nu^{(i-1)}}(\gamma)|> 2^{-6}\delta\\ |1-\gamma(x)| \geq 2^{-5}\delta \text{ for some }x \in X}}{|\wh{f}(\gamma)||1-\wh{\mu}(\gamma)||\wh{\nu^{(i-1)}}(\gamma)|} & \geq 2^{1-i}\|f\|_{A(G)} - 2^{-4}\delta \|f\|_{A(G)}\\ & \geq 2^{1-i}\|f\|_{A(G)} -2^{-i}\|f\|_{A(G)} = 2^{-i}\|f\|_{A(G)}.
\end{align*}
Put $B':=\lambda_{i-1}B^{(i-1)}$ and apply Lemma \ref{lem.nest} to $\nu^{(i-1)}$ and $B'$ with parameters $2^{-6}\delta$ and $\eta$ to see that
\begin{align*}
&\{\gamma: |\wh{\nu^{(i-1)}}(\gamma)|> 2^{-6}\delta \text{ and } |1-\gamma(x)| \geq 2^{-5}\delta \text{ for some }x \in X\}\\
& \qquad \qquad \qquad \subset N(B'_{2^{-7}\delta \eta},\eta) \setminus N(X,2^{-5}\delta).
\end{align*}
Writing $\rho:=2^{-i-1} =\Omega(\delta)$ and recalling that $|1-\wh{\mu}(\gamma)||\wh{\nu^{(i-1)}}(\gamma)| \leq 2$ by the triangle inequality we have
\begin{equation*}
\sum_{\gamma \in N\left(B'_{2^{-7}\delta \eta},\eta\right)\setminus N\left(X, 2^{-5}\delta\right)}{|\wh{f}(\gamma)|} \geq \rho\|f\|_{A(G)} .
\end{equation*}
It remains to note that $\epsilon_{i-1} > \kappa \delta_{i-1}2^{i-2} = \Omega(\kappa \delta \rho^{-1})$ and so by Lemma \ref{lem.smi} part (\ref{lem.smi.2}), and (\ref{eqn.ds}) we see that $\dim B'$ satisfies the claimed bound.  Finally, by Lemma \ref{lem.bss} part (\ref{lem.bss.1}), (\ref{eqn.covf}), (\ref{eqn.ds}), and the lower bound on $\lambda_i$ we have
\begin{align*}
\mathcal{C}^\Delta(A;B_1') & = \mathcal{C}^\Delta(A;B_{\lambda_{i-1}}^{(i-1)})\\
& \leq (4\lambda_{i-1}^{-1})^{\dim  B^{(i-1)}}\mathcal{C}^\Delta(A;B_1^{(i-1)})\\
& \leq (4\lambda_{i-1}^{-1})^{\dim  B^{(i-1)}}\exp(O(p\epsilon_{i-1}^{-2}\log 2\epsilon_{i-1}^{-1}))\mathcal{C}^\Delta(A;B_1)\\
& \leq d^{O(d)}\exp(O(p\rho^2\delta^{-2}\log^3 2p\delta^{-1}))\mathcal{C}^\Delta(A;B_1),
\end{align*}
for any $A \subset G$ from which the lemma follows.
\end{proof}
\begin{proof}[Proof of Proposition \ref{prop.screl}]
We proceed iteratively constructing Bohr systems $(B^{(i)})_{i =0}^J$ and reals $(\rho_i)_{i=1}^J$, and $(d_i)_{i=0}^J$, such that 
\begin{enumerate}
\item $\dim B^{(i)} \leq d_i$;
\item \label{pt.nes} $B^{(i+1)} \leq B^{(i)}$;
\item \label{pt.rho} $1 \geq \rho_i = \Omega(\delta)$ and
\begin{equation*}
 \sum_{N(B^{(i+1)}_1,2^{-5}\delta)\setminus N(B^{(i)}_1,2^{-5}\delta)}{|\wh{f}(\gamma)|} \geq \rho_i\|f\|_{A(G)};
\end{equation*}
\item \label{pt.grow} for any $A \subset G$ we have
\begin{equation*}
\mathcal{C}^\Delta(A;B_1^{(i+1)}) \leq \exp(O(p\rho_i^2\delta^{-2}\log^32p\delta^{-1} + d_i\log \kappa^{-1}d_i))\mathcal{C}^\Delta(A;B_1^{(i)});
\end{equation*}
\item \label{pt.gd}
\begin{equation*}
d_{i+1}\leq d_i + O(p\rho_i^2\delta^{-2}\log^22\delta^{-1}).
\end{equation*}
\end{enumerate}
We initialise with $B^{(0)}:=B$ and $d_0:=d$.  Suppose that we are at stage $i$ of the iteration.  Apply Corollary \ref{cor.ubreg} to $B^{(i)}$ to get some $\lambda_i = \Omega((1+\dim B^{(i)})^{-1})$ and a $\lambda_i B^{(i)}$-approximately invariant probability measure $\mu_i$ supported on $B^{(i)}_1$.  Apply Corollary \ref{cor.ubreg} to $\kappa\lambda_i B^{(i)}$ to get some
\begin{equation*}
\lambda_i'= \Omega((1+\dim \kappa\lambda_iB^{(i)})^{-1}) = \Omega(d_i^{-1})
\end{equation*}
and a $\lambda_i'\kappa\lambda_i B^{(i)}$-approximately invariant probability measure $\nu_i$ supported on $\kappa\lambda_i B^{(i)}_1$.

By Lemma \ref{lem.smi} part (\ref{lem.smi.2}) we see that
\begin{equation*}
\dim \lambda_i'\kappa\lambda_i B^{(i)} \leq \dim B^{(i)} \leq d_i.
\end{equation*}
Apply Lemma \ref{lem.it} to $A$, $\lambda_i'\kappa\lambda_i B^{(i)}$, $d_i$, $\nu_i$, $\mu_i$, $B^{(i)}_1$ and $f$ with parameters $\delta$ and $2^{-5}\delta$ (and $p$).

Suppose the conclusion of the second case of Lemma \ref{lem.it} holds.  Then there is some $\rho_i=\Omega(\delta)$ and a Bohr system $B^{(i,1)} \leq \lambda_i'\kappa\lambda_i B^{(i)}$ such that
\begin{equation}\label{eqn.dddd}
\dim B^{(i,1)} \leq \dim \lambda_i'\kappa\lambda_i B^{(i)} + O(p\rho_i^2\delta^{-2}\log^2\delta^{-1});
\end{equation}
and for any $A \subset G$ we have
\begin{align*}
\mathcal{C}^\Delta(A;B^{(i,1)}_1) & \leq \exp(O(p\rho_i^2\delta^{-2}\log^32p\delta^{-1} + d_i\log d_i))\mathcal{C}^\Delta(A;(\lambda_i'\kappa\lambda_i B^{(i)})_1).
\end{align*}
However, for any $A \subset G$ we have
\begin{equation*}
\mathcal{C}^\Delta(A;(\lambda_i'\kappa\lambda_i B^{(i)})_1) = \mathcal{C}^\Delta(A;B^{(i)}_{\lambda_i'\kappa\lambda_i }) \leq (4\lambda_i^{-1}(\lambda_i')^{-1}\kappa^{-1})^{\dim B^{(i)}} \mathcal{C}^\Delta(A;B^{(i)}_1),
\end{equation*}
by Lemma \ref{lem.bss} part (\ref{lem.bss.1}).  Thus for any $A \subset G$ we have
\begin{equation*}
\mathcal{C}^\Delta(A;B^{(i,1)}_1) \leq \exp(O(p\rho_i^2\delta^{-2}\log^32p\delta^{-1} + d_i\log \kappa^{-1}d_i))\mathcal{C}^\Delta(A;B^{(i)}_1).
\end{equation*}
 Additionally we have
\begin{equation*}
 \sum_{N\left(B^{(i,1)}_{2^{-12}\delta^2},2^{-5}\delta\right) \setminus N(B^{(i)}_1,2^{-5}\delta)}{|\wh{f}(\gamma)|} \geq \rho_i \|f\|_{A(G)}.
\end{equation*}
Put $B^{(i+1)}:=(2^{-12}\delta^2)B^{(i,1)}$ and we get (\ref{pt.rho}).  Moreover,
\begin{equation*}
B^{(i+1)}=(2^{-12}\delta^2)B^{(i,1)} \leq B^{(i,1)} \leq \lambda_i'\kappa\lambda_i B^{(i)} \leq B^{(i)}
\end{equation*}
by the order preserving nature of dilation and the fact that $2^{-12}\delta \leq 1$ and $\lambda_i'\kappa\lambda_i \leq 1$; it follows that we have (\ref{pt.nes}).  Now, Lemma \ref{lem.smi} part (\ref{lem.smi.2}) and (\ref{eqn.dddd}) gives
\begin{align*}
\dim B^{(i+1)} = \dim(2^{-12}\delta^2)B^{(i,1)} \leq \dim B^{(i,1)} & \leq \dim \lambda_i'\kappa\lambda_i B^{(i)} + O(p\rho_i^2\delta^{-2}\log^2\delta^{-1})\\ & \leq \dim B^{(i)} + O(p\rho_i^2\delta^{-2}\log^2\delta^{-1})\\ & \leq d_i + O(p\rho_i^2\delta^{-2}\log^2\delta^{-1}),
\end{align*}
from which we get (\ref{pt.gd}).  Finally, Lemma \ref{lem.bss} part (\ref{lem.bss.1}) tells us that for any $A \subset G$ we have
\begin{align}\label{eqn.previousil}
\mathcal{C}^\Delta(A;B^{(i+1)}_1)
& =\mathcal{C}^\Delta(A;B^{(i,1)}_{2^{-12}\delta^2}) \\ \nonumber
& \leq (2^{14}\delta^{-2})^{\dim  B^{(i,1)}}\mathcal{C}^\Delta(A;B^{(i,1)}_1)\\ \nonumber
& \leq \exp(O(p\rho_i^2\delta^{-2}\log^32p\delta^{-1} + d_i\log \kappa^{-1}d_i))\mathcal{C}^\Delta(A;B^{(i)}_1),
\end{align}
from which we get (\ref{pt.grow}).

In the light of (\ref{pt.nes}) we see that $B^{(i+1)}_1 \subset B^{(i)}_1$ and hence
\begin{equation*}
N(B_1^{(i+1)},2^{-5}\delta) \supset N(B_1^{(i)},2^{-5}\delta).
\end{equation*}
It follows that after $i$ steps we have
\begin{equation*}
\|f\|_{A(G)} \geq \sum_{N(B_1^{(i)},2^{-5}\delta)}{|\wh{f}(\gamma)|} \geq \sum_{j \leq i}{\rho_j\|f\|_{A(G)}},
\end{equation*}
and hence
\begin{equation*}
\sum_{j \leq i}{\rho_j} \leq 1.
\end{equation*}
Since $\rho_j=\Omega(\delta)$ we conclude that we must be in the first case of Lemma \ref{lem.it} at some step $J=O(\delta^{-1})$ of the iteration. In light of (\ref{pt.gd}) we see that
\begin{equation*}
d_i \leq d + O(p\delta^{-2}\log^{2}2\delta^{-1}) \text{ for all }i \leq J.
\end{equation*}
It then follows from (\ref{eqn.previousil}) that for any $A \subset G$ we have
\begin{align*}
\mathcal{C}^\Delta(A;B_1^{(J)}) & \leq \left(\prod_{j <J}{\exp(O(p\rho_j^2\delta^{-2}\log^32p\delta^{-1} + d_j\log \kappa^{-1}d_j))}\right)\mathcal{C}^\Delta(A;B_1)\\
& \leq \exp(O(Jd\log 2\kappa^{-1}d + Jp\delta^{-2}\log^32p\kappa^{-1}\delta^{-1}))\mathcal{C}^\Delta(A;B_1).
\end{align*}

We now put $B':=\lambda_J B^{(J)}$, $\mu:=\mu_J$ and $\nu:=\nu_J$, so that
\begin{equation*}
\sup_{x \in G}{\|f - f\ast \mu\|_{L_p(\tau_x(\nu))}} \leq \delta \|f\|_{A(G)}.
\end{equation*}
By Lemma \ref{lem.bss} part (\ref{lem.bss.1}) we see that for any $A \subset G$ we have
\begin{align*}
\mathcal{C}^\Delta(A;B_1') & = \mathcal{C}^\Delta(A;B^{(J)}_{\lambda_J})\\
 & \leq (4\lambda_J^{-1})^{\dim B^{(J)}} \mathcal{C}^\Delta(A;B^{(J)}_1)  \leq \exp(O(d_J\log 2d_J))\mathcal{C}^\Delta(A;B^{(J)}_1) ;
\end{align*}
and by Lemma \ref{lem.smi} part (\ref{lem.smi.2}) we have
\begin{equation*}
\dim B' = \dim \lambda_J B^{(J)} \leq \dim B^{(J)} \leq d_J.
\end{equation*}
The result follows.
\end{proof}

\section{A Freiman-type theorem}\label{sec.ft}

The purpose of this section is to prove the following proposition, which is a routine if slightly fiddly variation on existing material in the literature.
\begin{proposition}\label{prop.f}
Suppose that $A$ is non-empty and $m_G(A+A)\leq Km_G(A)$.  Then there is a Bohr system $B$ with
\begin{equation*}
 \mathcal{C}^\Delta(A;B_1) = \exp(O(\log^32K (\log (2\log 2K))^4))
\end{equation*}
and
\begin{equation*}
\dim B=O(\log^32K (\log (2\log 2K))^4),
\end{equation*}
such that
\begin{equation}\label{eqn.jp}
\|1_A \ast \beta\|_{L_\infty(G)} =\exp(-O(\log2K (\log (2\log 2K))))
\end{equation}
for any probability measure $\beta$ supported on $B_1$.
\end{proposition}
The proposition itself is closely related to Freiman's theorem and we refer the reader to \cite[Chapter 5]{taovu::} for a discussion of Freiman's theorem.  For our purposes there are two key differences:
\begin{enumerate}
\item Freiman's theorem is usually only stated with the first two conclusions.  It is possible to infer the fact that
\begin{equation*}
\|1_A \ast \beta\|_{L_\infty(G)} = \exp(-O(\log^32K (\log (2\log 2K))^4))
\end{equation*}
for any probability measure $\beta$ supported on $B_1$ from the bound on $ \mathcal{C}^\Delta(A;B_1)$, and the fact that one can do better and get (\ref{eqn.jp}) in this sort of situation is an unpublished observation of Green and Tao.
\item Freiman's theorem also produces a coset progression rather than a Bohr system.  A set $M$ is a \textbf{$d$-dimensional coset progression} if there are arithmetic progressions $P_1,\dots,P_d$ and a subgroup $H$ such that $M=P_1+\dots +P_d+H$.  This definition was made by Green and Ruzsa in \cite{greruz::0} when they gave the first proof of Freiman's theorem for Abelian groups.  The conclusion of Freiman's theorem then is that there is a coset progression $M$ with
\begin{equation*}
\mathcal{C}^\Delta(A;M) =O_K(1) \text{ and } \dim M =O_K(1),
\end{equation*}
and the challenge is to identify good estimates for the $O_K(1)$-terms.
\end{enumerate}
For us it is the quantitative aspects of Proposition \ref{prop.f} that are important.  The quantitative aspects of Freiman's theorem are surveyed in \cite{san::10}, and primarily arise from the quantitative strength of the Croot-Sisask Lemma (in particular the $m$-dependence in \cite[Proposition 3.3]{crosis::}), but also some combinatorial arguments of Konyagin \cite{kon::1} discussed just before \cite[Corollary 8.4]{san::10}. Conjecturally all the big-$O$ terms should be $O(\log 2K)$, though the proof below does not come close to that.  It could probably be tightened up to same on the power of $\log (2\log 2K)$ in the first two estimates above, at least reducing the $4$ to a $3$ but quite possible further.

We shall prove Proposition \ref{prop.f} as a combination of the next three results which we shall show in \S\ref{sec.42}, \S\ref{sec.43}, and \S\ref{sec.44} respectively.  We say that a set $X$ has \textbf{relative polynomial growth of order $d$} if
\begin{equation*}
m_G(nX) \leq n^dm_G(X) \text{ for all }n \geq 1.
\end{equation*}
The first result can be read out of the proof of \cite[Proposition 2.5]{san::10} and essentially captures the power of the Croot-Sisask Lemma for our purposes.
\begin{lemma}\label{lem.fsurvey}
Suppose that $A$ is non-empty with $m_G(A+A) \leq Km_G(A)$.  Then there is a symmetric set $X$ containing the identity of relative polynomial growth of order $O(\log^32K (\log (2\log 2K))^3)$ and
\begin{equation*}
m_G(X) \geq \exp(-O(\log^32K (\log (2\log 2K))^3))m_G(A),
\end{equation*}
and some naturals $m=\Omega(\log 2K (\log (2\log 2K)))$ and $r=O(\log (2\log 2K))$ such that $mX \subset r(A-A)$.
\end{lemma}
The second result is one we have already touched on and captured a key insight of Green and Ruzsa in \cite{greruz::0} that allows passage from relative polynomial growth to structure.
\begin{lemma}\label{lem.rpgb}
Suppose that $X$ is a symmetric non-empty set with relative polynomial growth of order $d\geq 1$.  Then there is a Bohr system $B$ with
\begin{equation*}
\dim B=O(d) \text{ and } m_G(B_1) = d^{O(d)}m_G(X).
\end{equation*}
such that $X-X \subset B_1$.
\end{lemma}
Finally the last lemma is a development of a result of Bogolio{\`u}boff \cite{bog::} revived for this setting by Ruzsa \cite{ruz::9}, and then refined by Chang \cite{cha::0}.  
\begin{lemma}\label{lem.bc}
Suppose that $A$ is a non-empty set, $B$ is a Bohr system and $\mu$ is a $B$-approximately invariant probability measure, $S \subset B_1$ has $\mu(S)>0$, and $L$, non-empty, is such that $\|1_L \ast \mu_S\|_{L_2(m_G)}^2 \geq \epsilon m_G(L)$.  Then there is a Bohr system $B' \leq B$ with
\begin{equation*}
\mathcal{C}^\Delta(A;B_1') \leq (2\epsilon^{-1})^{O(\epsilon^{-2}\log 2
\mu(S)^{-1})}\mathcal{C}^\Delta(A;B_1)
\end{equation*}
and
\begin{equation*}
\dim B' = \dim B+ O( \epsilon^{-2}\log 2\mu(S)^{-1})
\end{equation*}
such that $B'_{1} \subset L-L + S-S$.
\end{lemma}
With these results in hand we can turn to proving the main result of the section.
\begin{proof}[Proof of Proposition \ref{prop.f}]
We apply Lemma \ref{lem.fsurvey} to $A$ to get a non-empty symmetric set $X$ of relative polynomial growth of order $O(\log 2K \log (2\log 2K))^3$ with
\begin{equation}\label{eqn.mv}
m_G(X) \geq \exp(-O(\log 2K \log (2\log 2K))^3)m_G(A),
\end{equation}
and natural numbers $m=\Omega(\log 2K \log (2\log 2K))$ and $r=O(\log (2\log 2K))$ such that $mX \subset r(A-A)$. By Lemma \ref{lem.rpgb} there is a Bohr system $B'$ with $X-X\subset B'_1$ such that
\begin{equation*}
\dim B'=O(\log 2K \log (2\log 2K))^3 \text{ and } m_G(B'_1) \leq \exp(O(\log^32K (\log (2\log 2K))^4))m_G(X).
\end{equation*}
By nesting of Bohr we have that
\begin{equation*}
\mathcal{C}^\Delta(X-X;B'_{1}) \leq \mathcal{C}^\Delta\left(B'_1;B'_{1}\right) \leq \mathcal{C}^\Delta\left(B'_1;B'_{\frac{1}{2}}\right) \leq 2^{\dim B'}=\exp(O(\log 2K \log (2\log 2K))^3).
\end{equation*}
By Corollary \ref{cor.ubreg} there is a probability measure $\mu$ and a Bohr system $B''=\lambda B'$ for some $\lambda = \Omega((1+\dim B')^{-1})$ such that $\mu$ is supported on $B'_1$ and $\mu$ is $B''$-approximately invariant.  By Lemma \ref{lem.bss} part (\ref{lem.bss.1}) (with reference set $X-X$) we have
\begin{equation*}
\mathcal{C}^\Delta(X-X;B_{1}'') \leq (4\lambda^{-1})^{\dim B''}\mathcal{C}^\Delta(X-X;B'_{1}) \leq  \exp(O(\log^32K (\log (2\log 2K))^4)).
\end{equation*}
By the second inequality in Lemma \ref{lem.smi} part (\ref{lem.smi.3}) and the definition of dimension there is a set $T$ with 
\begin{equation*}
|T| \leq 2^{2\dim^* B'} = \exp(O(\log 2K \log (2\log 2K))^3) \text{ and } B'_1 \subset T+B'_{\frac{1}{2}}.
\end{equation*}
It follows from nesting of Bohr sets that
\begin{equation*}
B_1'+B_1' \subset T+T + B_{\frac{1}{2}}' +B_{\frac{1}{2}}' \subset T+T+B_1'.
\end{equation*}
Now, since $\supp \mu \subset B'_1$ we see that $1_{B_1'+B_1'} \ast \mu(x) =1$ for all $x \in B_1'$ and so (since $0_G \in X$) we have
\begin{align*}
m_G(X) & \leq \langle 1_X, 1_{B_1'+B_1'} \ast \mu\rangle_{L_2(m_G)}\\
& \leq \sum_{t \in T-T}{\langle 1_X\ast \mu,1_{t+B_1'}\rangle_{L_2(m_G)}} \leq |T-T| \sup_{x \in G}{\mu(x+X)} m_G(B_1').
\end{align*}
Inserting the upper bound for $m_G(B_1')$ and the upper bound for $|T|$, it follows that there is some $x$ such that
\begin{equation*}
\mu(x+X) \geq \exp(-O(\log^32K (\log (2\log 2K))^4)).
\end{equation*}

Now, put $S:=x+X$ and note from Pl{\"u}nnecke's inequality that
\begin{align*}
\prod_{l=0}^{m-1}{\frac{m_G(A-A+lS+S)}{m_G(A-A+lS)}} & = \prod_{i=0}^{m-1}{\frac{m_G(A-A+lX+X)}{m_G(A-A+lX)}}\\\ & = \frac{m_G(A-A+mX)}{m_G(A-A)} \leq K^{2(r+1)}.
\end{align*}
Given the lower bound on $m$ and upper bound on $r$ it follows that there is some $0 \leq l \leq m-1$ such that
\begin{equation*}
m_G(A-A+lS+S) \leq K^{\frac{2(r+1)}{m}}m_G(A-A+lS) = O(m_G(A-A+lS)).
\end{equation*}
Putting $L:=A-A+lS$ it follows by the Cauchy-Schwarz inequality that
\begin{equation*}
\|1_L \ast \mu_S\|_{L_2(m_G)}^2 \geq \frac{m_G(L)^2}{m_G(L+S)} = \Omega(m_G(L)).
\end{equation*}
By Lemma \ref{lem.bc} (with reference set $X-X$) we then see that there is a Bohr system $B \leq B''$ with
\begin{align}
\nonumber \mathcal{C}^\Delta(X-X;B_1) & \leq \exp(O(\log^32K (\log (2\log 2K))^4))\mathcal{C}^\Delta(X-X;B_1'')\\
\label{eqn.fF} &\leq \exp(O(\log^32K (\log (2\log 2K))^4))
\end{align}
and
\begin{align*}
\dim B&=\dim B'' +O(\log^32K (\log (2\log 2K))^4) =O(\log^32K (\log (2\log 2K))^4),
\end{align*}
such that
\begin{align*}
B_1  \subset S+L-L-S & \subset 2(A-A) + (l+1)(S-S)\\ & =2(A-A) + 2(l+1)X \subset (2r+1)(A-A).
\end{align*}
Since $0_G \in X$ we see that $X \subset r(A-A)$ and hence by Lemma \ref{lem.rcr} and Pl{\"u}nnecke's inequality (and (\ref{eqn.mv}) and (\ref{eqn.fF})) we have
\begin{align*}
\mathcal{C}^\Delta (A;B_{1})& \leq \frac{m_G(A+X)}{m_G(X)}\mathcal{C}^\Delta (X-X;B_{1})\\
& \leq \frac{K^{r+1}m_G(A)}{\exp(-O(\log 2K \log (2\log 2K))^3)m_G(A)}\exp(O(\log^32K (\log (2\log 2K))^4))\\
& = \exp(O(\log^32K (\log (2\log 2K))^4)).
\end{align*}
Finally, if $\beta$ is supported on $B_1$ then
\begin{equation*}
m_G(A) \leq \langle 1_A \ast \beta,1_{A+4r(A-A)}\rangle_{L_2(m_G)} \leq \|1_A \ast \beta\|_{L_\infty(G)}m_G(A+4r(A-A))
\end{equation*}
from which the final bound follows by Pl{\"u}nnecke's inequality.
\end{proof}

\subsection{Croot-Sisask Lemma arguments}\label{sec.42}

The aim of this section is to prove the following lemma.
\begin{lemma*}[Lemma \ref{lem.fsurvey}]
Suppose that $A$ is non-empty with $m_G(A+A) \leq Km_G(A)$.  Then there is a symmetric set $X$ containing the identity of relative polynomial growth of order $O(\log^32K (\log (2\log 2K))^3)$ and
\begin{equation*}
m_G(X) \geq \exp(-O(\log^32K (\log (2\log 2K))^3))m_G(A),
\end{equation*}
and some naturals $m=\Omega(\log 2K (\log (2\log 2K)))$ and $r=O(\log (2\log 2K))$ such that $mX \subset r(A-A)$.
\end{lemma*}
The material follows the proof of \cite[Proposition 8.5]{san::10} very closely, though we shall need some minor modifications.  We start by recording two results used to prove that proposition.
\begin{corollary}[{\cite[Corollary 5.3]{san::10}}]\label{cor.useful}
Suppose that $X \subset G$ is a symmetric set and $m_G((3k+1)X) < 2^km_G(X)$ for some $k \in \N$.  Then $X$ has relative polynomial growth of order $O(k)$.
\end{corollary}
This is just a variant of Chang's covering lemma from \cite{cha::0} (see also \cite[Lemma 5.31]{taovu::}).
\begin{lemma}[Croot-Sisask, {\cite[Lemma 7.1]{san::10}}]\label{lem.cs}
Suppose that $f \in L_p(m_G)$ for some $p \in [2,\infty)$, $S,T \subset G$ are non-empty such that $m_G(S+T) \leq Lm_G(S)$, and $\eta \in (0,1]$ is a parameter.  Then there is a symmetric set $X$ containing the identity with
\begin{equation*}
m_G(X) \geq (2L)^{-O(\eta^{-2}p)}m_G(T)
\end{equation*}
such that
\begin{equation*}
 \|\tau_x(f \ast m_S) - f \ast m_S\|_{L_p(m_G)} \leq \eta \|f\|_{L_p(m_G)}\text{ for all }x \in X.
\end{equation*}
\end{lemma}
This captures the content of the Croot-Sisask Lemma \cite[Proposition 3.3]{crosis::} for our purposes.

We shall also need a slight variant of \cite[Proposition 8.3]{san::10}.
\begin{proposition}\label{prop.key}
Suppose that $A,S$ and $T$ are non-empty with $m_G(A+S) \leq Km_G(A)$ and $m_G(S+T) \leq Lm_G(S)$, and $m \in \N$ is a parameter.  Then there is a symmetric set $S$ containing the identity with
\begin{equation*}
m_G(X) \geq \exp(-O(m^2(\log 2K) \log 2L))m_G(T) \text{ and } mX \subset S+A-A-S.
\end{equation*}
\end{proposition}
\begin{proof}
Let $f:=1_{A+S}$ and apply the Croot-Sisask lemma (Lemma \ref{lem.cs}) with parameters $\eta$ and $p$ (to be optimised later) to get a symmetric set $X$ containing the identity with $m_G(X) \geq (2L)^{-O(\eta^{-2}p)}m_G(T)$ such that
\begin{equation*}
\|\tau_x(1_{A+S} \ast m_{-S}) - 1_{A+S}\ast m_{-S}\|_{L_p(m_G)} \leq \eta \|1_{A+S}\|_{L_p(m_G)} \text{ for all } x \in X.
\end{equation*}
It follows by the triangle inequality that
\begin{equation*}
\|\tau_x(1_{A+S} \ast m_{-S}) - 1_{A+S}\ast m_{-S}\|_{L_p(m_G)} \leq \eta m \|1_{A+S}\|_{L_p(m_G)} \text{ for all } x \in mX.
\end{equation*}
Taking an inner product with $m_A$ we see that for all $x \in X$ we have
\begin{equation*}
|\langle \tau_x(1_{A+S} \ast m_{-S}),m_A\rangle - \langle 1_{A+S} \ast m_{-S},m_A\rangle| \leq \eta m \|1_{A+S}\|_{L_p(m_G)}\|m_A\|_{L_{p'}(m_G)}
\end{equation*}
where $p'$ is the conjugate exponent to $p$.  Now
\begin{equation*}
 \langle 1_{A+S} \ast m_{-S},m_A\rangle = \langle 1_{A+S},m_A \ast m_S\rangle = 1.
\end{equation*}
Thus
\begin{equation*}
|m_{A} \ast 1_{-(A+S)} \ast m_{S}(x) - 1| \leq \eta m K^{1/p} \text{ for all }x \in X.
\end{equation*}
We take $p = 2+\log K$, and then $\eta = \Omega(m^{-1})$ such that the term on the right is at most $1/2$ to get the desired conclusion.
\end{proof}
The above proposition is almost all we need for our main argument and it can be used in the proof of Lemma \ref{lem.fsurvey} below to give a result with only slightly weaker bounds.  However, we shall want a slight strengthening proved using the aforementioned idea of Konyagin \cite{kon::1}. 
\begin{proposition}\label{prop.kony}
Suppose that $A$ is non-empty with $m_G(A+A) \leq Km_G(A)$ and $r,s \in \N$ are parameters with $r \geq 3$.  Then there is an integer $m=\Omega(sr\log^{1-O(r^{-1})} 2K)$ and a symmetric set $T$ such that
\begin{equation*}
mT \subset r(A-A) \text{ and } m_G(T) \geq \exp(-O(s^2r^3\log^3 2K))m_G(A).
\end{equation*}
\end{proposition}
\begin{proof}
Define sequences
\begin{equation*}
r_i:=3 \times 2^i-2 \text{ and } K_i:=\frac{m_G(r_i(A-A))}{m_G(A)};
\end{equation*}
by Pl{\"u}nnecke's inequality we have $K_i \leq K^{2r_i}$.

We proceed inductively to define sequences of non-empty sets $(S_i)_{i \geq 0}$ and $(T_i)_{i \geq 0}$ with
\begin{equation*}
L_i:=\frac{m_G(S_i+T_i)}{m_G(S_i)} \text{ and }m_i:= s\left\lceil \frac{\log 2K_{i+1}}{\sqrt{\log 2L_i}}\right\rceil.
\end{equation*}
We shall establish the following properties inductively for all $i \geq 0$.
\begin{enumerate}
\item \label{dddd} $S_i$ and $T_i$ are symmetric sets containing the identity such that
\begin{equation*}
(A-A) \subset S_i \subset r_i(A-A);
\end{equation*}
\item \label{ffff} and
\begin{equation*}
L_i \leq \exp(4\log^{2^{-i}}2K);
\end{equation*}
\item \label{hhhh} and
\begin{equation*}
m_iT_{i+1} \subset S_i + A-A -S_i;
\end{equation*}
\item \label{szz} and
\begin{equation*}
m_G(T_{i+1}) \geq \exp\left(-O\left(s^2\left(\sum_{j=0}^i{r_{j+1}^3}\right)\log^32K\right)\right)m_G(T_0).
\end{equation*}
\end{enumerate}
We initialise with $S_0:=A-A$ and $T_0:=A-A$ so that $S_0$ and $T_0$ are symmetric sets containing the identity (since $A$ is non-empty) and
\begin{equation*}
(A-A) = S_0=A-A = 1(A-A)=r_0(A-A),
\end{equation*}
whence (\ref{dddd}) holds.  Moreover, by Pl{\"u}nnecke's inequality we have
\begin{equation*}
L_0=\frac{m_G(S_0+T_0)}{m_G(S_0)} = \frac{m_G((A-A)+(A-A))}{m_G(A-A)} \leq K^4\leq \exp(4\log 2K),
\end{equation*}
so that (\ref{ffff}) holds.

Suppose that we are at stage $i$ of the iteration.  Apply Proposition \ref{prop.key} to the sets $A$, $S_i$, and $T_i$ with parameter $m_i$.  This produces a symmetric set $T_{i+1}$ containing the identity such that
\begin{equation}\label{la}
m_G(T_{i+1}) \geq \exp(-O(m_i^2(\log 2K_i)\log 2L_i))m_G(T_i) \text{ and } m_iT_{i+1} \subset S_i + A-A -S_i.
\end{equation}
First note that given the definition of $m_{i}$, $r_i$ and $r_{i+1}$ we have
\begin{align*}
m_G(T_{i+1}) & \geq \exp(-O(s^2(\log^22K_{i+1})\log 2K_i))m_G(T_i)\\ &  = \exp(-O(s^2r_{i+1}^3\log^32K))m_G(T_i),
\end{align*}
and so we get (\ref{szz}).  The second part of (\ref{la}) ensures (\ref{hhhh}).  Moreover, we have
\begin{align*}
m_{i}T_{i+1} + (A-A) & \subset S_i+A-A-S_i + A-A\\ & \subset r_i(A-A) + (A-A) - r_i(A-A) + (A-A)\\ & = (2r_i+2)(A-A)=r_{i+1}(A-A).
\end{align*}
By the pigeon-hole principle there is some non-negative integer $l_{i} \leq m_{i}/s-1$ such that
\begin{equation}\label{eqn.phole}
\frac{m_G(sT_{i+1} + sl_{i}T_{i+1} +(A-A))}{m_G( sl_{i}T_{i+1} +(A-A))} \leq \frac{m_G(r_{i+1}(A-A))}{m_G(A-A)}^{\frac{s}{m_i}}.
\end{equation}
Set $S_{i+1}:= sl_{i}T_{i+1} +(A-A)$ which is a symmetric set containing the identity since both $T_{i+1}$ and $A-A$ are.  Since $0_G \in T_{i+1}$ and $l_i \leq m_i/s-1$ we have
\begin{equation*}
A-A \subset S_{i+1} \subset m_iT_{i+1}+(A-A) \subset r_{i+1}(A-A)
\end{equation*}
which gives (\ref{dddd}). Moreover, from (\ref{eqn.phole}) we have
\begin{align*}
L_{i+1}=\frac{m_G(T_{i+1}+S_{i+1})}{m_G(S_{i+1})} & \leq \frac{m_G(r_{i+1}(A-A))}{m_G(A-A)}^{\frac{s}{m_i}}\\
& \leq K_{i+1}^{\frac{s}{m_i}}\\ & \leq (2K_{i+1})^{\frac{s}{m_i}}\\ & \leq \exp\left(\sqrt{\log 2L_i}\right)\\ & \leq \exp\left(\sqrt{4\log^{2^{-i}}2K}\right) \leq \exp(4\log 2^{-(i+1)}2K),
\end{align*}
so that (\ref{ffff}) holds.

Let $i \geq 1$ be maximal such that $2r_{i-1}+1 \leq r$ (possible since $r \geq 3=2r_0+1$, so that
\begin{equation*}
\sum_{j=0}^{i}{r_i^3}=O(r^3) \text{ and } 2^{-i} = O(r^{-1}),
\end{equation*}
and put $T:=T_i$.  The result follows since
\begin{equation*}
m_{i-1}T \subset S_{i-1}+A-A-S_{i-1} \subset (2r_{i-1}+1)(A-A) \subset r(A-A),
\end{equation*}
and $m_G(T_0) \geq m_G(A)$.
\end{proof}
\begin{proof}[Proof of Lemma \ref{lem.fsurvey}]
Let $3 \leq r =O(\log 2 \log 2K)$ be such that $\log^{O(r^{-1})} 2K = O(1)$ and apply Proposition \ref{prop.kony} to the set $A$ with the parameter $s$ to be optimised shortly.  We get a natural $m=\Omega(rs\log 2K)$ and a symmetric set $S$ containing the identity such that
\begin{equation*}
mX \subset r(A-A) \text{ and } m_G(X) \geq \exp(-O(s^2r^3\log^32K))m_G(A).
\end{equation*}
Let $k:=m^3$.  By Pl{\"u}nnecke's inequality we have
\begin{align*}
m_G((3k+1)X) & \leq m_G(3(m^2+1)mX)\\
& \leq m_G(3(m^2+1)r(A-A))\\
& \leq K^{3m^2r}\exp(O(s^2r^3\log^32K))m_G(X) \leq \exp(O(k/s))m_G(X).
\end{align*}
For $s=O(1)$ sufficiently large the right hand side is strictly less than $2^k$ (since $X$ is non-empty) and hence we can apply Corollary \ref{cor.useful} to see that $X$ has relative polynomial growth of order $O((\log (2\log 2K))^3\log^32K)$.  The result is proved.
\end{proof}

\subsection{From relative polynomial growth to Bohr sets of bounded dimension}\label{sec.43}
The next proposition is routine with the core of the argument coming from \cite{greruz::0}.
\begin{lemma*}[Lemma \ref{lem.rpgb}]
Suppose that $X$ is a symmetric non-empty set with relative polynomial growth of order $d\geq 1$.  Then there is a Bohr system $B$ with
\begin{equation*}
\dim B=O(d) \text{ and } m_G(B_1) = d^{O(d)}m_G(X),
\end{equation*}
such that $X-X \subset B_1$.
\end{lemma*}
\begin{proof}
Let $m=O(d\log 2d)$ be a natural number such that $m^{\frac{d}{m-1}}\leq \frac{3}{2}$.  Since $X$ has relative polynomial growth of order $d$ we see by the pigeonhole principle that there is some $2 \leq l \leq m$ such that
\begin{equation*}
\frac{m_G(lX)}{m_G((l-1)X)} \leq \left(\frac{m_G(mX)}{m_G(X)}\right)^{\frac{1}{m-1}} \leq m^{\frac{d}{m-1}} \leq \frac{3}{2}.
\end{equation*}
Let $\epsilon:=1/2^{18}d^2$ (the reason for which choice will become clear later) and write
\begin{equation*}
\Gamma:=\{\gamma \in \wh{G}: |\wh{1_{lX}}(\gamma)|>(1-\epsilon)m_G((l+1)X)\}
\end{equation*}
so that by Lemma \ref{lem.tvspec} (applicable since $l \geq 2$) we have that
\begin{equation*}
\Gamma \subset N(X-X,2\sqrt{3\epsilon}).
\end{equation*}
Let $\delta:\Gamma \rightarrow \R_{>0}$ be the constant function taking the value $2^{-4}$ and $B'$ be the Bohr system with frequency set $\Gamma$ and width function $\delta$.  By the first part of Lemma \ref{lem.bsann} we see that
\begin{align}
\nonumber X-X & \subset \Bohr\left(N(X-X,2\sqrt{3\epsilon}),1_{N(X-X,2\sqrt{3\epsilon})}\frac{\sqrt{3\epsilon}}{2}\right)\\
\label{eqn.nsd} & \subset \Bohr\left(\Gamma, \sqrt{\frac{3\epsilon}{2}1_\Gamma}\right) \subset B'_{1/2^5d}.
\end{align}

We now show that this Bohr system is not too large.  Let $k \in \N$ be a natural number to be optimised shortly.  Begin by noting that
\begin{equation}\label{eqn.u}
\int{\left(1_{lX} ^{(k)}\right)^2dm_G} \geq \frac{1}{m_G(k(lX))}\left(\int{1_{lX}^{(k)}dm_G}\right)^2 \geq \frac{m_G(lX)^{2k-1}}{(kl)^d},
\end{equation}
where $1_{lX}^{(k)}$ denotes the $k$-fold convolution of $1_{lX}$ with itself, and the inequalities are Cauchy-Schwarz and then the relative polynomial growth hypothesis.  On the other hand, by Parseval's theorem
\begin{eqnarray*}
\sum_{\gamma \not \in \Gamma}{|\wh{1_{lX}}(\gamma)|^{2k}} &\leq& ((1-\epsilon)m_G(lX))^{2k-2}\sum_{\gamma \in \wh{G}}{|\wh{1_{lX}}(\gamma)|^2}\\ & \leq & \exp(-\Omega(kd^{-2}))m_G(lX)^{2k-1}\leq \frac{m_G(lX)^{2k-1}}{2(kl)^d}
\end{eqnarray*}
for some natural $k=O(d^3\log d)$.  In particular, from (\ref{eqn.u}) we have that
\begin{equation*}
\sum_{\gamma \not \in \Gamma}{|\wh{1_{lX}}(\gamma)|^{2k}} \leq \frac{1}{2}\int{\left(1_{lX} ^{(k)}\right)^2dm_G}.
\end{equation*}
It then follows from Parseval's theorem and the triangle inequality that
\begin{eqnarray*}
\sum_{\gamma \in \Gamma}{|\wh{1_{lX}}(\gamma)|^{2k}}& = &\sum_{\gamma \in \wh{G}}{|\wh{1_{lX}}(\gamma)|^{2k}} - \sum_{\gamma \not \in \Gamma}{|\wh{1_{lX}}(\gamma)|^{2k}}\\ & \geq & \int{\left(1_{lX} ^{(k)}\right)^2dm_G}-\frac{1}{2}\int{\left(1_{lX} ^{(k)}\right)^2dm_G} = \frac{1}{2}\int{\left(1_{lX} ^{(k)}\right)^2dm_G}.
\end{eqnarray*}
Write $\beta$ for the uniform probability measure induced on $B'_1$.  By the second part of Lemma \ref{lem.bsann} and the nesting of approximate annihilators we see that
\begin{equation*}
\Gamma \subset N\left(B_1',2\pi \|\delta\|_{\ell_\infty(\Gamma)}\right) \subset N\left(B_1',\frac{2\pi}{2^4}\right) \subset N\left(B_1',\frac{1}{2}\right).
\end{equation*}
Thus by the triangle inequality, if $\gamma \in \Gamma$ then
\begin{equation*}
|1-\wh{\beta}(\gamma)| = \left|\int{(1-\gamma(x))d\beta(x)}\right| \leq \int{|1-\gamma(x)|d\beta(x)} \leq \frac{1}{2},
\end{equation*}
and hence $|\wh{\beta}(\gamma)| \geq \frac{1}{2}$.  We conclude that
\begin{equation*}
\sum_{\gamma \in \wh{G}}{|\wh{1_{lX}}(\gamma)|^{2k}|\wh{\beta}(\gamma)|^2}\geq \frac{1}{4}\sum_{\gamma \in \Gamma}{|\wh{1_{lX}}(\gamma)|^{2k}} \geq \frac{m_G(lX)^{2k-1}}{8(kl)^d}.
\end{equation*}
But, by Parseval's theorem and H{\"o}lder's inequality we have that
\begin{align*}
\sum_{\gamma\in \wh{G}}{|\wh{1_{lX}}(\gamma)|^{2k}|\wh{\beta}(\gamma)|^2}
& = \int{\left(1_{lX} ^{(k)}\ast \beta\right)^2dm_G}\\ 
& = \int{1_{lX}^{(k)} \ast 1_{-lX}^{(k)} d\beta \ast \tilde{\beta}}\\
& = m_G(B_1')^{-1} \int{1_{lX}^{(k)} \ast 1_{-lX}^{(k)}1_{B_1'} \ast \tilde{\beta} dm_G}\\
&\leq m_G(B_1')^{-1} \|1_{lX}^{(k)}\ast 1_{-lX}^{(k)}\|_{L_1(m_G)}\left\|1_{B_1'} \ast \tilde{\beta}\right\|_{L_\infty(G)} = \frac{m_G(lX)^{2k}}{m_G(B'_1)},
\end{align*}
and so
\begin{equation}\label{eqn.sui}
m_G(B'_1) \leq 8(kl)^dm_G(lX) \leq \exp(O(d\log 2d))m_G(X).
\end{equation}
Now, note by sub-additivity and symmetry of Bohr sets and Ruzsa's Covering Lemma (Lemma \ref{lem.rc}) that for $i \geq 1$ we have
\begin{align*}
\mathcal{C}\left(B'_{2^{-i}};B'_{2^{-(i+3)}}\right) \leq \mathcal{C}\left(B'_{2^{-i}};B'_{2^{-(i+4)}}-B'_{2^{-(i+4)}}\right) & \leq \frac{m_G\left(B'_{2^{-i}} + B'_{2^{-(i+4)}}\right)}{m_G\left(B'_{2^{-(i+4)}}\right)}\\
& \leq  \frac{m_G\left(B'_{2^{-(i-1)}}\right)}{m_G\left(B'_{2^{-(i+4)}}\right)}.
\end{align*}
Let $J:=\left\lfloor \frac{\log_2 d}{5}\right\rfloor$ so that
\begin{align*}
\prod_{j=0}^J{\mathcal{C}\left(B'_{2^{-(5j+1)}};B'_{2^{-(5j+4)}}\right) } &\leq \prod_{j=0}^J{ \frac{m_G\left(B'_{2^{-5j}}\right)}{m_G\left(B'_{2^{-5(j+1)}}\right)}}\\ & \leq \frac{m_G(B'_1)}{m_G(B'_{2^{-5(J+1)}})} \leq \frac{m_G(B'_1)}{m_G(B'_{1/2^5d})} \leq \frac{m_G(B'_1)}{m_G(X-X)},
\end{align*}
where the last inequality is from (\ref{eqn.nsd}).

By averaging there is some $0 \leq j \leq J$ such that
\begin{equation*}
\mathcal{C}\left(B'_{2^{-(5j+1)}};B'_{2^{-(5j+4)}}\right) \leq \left(\frac{m_G(B'_1)}{m_G(X-X)}\right)^{\frac{1}{J}} = \exp(O(d)),
\end{equation*}
where the last inequality is from (\ref{eqn.sui}).

Set $B:=2^{-(5j+1)}B'$ and apply Lemma \ref{lem.grow} (possible since $w(B) \leq 2^{-5}<\frac{1}{4}$) to see that $\dim^* B = O(d)$.  It follows by the second inequality in Lemma \ref{lem.smi} part (\ref{lem.smi.3}) that $\dim B = O(d)$.  Moreover, nesting of Bohr sets tells us that $X-X \subset B_1$ and
\begin{equation*}
m_G(B_1) \leq m_G(B_1') \leq \exp(O(d\log 2d)).
\end{equation*}
The result is proved.
\end{proof}

\subsection{Bogolio{\`u}boff-Chang}\label{sec.44}

In the paper \cite{bog::} Bogolio{\`u}boff showed how to find Bohr sets inside four-fold sumsets.  The importance of this was emphasised by Ruzsa in \cite{ruz::9} and refined by Chang in \cite{cha::0}.  We shall need the following result in our work.
\begin{lemma*}[Lemma \ref{lem.bc}]
Suppose that $A$ is a non-empty set, $B$ is a Bohr system and $\mu$ is a $B$-approximately invariant probability measure, $S \subset B_1$ has $\mu(S)>0$, and $L$, non-empty, is such that $\|1_L \ast \mu_S\|_{L_2(m_G)}^2 \geq \epsilon m_G(L)$.  Then there is a Bohr system $B' \leq B$ with
\begin{equation*}
\mathcal{C}^\Delta(A;B_1') \leq (2\epsilon^{-1})^{O(\epsilon^{-2}\log 2
\mu(S)^{-1})}\mathcal{C}^\Delta(A;B_1)
\end{equation*}
and
\begin{equation*}
\dim B' = \dim B+ O( \epsilon^{-2}\log 2\mu(S)^{-1})
\end{equation*}
such that $B'_{1} \subset L-L + S-S$.
\end{lemma*}
\begin{proof}
Since $\mu$ is $B$-approximately invariant and $\wt{\mu}$ is a probability measure, Lemma \ref{lem.closbm} tells us that $\mu \ast \wt{\mu}$ is $B$-approximately invariant.  By Parseval's theorem we have
\begin{equation*}
\|1_L \ast 1_{-L}\|_{A(G)} = \sum_{\gamma \in \wh{G}}{|\wh{1_L}(\gamma)|^2} = \int{1_L^2dm_G} = m_G(L).
\end{equation*}

Apply Lemma \ref{lem.chang} to $B$, $\mu \ast\wt{\mu}$, and $1_L \ast 1_{-L}$ with parameters $p\geq 2$ and $\eta \in (0,1]$ to be optimised later.  This gives us a Bohr system $B'$ with
\begin{equation*}
\mathcal{C}^\Delta(A;B_1') \leq (2\eta^{-1})^{O(p\eta^{-2})}\mathcal{C}^\Delta(A;B_1)\text{ and }\dim B' \leq \dim B + O(p\eta^{-2})
\end{equation*}
such that
\begin{equation*}
\|\tau_x(1_L \ast 1_{-L}) - 1_L \ast 1_{-L}\|_{L_p(\mu \ast \wt{\mu})} \leq \eta m_G(L) \text{ for all }x \in B'_{1}.
\end{equation*}
Since $\mu$ is non-negative we have
\begin{equation*}
0 \leq \mu_S \ast \wt{\mu_S} \leq \mu(S)^{-2}\mu \ast \wt{\mu},
\end{equation*}
and so there is a function $f$ with $0 \leq f \leq \mu(S)^{-2}$ point-wise such that
\begin{equation*}
\int{gd\mu_S \ast \wt{\mu_S}} = \int{gfd\mu\ast \wt{\mu}} \text{ for all }g \in L_1(\mu_S \ast \wt{\mu_S}).
\end{equation*}
($f$ is the Radon-Nikodym derivative of $\mu_S \ast \wt{\mu_S}$ with respect to $\mu \ast \wt{\mu}$.)

Write $p'$ for the conjugate index of $p$ (so $\frac{1}{p}+\frac{1}{p'}=1$) we have
\begin{equation*}
\|f\|_{L_{p'}(\mu \ast \wt{\mu})} \leq \left(\int{\mu(S)^{-2(p'-1)}fd\mu\ast\wt{\mu}}\right)^{1/p'} =\mu(S)^{-2/p}.
\end{equation*}
If we take $p=2+2\log \mu(S)^{-1}$ then we see from H{\"o}lder's inequality that for all $x \in B_1'$ we have
\begin{align*}
&\left|\langle 1_L \ast 1_{-L},\mu_S \ast \wt{\mu_S}\rangle - \langle \tau_x(1_L \ast 1_{-L}),\mu_S \ast \wt{\mu_S}\rangle\right|\\
& = \left|\langle 1_L \ast 1_{-L},f\rangle_{L_2(\mu \ast \wt{\mu})} - \langle \tau_x(1_L \ast 1_{-L}),f\rangle_{L_2(\mu \ast \wt{\mu})} \right|\\
&= \left|\langle 1_L \ast 1_{-L} - \tau_x(1_L \ast 1_{-L}),f\rangle_{L_2(\mu \ast \wt{\mu})}\right|\\
& \leq \left\|1_L \ast 1_{-L} - \tau_x(1_L \ast 1_{-L})\right\|_{L_p(\mu \ast \wt{\mu})}\|f\|_{L_{p'}(\mu \ast \wt{\mu})} \leq
e\eta m_G(L).
\end{align*}
By hypothesis 
\begin{equation*}
\langle 1_L \ast 1_{-L},\mu_S \ast \wt{\mu_S}\rangle =\|1_L \ast \mu_S\|_{L_2(m_G)}^2 \geq \epsilon m_G(L);
\end{equation*}
it follows that for $\eta=\frac{1}{2e}\epsilon$ we have
\begin{equation*}
 \langle \tau_x(1_L \ast 1_{-L}),\mu_S \ast \wt{\mu_S}\rangle \geq \frac{\epsilon}{2}m_G(L) \text{ for all }x \in B_1'.
\end{equation*}
However, the left hand side is $0$ if $x+L-L \cap S-S =\emptyset$ \emph{i.e.} if $x \not \in L-L+S-S$. The result is proved.
\end{proof}

\section{Arithmetic connectivity}\label{sec.ac}

The basic approach of our main argument (captured in Lemma \ref{lem.mitlem}) is iterative and to make this work we need to consider not just integer-valued functions, but \emph{almost} integer-valued functions.  For $\epsilon \in (0,1/2)$ we say that $f:G \rightarrow \C$ is \textbf{$\epsilon$-almost integer-valued} if there is a function $f_\Z:G \rightarrow \Z$ such that
\begin{equation*}
\|f-f_\Z\|_{L_\infty(G)} \leq \epsilon.
\end{equation*}
Since $\epsilon<1/2$ this actually means that $f_\Z$ is uniquely defined. 

When a function $f$ has small algebra norm and is close to integer-valued, it turns out that $f_{\Z}$ has a lot of additive structure.  This is captured by a concept called arithmetic connectivity identified by Green in \cite[Definition 6.4]{gresan::0}.  We shall need a slight refinement of this: for $m,l \geq 2$ we say that a set $A\subset G$ is \textbf{$(m,l)$-arithmetically connected} if for every $x \in A^m$ there is some $\sigma \in \Z^m$ with $\|\sigma\|_{\ell_1^m} \leq l$ and $|\sigma_i|=1$ for at least two $i$s such that
\begin{equation*}
\sigma\cdot x := \sum_i{\sigma_ix_i} \in A.
\end{equation*}
The definition is perhaps a little odd. To help we present some simple examples we leave as exercises.
\begin{enumerate}
\item $A$ is $(m,1)$-arithmetically connected for some $m$ if and only if $A=\emptyset$.  (Of course this is not a significant example and can easily be removed by simply restricting to $m,l \geq 2$.)
\item If every element of $A$ has order $2$ then $A$ is $(m,m+k)$-arithmetically connected for some $k \geq 0$ if and only if it is $(m,m)$-arithmetically connected.
\item If $A$ is a subgroup then $x+y \in A$ for all $x,y \in A$ and so $A$ is $(2,2)$-arithmetically connected. On the other hand, if $G=\Z$ and $A=\N$ then $A$ is also $(2,2)$-arithmetically connected (for the same reason) but not `close' to any subgroup.
\item If $A$ is a union of $k$ cosets (of possibly different subgroups) then by the pigeonhole principle for any vector $x \in A^{2k+1}$ there are indices $i<j<k$ such that $x_i,x_j,x_k$ are all in the same coset.  It follows that $x_i+x_j-x_k$ is in that same coset and hence in $A$.  We conclude that $A$ is $(2k+1,3)$-arithmetically connected.
\end{enumerate}
Arithmetic connectivity is related to additive structure by the following easy adaptation of \cite[Proposition 6.5]{gresan::0}.
\begin{lemma}\label{lem.s}
Suppose that $A$ is $(m,l)$-arithmetically connected (for $m,l \geq 2$).  Then
\begin{equation*}
\|1_A \ast 1_A\|_{L_2(m_G)}^2 \geq m^{-O(l)}m_G(A)^3.
\end{equation*}
\end{lemma}
\begin{proof}
First we count the number of $\sigma \in \Z^m$ such that $\|\sigma\|_{\ell_1^m} \leq l$.  The number of ways of writing a total of $r$ as a sum of $m$ non-negative integers is $\binom{r+m}{m}$.  For each such $\sigma$ we can choose the signs of the various integers in at most $2^l$ ways (since at most $l$ of them are non-zero) and so the total number of $\sigma \in \Z^m$ with $\|\sigma\|_{\ell_1^m} \leq l$ is at most
\begin{equation*}
\sum_{r=0}^l{\binom{r+m}{m}2^l} \leq l\binom{m+l}{m}2^l = l \binom{m+l}{l}2^l \leq l\left(\frac{2e(m+l)}{l}\right)^l= m^{O(l)}.
\end{equation*}
It follows that there is such a $\sigma \in \Z^m$ such that for at least $m^{-O(l)}|A|^m$ vectors $x \in A^m$ we have $\sigma \cdot x \in A$.  Rewriting this we have
\begin{align*}
m_G(A)^mm^{-O(l)} & \leq \int{1_A\left(\sum_{i=1}^m{\sigma_ix_i} \right)\prod_{i=1}^m{1_A(x_i)dm_G(x_i)}}\\ & = \sum_{\gamma}{\wh{1_A}(\gamma)\prod_{i=1}^m{\wh{1_A}(-\sigma_i \cdot \gamma)}}.
\end{align*}
Since $|\sigma_i|=1$ for at least two $i \in [m]$, $|\wh{1_A}(\gamma)|=|\wh{1_A}(-\gamma)|$, and $|\wh{1_A}(-\sigma_i\cdot \gamma)| \leq m_G(A)$ we conclude that
\begin{equation*}
m_G(A)^{m-2} \sum_{\gamma}{|\wh{1_A}(\gamma)|^3} \geq m_G(A)^mm^{-O(l)}.
\end{equation*}
The result now follows from Cauchy-Schwarz and Parseval's theorem which gives
\begin{align*}
\sum_{\gamma}{|\wh{1_A}(\gamma)|^3} & \leq \left(\sum_{\gamma}{|\wh{1_A}(\gamma)|^4}\right)^{1/2}\left(\sum_{\gamma}{|\wh{1_A}(\gamma)|^2}\right)^{1/2}\\ & =  \left(\sum_{\gamma}{|\wh{1_A}(\gamma)|^4}\right)^{1/2} \left(\int{1_A(x)^2dm_G(x)}\right)^{1/2} =  \left(\sum_{\gamma}{|\wh{1_A}(\gamma)|^4}\right)^{1/2} m_G(A)^{1/2}.
\end{align*}
\end{proof}

On the other hand additive connectivity is related to small algebra norm via the following result.
\begin{proposition}\label{prop.ac}
There is an absolute constant $C_{\textsc{M{\'e}l}}>0$ such that the following holds.  Suppose that $g\in A(G)$ is $\epsilon$-almost integer-valued for some $\epsilon \in (0,1/2)$ and has $\|g\|_{A(G)} \leq M$ for some $M \geq 1$.  Then provided $\epsilon \leq \exp(-C_{\textsc{M{\'e}l}} M)$, the set $\supp g_\Z$ is $(O(M^3),O(M))$-arithmetically connected.
\end{proposition}
The proof of this owes a lot to \cite[Lemme 1]{mel::} of M{\'e}la, and we are grateful to Ben Green for directing us to that paper.  Indeed, as noted in \cite[\S9]{gresan::0} an example in M{\'e}la's paper shows that one cannot hope to weaken the requirement that $\epsilon \leq \exp(-CM))$ to anything with $C$ below a certain absolute threshold.  One can also make use of the auxiliary measures \cite[Lemme 4]{mel::} constructed in M{\'e}la's paper to show that $\supp g_\Z$ is $(O(M^2\log 2M),O(M\log 2M))$-arithmetically connected but for us this extra logarithm in the second parameter is worse than the benefit of a power saving in the first when we apply Lemma \ref{lem.s}.

We write $T_n(x)$ for the Chebychev polynomial of degree $n$.  Recall (from, for example, \cite[\S6.10.6]{zwikraros::}) that we have a formula for $T_n$:
\begin{equation*}
T_n(x)=\frac{n}{2}\sum_{r=0}^{\lfloor n/2\rfloor}{\frac{(-1)^r}{n-r}\binom{n-r}{r}(2x)^{n-2r}} = \cos(n\arccos x) \text{ for }|x|\leq 1;
\end{equation*}
the last form tells us immediately that $\|T_n\|_{L_\infty([-1,1])}\leq 1$.

We shall be particularly interested in the Chebyshev polynomials of odd degree.  Indeed, note from the above formula that if $n=2l+1$ for some non-negative integer $l$, then only the coefficients of odd powers of $x$ are non-zero and
\begin{equation*}
T_{2l+1}(x) = \sum_{j=0}^l{c(j,l)x^{2j+1}},
\end{equation*}
where
\begin{equation*}
c(j,l) = 2^{2j}(-1)^{l-j}\frac{2l+1}{2j+1}\binom{l+j}{l-j} = 2^{2j}(-1)^{l-j}\frac{2l+1}{2j+1}\binom{l+j}{2j} .
\end{equation*}
In view of this we have
\begin{equation}\label{eqn.sss}
|c(0,l)| = 2l+1 \textrm{ and } |c(j,l)| = (O(l/j))^{2j+1}.
\end{equation}
Added to this information we shall need the following lemma.
\begin{lemma}\label{lem.key}
Suppose that $m \in \N$, and $l \in \N_0$ are parameters, $g:G \rightarrow \C$ has support $A$ and $x \in G^m$ is such that if $\sigma \in \Z^m$ has $\|\sigma\|_{\ell_1^m}\leq 2l+1$ and $\sigma \cdot x \in A$ then $|\sigma_i|=1$ for at most one value of $i$.  Then for every $\omega \in \ell_\infty^m$ with $\|\omega\|_{\ell_\infty^m} \leq 1$ and $0 \leq r \leq l$ we have
\begin{equation*}
\left|\sum_{\gamma}{\left(\Re \sum_{i=1}^m{\omega_i\gamma(x_i)}\right)^{2r+1}\overline{\wh{g}(\gamma)}}\right| =\exp(O(r+1))(r+1)^rm^{r+1}\|g\|_{L_\infty(G)}.
\end{equation*}
\end{lemma}
\begin{proof}
We write $\mathcal{C}$ for the conjugation operator and note that by Fourier inversion we have
\begin{align*}
&\sum_{\gamma}{\left(\Re \sum_{i=1}^m{\omega_i\gamma(x_i)}\right)^{2r+1}\mathcal{C}(\wh{g}(\gamma))}\\
& \qquad = \sum_{\gamma}{\left( \sum_{i=1}^m{\frac{1}{2}\left(\omega_i\gamma(x_i) + \mathcal{C}(\omega_i)\gamma(-x_i)\right)}\right)^{2r+1}\mathcal{C}(\wh{g}(\gamma))}\\
& \qquad  =\frac{1}{2^{2r+1}}\sum_{\substack{\pi:[2r+1]\rightarrow [m]\\ \iota:[2r+1]\rightarrow \{0,1\}}}{\sum_{\gamma}{\mathcal{C}(\wh{g})(\gamma)\gamma\left(\sum_{i=1}^{2r+1}{(-1)^{\iota_i}x_{\pi_i}}\right)\prod_{i=1}^{2r+1}{\mathcal{C}^{\iota_i}(\omega_{\pi_i})}}}\\
& \qquad  =\frac{1}{2^{2r+1}}\sum_{\substack{\pi:[2r+1]\rightarrow [m]\\ \iota:[2r+1]\rightarrow \{0,1\}}}{\mathcal{C}(g)\left(-\sum_{i=1}^{2r+1}{(-1)^{\iota_i}x_{\pi_i}}\right)\prod_{i=1}^{2r+1}{\mathcal{C}^{\iota_i}(\omega_{\pi_i})}}.
\end{align*}
Applying the triangle inequality we see that
\begin{equation}\label{eqn.long}
\left|\sum_{\gamma}{\left(\Re \sum_{i=1}^m{\omega_i\gamma(x_i)}\right)^{2r+1}\overline{\wh{g}(\gamma)}}\right| \leq \frac{1}{2^{2r+1}}\sum_{\substack{\pi:[2r+1]\rightarrow [m]\\ \iota:[2r+1]\rightarrow \{0,1\}}}{\|g\|_{L_\infty(G)}1_{A}\left(-\sum_{i}{(-1)^{\iota_i}x_{\pi_i}}\right)}.
\end{equation}
Given $\pi:[2r+1]\rightarrow [m]$ and $\iota:[2r+1]\rightarrow \{0,1\}$ we define $\sigma(\pi,\iota) \in \Z^m$ by
\begin{equation*}
\sigma_j(\pi,\iota):=-\sum_{i: \pi_i=j}{(-1)^{\iota_i}}.
\end{equation*}
By the triangle inequality we have
\begin{equation*}
\|\sigma(\pi,\iota)\|_{\ell_1^m} =\sum_{j=1}^m{|\sigma_j|} \leq \sum_{j=1}^m{\sum_{i:\pi_i=j}{1}} = 2r+1 \leq 2l+1.
\end{equation*}
Moreover,
\begin{equation*}
\sigma(\pi,\iota) \cdot x = \sum_{j=1}^m{\sigma_j(\pi,\iota)x_j} = - \sum_{j=1}^m{x_j\sum_{i: \pi_i=j}{(-1)^{\iota_i}}} = -\sum_{i=1}^{2r+1}{(-1)^{\iota_i}x_{\pi_i}},
\end{equation*}
and so $1_A(\sigma(\pi,\iota)\cdot x)=0$ unless $|\sigma_j(\pi,\iota)|=1$ for at most one $j \in [m]$.  It remains to bound from above the number of functions $\pi:[2r+1]\rightarrow [m]$ and $\iota:[2r+1]\rightarrow \{0,1\}$ such that $|\sigma_j(\pi,\iota)|=1$ for at most one $j \in [m]$.  Since $|\sigma_j(\pi,\iota)|=1$ for at most one $j$ it follows that the image of $\pi$ has size at most $r+1$, and hence the number of pairs $(\pi,\iota)$ is at mosst
\begin{equation*}
\binom{m}{r+1}\cdot (r+1)^{2r+1}\cdot 2^{2r+1} = \exp(O(r+1))(r+1)^rm^{r+1}.
\end{equation*}
Inserting this into (\ref{eqn.long}) gives the result.
\end{proof}

\begin{proof}[Proof of Proposition \ref{prop.ac}]
Let $A:=\supp g_\Z$, and take $l$ and $m$ to be parameters to be chosen later.  Suppose that $A$ is not $(m,2l+1)$-arithmetically connected, so that there is some $x \in A^m$ such that for all $\sigma \in \Z^m$ with $\|\sigma\|_{\ell_1^m} \leq 2l+1$ and $|\sigma_i|=1$ for at least two $i \in [m]$, we have $g_\Z(\sigma\cdot x)=0$. 

Our first task is to define $\omega \in \ell_\infty^m$.  With $\omega$ appropriately defined we shall put
\begin{equation*}
h:=\frac{|G|}{m}\sum_{j=1}^m{\frac{1}{2}\left(\omega_j 1_{\{x_j\}}+ \overline{\omega_j}1_{\{-x_j\}}\right)},
\end{equation*}
so that
\begin{equation*}
\|h\|_{L_1(m_G)} \leq 1 \text{ and } \wh{h}(\gamma)=\frac{1}{m}\Re \sum_{j=1}^m{\omega_j \gamma(x_j)}.
\end{equation*}

The function $g_{\Z}$ is real and since $x_j \in A$ we see that $|g_{\Z}(x_j)| \geq 1$ for all $j \in [m]$.  It follows that
\begin{enumerate}
\item either at least $1/3$ of the indices $j \in [m]$ have $g_{\Z}(-x_j)=0$, in which case we set $\omega_j=\sgn g_{\Z}(x_j)$ for all these indices and $\omega_j=0$ for all others, and get
\begin{equation*}
\sum_{j=1}^m{\frac{1}{2}\left(\omega_jg_{\Z}(x_j) + \overline{\omega_j}g_{\Z}(-x_j)\right)} \geq \frac{m}{6};
\end{equation*}
\item or at least $1/3$ of the indices $j \in [m]$ have $\sgn g_{\Z}(x_j)=\sgn g_{\Z}(-x_j)$, in which case we set $\omega_j=\sgn g_{\Z}(x_j)$ for all these indices and $\omega_j=0$ for all others and get
\begin{equation*}
\sum_{j=1}^m{\frac{1}{2}\left(\omega_jg_{\Z}(x_j) + \overline{\omega_j}g_{\Z}(-x_j)\right)} \geq \frac{m}{3};
\end{equation*}
\item or at least $1/3$ of the indices $j \in [m]$ have $\sgn g_{\Z}(x_j)=-\sgn g_{\Z}(-x_j)$, in which case we set $\omega_j=i$ for all these indices and $\omega_j=0$ for all others
and get
\begin{equation*}
\left|\sum_{j=1}^m{\frac{1}{2}\left(\omega_jg_{\Z}(x_j) + \overline{\omega_j}g_{\Z}(-x_j)\right)}\right|=\left|\sum_{j=1}^m{\frac{1}{2}\left(g_{\Z}(x_j) -g_{\Z}(-x_j)\right)}\right|  \geq \frac{m}{3}.
\end{equation*}
\end{enumerate}
By construction $\|\omega\|_{\ell_\infty^m} \leq 1$ and
\begin{equation*}
\left|\left\langle \wh{h},\wh{g_{\Z}} \right\rangle_{\ell_2(\wh{G})} \right| = \left|\sum_{j=1}^m{\frac{1}{2}\left(\omega_jg_{\Z}(x_j) + \overline{\omega_j}g_{\Z}(-x_j)\right)}\right| \geq \frac{1}{6}.
\end{equation*}
By Lemma \ref{lem.key} for every $1 \leq r \leq l$ we have
\begin{align*}
\left|\langle \wh{h}^{2r+1},\wh{g_{\Z}}\rangle_{\ell_2(\wh{G})}\right| & = \left|\sum_{\gamma}{\left(\Re \sum_{i=1}^m{\omega_i\gamma(x_i)}\right)^{2r+1}\overline{\wh{g_\Z}(\gamma)}}\right|\\
 & =\exp(O(r+1))(r+1)^rm^{r+1}\|g_{\Z}\|_{L_\infty(G)}\\
& = O(r)^rm^{r+1}(\|g\|_{L_\infty(G)}+\epsilon)=O(r)^rm^{r+1}M.
\end{align*}
On the other hand, by Young's inequality $\|h^{(2r+1)}\|_{L_1(m_G)} \leq 1$ and so by Plancherel's theorem we see that
\begin{align*}
\left|\langle \wh{h}^{2r+1},\wh{g_\Z}\rangle_{\ell_2(\wh{G})} - \langle \wh{h}^{2r+1},\wh{g}\rangle_{\ell_2(\wh{G})}\right| & = \left|\langle (h^{(2r+1)})^\wedge,\wh{g_\Z}\rangle_{\ell_2(\wh{G})} - \langle (h^{(2r+1)})^\wedge,\wh{g}\rangle_{\ell_2(\wh{G})}\right|\\
& = \left|\langle (h^{(2r+1)})^\wedge,(g_\Z-g)^\wedge\rangle_{\ell_2(\wh{G})}\right|\\ &  = \left|\langle h^{(2r+1)}, g_\Z-g\rangle_{L_2(m_G)}\right|  \leq \|g-g_{\Z}\|_{L_\infty(G)} \leq \epsilon
\end{align*}
for all $0 \leq r \leq l$.

Finally, $-1 \leq \wh{h}(\gamma) \leq 1$, and so $|T_{2l+1}(\wh{h})| \leq 1$ and hence by (\ref{eqn.sss}) we get
\begin{align*}
M &\geq \left|\langle T_{2l+1}(\wh{h}),\wh{g}\rangle_{\ell_2(\wh{G})} \right|\\ & \geq \left|\sum_{r=0}^l{c(r,l)\langle\wh{h}^{2r+1},\wh{g}\rangle_{\ell_2(\wh{G})}}\right|\\
& \geq |c(0,l)||\langle \wh{h},\wh{g}\rangle_{\ell_2(\wh{G})}| - \sum_{r=1}^l{|c(r,l)||\langle \wh{h}^{2r+1},\wh{g}\rangle_{\ell_2(\wh{G})}|}\\
& \geq |c(0,l)||\langle \wh{h},\wh{g_{\Z}}\rangle_{\ell_2(\wh{G})}| - \epsilon\sum_{r=0}^l{|c(r,l)|} - \sum_{r=0}^l{|c(r,l)||\langle \wh{h}^{2r+1},\wh{g_{\Z}}\rangle_{\ell_2(\wh{G})}|}\\
& \geq (2l+1)\frac{1}{6} - \epsilon \sum_{r=1}^l{O\left(\frac{l}{r}\right)^{2r+1}} - M\sum_{r=1}^l{O\left(\frac{l}{r}\right)^{2r+1}O(r)^rm^{-r}}\\
& \geq \frac{l}{3} - \epsilon \exp(O(l)) -M\frac{l^3}{m} \exp(O(l^2/m)).
\end{align*}
It follows that if $\epsilon \leq \exp(-C_1l)$ for some sufficiently large $C_1>0$, $m = C_2l^3$ for some sufficiently large $C_2>0$ and $l=C_3M$ for some sufficiently large $C_3>0$ then we arrive at a contradiction, and we find that $A$ \emph{is} $(m,2l+1)$-arithmetically connected.
\end{proof}

\section{The main argument}\label{sec.main}

We shall prove the following of which Theorem \ref{thm.main} is a special case arising from taking $\delta:=1$ and $\epsilon:=\exp(-C_{\textsc{M{\'e}l}}'M)$.
\begin{theorem}\label{thm.main2}
There is an absolute constant $C_{\textsc{M{\'e}l}}'>0$ such that if $M \geq1$ and $\epsilon,\delta \in (0,1]$ are such that $\epsilon \leq \delta\exp(-C_{\textsc{M{\'e}l}}'M)$, and $f:G \rightarrow \Z$ is $\epsilon$-almost integer-valued with $\|f\|_{A(G)}\leq M$, then there is some non-negative integer $l \leq M(1+\delta)$, subgroups $H_1,\dots,H_l \leq G$, and functions $z_1:G/H_1 \rightarrow \Z,\dots, z_l:G/H_l \rightarrow \Z$ such that
\begin{equation*}
f_\Z=\sum_{i=1}^l{\sum_{W \in G/H_i}{z_i(W)1_W}}
\end{equation*}
and
\begin{equation*}
 \|z_i\|_{\ell_1(G/H_i)} \leq  \exp(O(M^{4}\log^82M + M^3\log \delta^{-1}(\log (2\log 2\delta^{-1})))) \text{ for }1 \leq i \leq l.
\end{equation*}
\end{theorem}
To do this we combine all our previous work  into our key iterative lemma.
\begin{lemma}\label{lem.mitlem}
Suppose that $f \in A(G)$ is $\epsilon$-almost integer-valued, $\|f\|_{A(G)} \leq M$ for some $M \geq 1$, $\supp f_{\Z}$ is non-empty and $\eta \in \left(0,\frac{1}{4}\right]$ a parameter.  Then provided we have $\epsilon \leq \min\{\exp(-C_{\textsc{M{\'e}l}}M),1/8\}$ there is a function $g$ that is $(\epsilon+\eta)$-almost integer-valued, a subgroup $H \leq G$, and a function $z:G/H \rightarrow \Z$ with
\begin{equation*}
\|z\|_{\ell_1(G/H)} \leq \exp(O(M^{4}\log^82M + M^3\log \eta^{-1}(\log (2\log 2\eta^{-1})))),
\end{equation*}
such that
\begin{equation*}
g_\Z=\sum_{W \in G/H}{z(W)1_{W}} \text{ and }\|f-g\|_{A(G)} \leq \|f\|_{A(G)} -1+(\epsilon +\eta).
\end{equation*}
\end{lemma}
\begin{proof}
Apply Proposition \ref{prop.ac} to $f$ to get that the set $A:=\supp f_{\Z}$ is $(O(M^3),O(M))$-arithmetically connected (provided $\epsilon$ is sufficiently small).  By Lemma \ref{lem.s} we see that
\begin{equation*}
\|1_A \ast 1_A\|_{L_2(m_G)}^2 = \exp(-O(M\log 2M))m_G(A)^3.
\end{equation*}
It follows from the Balog-Szemer{\'e}di-Gowers Theorem that there is a set $A' \subset A$ such that
\begin{equation*}
m_G(A') = \exp(-O(M\log 2M))m_G(A) \text{ and } m_G(A'+A') \leq  \exp(O(M\log 2M))m_G(A').
\end{equation*}
By Proposition \ref{prop.f} there is a Bohr system $B$ with
\begin{equation*}
\dim B = O(M^3\log^72M) \text{ and } \mathcal{C}^\Delta(A';B_1) = \exp(O(M^3\log^72M))
\end{equation*}
and a constant $\psi= \exp(-O(M \log^22M))$ such that
\begin{equation}\label{eqn.sssss}
\|1_{A'} \ast \beta\|_{L_\infty(G)} \geq \psi \text{ for all probability measures }\beta\text{ with } \supp \beta \subset B_1.
\end{equation}
Apply Proposition \ref{prop.screl} to the set $A'$, the Bohr system $B$, $d:=1+\dim B$, and the function $f$ with parameters
\begin{equation*}
\delta:=1/2^4M \text{ and } \kappa:= 1/2^5M,
\end{equation*}
and
\begin{align*}
p& :=\max\{ 100C_{\textsc{M{\'e}l}}M,1+\log_2 \psi^{-1},3+\log_3 M + \log_3\eta^{-1}\}\\
& = O(\max\{M\log^22M,\log \eta^{-1}\})
\end{align*}
to get a Bohr system $B' \leq B$ with
\begin{align*}
\mathcal{C}^\Delta(A';B_1') & \leq \exp(O(\delta^{-1}d\log 2\kappa^{-1}d + p\delta^{-3}\log^32p\kappa^{-1}\delta^{-1})) \mathcal{C}^\Delta(A';B_1)\\
& \leq \exp(O(M^{4}\log^82M + M^{3}\log \eta^{-1}(\log (2\log 2\eta^{-1}))))
\end{align*}
and
\begin{equation*}
\dim B' \leq d + O(p\delta^{-2}\log^22\delta^{-1}) = O(M^3\log^72M + M^2(\log^2 2M)\log \eta^{-1}),
\end{equation*}
and a $B'$-approximately invariant probability measure $\mu$ and a probability measure $\nu$ supported on $B'_\kappa$ such that
\begin{equation*}
\sup_{x \in G}{\|f - f\ast \mu\|_{L_p(\tau_x(\nu))}} \leq \delta M.
\end{equation*}
By the integral triangle inequality it follows that
\begin{equation*}
\sup_{x \in G}{\|f - f\ast \mu\|_{L_p(\tau_x(\nu\ast \widetilde{\nu}))}} \leq \delta M.
\end{equation*}
Since $\mu$ is $B'$-approximately invariant and $\kappa \leq 1/2$ it follows from Lemma \ref{lem.cc} that for all $y \in \supp \nu \ast \widetilde{\nu}$ we have
\begin{equation*}
|f\ast \mu(y+x) - f\ast \mu(x)| \leq 2\kappa \|f\|_{L_\infty(G)} \leq 2\kappa M,
\end{equation*}
and hence
\begin{equation*}
\sup_{x \in G}{\|f - f\ast \mu(x)\|_{L_p(\tau_x(\nu\ast \widetilde{\nu}))}} \leq \delta M +2\kappa M = (\delta+2\kappa)M.
\end{equation*}
By the triangle inequality we then have
\begin{equation}\label{eqn.o}
\sup_{x \in G}{\|f_{\Z} - f\ast \mu(x)\|_{L_p(\tau_x(\nu\ast \widetilde{\nu}))}} \leq (\delta+2\kappa) M+ \epsilon \leq \frac{1}{4},
\end{equation}
given the choices of $\delta$ and $\kappa$, and the upper bound on $\epsilon$.  We put $k:=(f\ast \mu)_{\Z}$ which will turn out to be the $g_{\Z}$ in the conclusion.  We establish the various properties in order.
\begin{claim*}
$f\ast \mu$ is $\frac{1}{4}$-almost integer-valued \emph{i.e.} $\|k - f\ast \mu\|_{L_\infty(G)} \leq \frac{1}{4}$.
\end{claim*}
\begin{proof}
Suppose that there is some $x \in G$ such that $|f\ast \mu(x)-k(x)| > \frac{1}{4}$.  Then
\begin{align*}
\|f_{\Z} - f\ast \mu(x)\|_{L_p(\tau_x(\nu\ast \widetilde{\nu}))} & \geq \|(f\ast \mu)_{\Z} - f\ast \mu(x)\|_{L_p(\tau_x(\nu\ast \widetilde{\nu}))}\\
& \geq \|(f\ast \mu)_{\Z} - f\ast \mu(x)\|_{L_p(\tau_x(\nu\ast \widetilde{\nu}))} > \frac{1}{4}
\end{align*}
which contradicts (\ref{eqn.o}).
\end{proof}
\begin{claim*}
$k$ is invariant under translation by elements of $B_\kappa'$.
\end{claim*}
\begin{proof}
Since $\mu$ is $B'$-approximately invariant it follows by the triangle inequality and Lemma \ref{lem.cc} that for all $y \in B_\kappa'$ and $x \in G$ we have
\begin{align*}
|k(y+x) - k(x)|& \leq | k(y+x) - f\ast \mu(y+x)| \\
& \qquad \qquad + |f\ast \mu(y+x) - f\ast \mu(x)| + |f\ast \mu(x) - k(x)|\\& \leq \frac{1}{2} + 2M\kappa <1.
\end{align*}
It follows that $k(y+x)=k(x)$ as claimed.
\end{proof}
The next two claims require the same calculation.  Put $\theta_x:=\tau_x(\nu \ast \widetilde{\nu})(\{y: f_\Z(y) \neq k(x)\})$ and note that
\begin{align*}
\|f_{\Z} - f\ast \mu(x)\|_{L_p(\tau_x(\nu\ast \widetilde{\nu}))}^p & \geq \int_{\{z: f_\Z(z) \neq k(x)\}}{\left||f_{\Z}(y)-k(x)| - |k(x) - f \ast \mu(x)|\right|^p d\tau_x(\nu \ast \widetilde{\nu})(y)}\\
& \geq \theta_x\left(\frac{3}{4}\right)^p.
\end{align*}
In light of (\ref{eqn.o}) we then have $\theta_x \leq 3^{-p}$. 
\begin{claim*}
$\|f\ast \nu \ast\tilde{\nu}-k\|_{L_\infty(G)} \leq \eta+\epsilon$ so that $f \ast \nu \ast \tilde{\nu}$ is $(\epsilon+\eta)$-almost integer-valued and $(f \ast \nu \ast \tilde{\nu})_\Z=k$.
\end{claim*}
\begin{proof}
By the triangle inequality we see that
\begin{align*}
|f \ast \nu \ast \tilde{\nu}(x)-k(x)| & \leq |f_{\Z} \ast \nu \ast \tilde{\nu}(x)-k(x)| + |(f-f_{\Z})\ast \nu \ast\tilde{\nu}(x)| \\
& \leq \theta_x\|f_{\Z}\|_{L_\infty(G)} +\epsilon \leq (M+\epsilon) \theta_x +\epsilon \leq 2M3^{-p}+\epsilon.
\end{align*}
It follows that $f \ast \nu \ast \tilde{\nu}$ is $(\eta+\epsilon)$-almost integer-valued in light of the choice of $p$.  Since $2M3^{-p}+\epsilon <\frac{1}{2}$ we see that the integer part is unique and so $(f \ast \nu \ast \tilde{\nu})_{\Z}=k$.
\end{proof}
\begin{claim*}
$k \not \equiv 0$.
\end{claim*}
\begin{proof}Since $\kappa \leq 1/2$ and $B' \leq B$ we see that $\supp \nu \ast \tilde{\nu} \subset B_1$, and hence by (\ref{eqn.sssss}) that
\begin{equation*}
1_{A'} \ast \nu \ast \tilde{\nu}(x) \geq \psi
\end{equation*}
for some $x \in G$.  If $k(x)= 0$ then
\begin{equation*}
\psi \leq 1_{A'} \ast \nu \ast \tilde{\nu}(x) \leq 1_{A} \ast \nu \ast \tilde{\nu}(x) = \tau_x(\nu \ast \widetilde{\nu})(\{y: f_\Z(y) \neq 0\}) =\theta_x \leq 3^{-p},
\end{equation*}
which contradicts the choice of  $p$.  It follows that $k(x) \neq 0$.
\end{proof}
\begin{claim*}
$\|k\|_{L_1(m_G)} \leq 2Mm_G(\supp f_\Z)$.
\end{claim*}
\begin{proof}
Note that
\begin{equation*}
|k(x)| - \|(f_\Z - f) \ast \mu\|_{L_\infty(G)} - \|f \ast \mu - k\|_{L_\infty(G)} \leq |f_{\Z} \ast \mu(x)|,
\end{equation*}
and so
\begin{equation*}
\frac{1}{2}\int{|k(x)dm_G(x)} \leq \int{|f_{\Z} \ast \mu(x)|dm_G(x)} \leq (M+\epsilon)m_G(\supp f_\Z).
\end{equation*}
\end{proof}
Write $H$ for the group generated by $B_\kappa'$ so that Lemma \ref{lem.covsum}, Lemma \ref{lem.dfc} part (\ref{pt4.difc}), and Lemma \ref{lem.bss} part (\ref{lem.bss.1}) tell us
\begin{align*}
m_G(H) & \geq m_G(B_{\kappa}') \geq \frac{m_G(A')}{\mathcal{C}(A';B_\kappa')} \geq \frac{m_G(A')}{\mathcal{C}^\Delta(A';B_\kappa')} \geq \left(\frac{\kappa}{4}\right)^{\dim B'}\frac{m_G(A')}{\mathcal{C}^\Delta(A';B_1')}\\
& \geq \exp(-O(M^{4}\log^82M + M^3\log \eta^{-1}(\log (2\log 2\eta^{-1}))))m_G(\supp f_\Z).
\end{align*}
From the claims, $k$ is $H$-invariant and so there is a well-defined function $z:G/H \rightarrow \Z$ such that $z(W)=k(w)$ for all $w \in W$.  Now we have from the claims that
\begin{equation*}
\|z\|_{\ell_1(G/H)}m_G(H)= \|k\|_{L_1(m_G)} \leq 2Mm_G(\supp f_\Z),
\end{equation*}
which gives
\begin{equation*}
\|z\|_{\ell_1(G/H)} \leq  \exp(O(M^{4}\log^82M + M^3\log \eta^{-1}(\log (2\log 2\eta^{-1})))).
\end{equation*}
It remains to put $g:=f\ast \nu \ast\tilde{\nu}$ and note that $g_{\Z}=k$ has the required properties.  Moreover, since $k$ is not identically $0$ we see that
\begin{equation*}
\|g\|_{A(G)} \geq \|g\|_{L_\infty(G)} \geq \|k\|_{L_\infty(G)} - (\epsilon + \eta) \geq 1-\epsilon - \eta,
\end{equation*}
and
\begin{align*}
\|f\|_{A(G)} & = \sum_{\gamma}{|\wh{f}(\gamma)|}\\
& = \sum_{\gamma}{|\wh{f}(\gamma)|(1-|\wh{\nu}(\gamma)|^2)}+ \sum_{\gamma}{|\wh{f}(\gamma)||\wh{\nu}(\gamma)|^2}\\
& = \|f - f \ast \nu \ast \tilde{\nu}\|_{A(G)} + \|f \ast \nu \ast \tilde{\nu}\|_{A(G)} \geq \|f - f \ast \nu \ast \tilde{\nu}\|_{A(G)} -(1-(\epsilon+\eta)),
\end{align*}
from which we get the final inequality.
\end{proof}

We are now in a position to prove our main result.
\begin{proof}[Proof of Theorem \ref{thm.main2}]
We produce a sequence of functions $f_i$, reals $\epsilon_{i+1}$, subgroups $H_{i+1}$, and functions $z_{i+1}:G/H_{i+1} \rightarrow \Z$ such that
\begin{enumerate}
\item $\epsilon_{i} := 2^i\epsilon +4^{i-2M-4}\delta\exp(-C_{\textsc{M{\'e}l}}M)$;
\item $f_i$ is $\epsilon_i$-almost integer-valued;
\item $\|f_{i+1}\|_{A(G)} \leq \|f_{i}\|_{A(G)} - \frac{1}{1+\delta}$;
\item $(f_{i+1}-f_i)_{\Z} = \sum_{W \in G/H_{i+1}}{z_{i+1}(W)1_W}$.
\end{enumerate}
Set $f_0:=f$ and note that since $f$ is $\epsilon$-almost integer-valued it is certainly $\epsilon_0$-almost integer-valued.  At stage $i\leq 2M+1$ apply Lemma \ref{lem.mitlem} with parameter $\eta:=4^{-2M-3}\delta\exp(-C_{\textsc{M{\'e}l}}M)$, which is possible (provided $\epsilon$ is sufficiently small) since
\begin{equation*}
\epsilon_i \leq 2^{2M+1}\epsilon +4^{2M+1-2M-4}\delta\exp(-C_{\textsc{M{\'e}l}}M) \leq \min\{\exp(-C_{\textsc{M{\'e}l}}M),\delta 2^{-3}\}.
\end{equation*}
Either $(f_i)_{\Z}\equiv 0$ and we terminate the iteration, or we get a function $f_{i+1}$, a group $H_{i+1}$ and a function $z_{i+1}:G/H_{i+1} \rightarrow \Z$, such that $f_{i+1}-f_i$ is $(\epsilon_i+\eta)$-almost integer-valued,
\begin{equation*}
(f_{i+1}-f_i)_{\Z} = \sum_{W \in G/H_{i+1}}{z_{i+1}(W)1_W},
\end{equation*}
\begin{equation*}
\|z_{i+1}\|_{\ell_1(G/H_{i+1})} \leq \exp(O(M^{4}\log^82M + M^3\log \delta^{-1}(\log (2\log 2\delta^{-1}))))
\end{equation*}
and
\begin{equation*}
\|f_{i+1}\| \leq \|f_i\|_{A(G)} - (1-(\epsilon_i+\eta))\leq  \|f_i\|_{A(G)} -\frac{1}{1+\delta}.
\end{equation*}
Since $f_i$ is $\epsilon_i$-almost integer-valued it follows that $f_{i+1}$ is $(2\epsilon_i+\eta)$-almost integer-valued.  But
\begin{align*}
(2\epsilon_i+\eta) & \leq 2( 2^i\epsilon +4^{i-2M-4}\delta\exp(-C_{\textsc{M{\'e}l}}M)) +4^{-2M-3}\delta\exp(-C_{\textsc{M{\'e}l}}M)\\
& \leq 2^{i+1}\epsilon + 4^{(i+1)-2M-4}\delta\exp(-C_{\textsc{M{\'e}l}}M),
\end{align*}
and so $f_{i+1}$ is $\epsilon_{i+1}$-almost integer-valued.

Since $\|f_i\|_{A(G)} \geq 0$ we must have $(f_l)_\Z\equiv 0$ for some $l \leq M(1+\delta)$.  But then
\begin{align*}
& \left\|f - (f_l)_{\Z} - \sum_{j=0}^{l-1}{(f_{j+1}-f_j)_{\Z}}\right\|_{L_\infty(G)}\\
&\qquad  \leq \left\|f - f_l - \sum_{j=0}^{l-1}{(f_{j+1}-f_j)}\right\|_{L_\infty(G)}\\
&\qquad  \qquad  + \left\|f_l - (f_l)_{\Z}\right\|_{L_\infty(G)} + \sum_{j=0}^{l-1}{\left\|(f_{j+1}-f_j) - (f_{j+1}-f_j)_{\Z}\right\|_{L_\infty(G)}}\\
& \qquad = 0+ \epsilon_l + \sum_{j=0}^{l-1}{(\epsilon_{j}+\eta)} \leq \exp(O(M))\epsilon + \frac{1}{4}<\frac{1}{2},
\end{align*}
provided $\epsilon$ is sufficiently small.  The result follows since $f_{\Z}$ is uniquely defined in this case and $(f_i)_\Z \equiv 0$ when the iteration terminates.
\end{proof}

\section{Specific classes of groups}\label{sec.rel}

In this section we discuss work for specific classes of groups.  

\subsection{Groups of bounded exponent}  In \cite{gre::9} Green set out a model setting for additive combinatorics.  (See \cite{wol::3} for a recent perspective.)  In this setting a number of arguments simplify and Theorem \ref{thm.main2} could be proved for groups of bounded exponent without the need for any discussion of Bohr systems.  

As mentioned in the introduction \cite{san::18} carries out this simplification for finite groups of exponent $2$ -- \emph{i.e.} groups isomorphic to $\F_2^n$ for some $n$ -- though more general (Abelian) groups of bounded exponent are no harder.
\begin{theorem}\label{thm.f2}
Suppose that $G=\F_2^n$ and $f:G \rightarrow \Z$ has $\|f\|_{A(G)} \leq M$.  Then there is some $z:\mathcal{W}(G) \rightarrow \Z$ such that
\begin{equation*}
f=\sum_{W \in {\mathcal{W}}(G)}{z(W)1_W} \text{ and } \|z\|_{\ell_1(\mathcal{W}(G))} \leq \exp(M^{3+o(1)}).
\end{equation*}
\end{theorem}
In certain regimes there are already stronger results, at least for indicator functions of sets.  Indeed Shpilka, Tal, and Lee Volk established the following in \cite{shptalvol::1}.
\begin{theorem}[{\cite[Theorem 1.2]{shptalvol::1}}]\label{thm.stv}
Suppose that $G=\F_2^n$ and $A \subset G$ has $\|1_A\|_{A(G)} \leq M$.  Then there is some $z:\mathcal{W}(G) \rightarrow \Z$ such that
\begin{equation*}
1_A=\sum_{W \in {\mathcal{W}}(G)}{z(W)1_W} \text{ and } \|z\|_{\ell_1(\mathcal{W}(G))} \leq \exp(O(M^2 + M\log \log |G|)).
\end{equation*}
\end{theorem}
While our aim is to avoid any sort of $|G|$ dependence, it is worth noting that in the above theorem it is really rather mild.

It is also interesting that for this class of groups arithmetic progressions are no longer a limiting example -- we do not have Proposition \ref{prop.egg} -- and it might be that the bound on $\|z\|_{\ell_1(\mathcal{W}(G))}$ can be polynomial in $M$.  Some efforts in this direction for particular classes of function can be found in work of Tsang, Wong, Xie and Zhang, in particular \cite[Corollary 7]{tsawonxie::}.

\subsection{Cyclic groups of prime order}

For cyclic groups of prime order there are a range of results by Konyagin and various authors.  In particular the following is an easy consequence of \cite[Theorem 1.3]{grekon::}.
\begin{theorem}
Suppose that $G=\Z/p\Z$ and $A \subset G$ has $m_G(A) = \alpha \in \left(0,\frac{1}{2}\right]$.  Then
\begin{equation*}
\|1_A\|_{A(G)} = \alpha \log^{\frac{1}{3}-o(1)}p.
\end{equation*}
\end{theorem}
The above bound becomes weaker quite quickly as $A$ gets smaller, and Konyagin and Shkredov \cite{konshk::0,konshk::1} have the following results to deal with this.
\begin{theorem}[{\cite[Theorem 13]{konshk::0}}] Suppose that $G=\Z/p\Z$ and $A \subset G$ has size $2 \leq |A| \leq \exp((\log p/\log \log p)^{1/3})$.  Then
\begin{equation*}
\|1_A\|_{A(G)} = \Omega(\log |A|).
\end{equation*}
\end{theorem}
\begin{theorem}[{\cite[Theorem 3]{konshk::1}}] Suppose that $G=\Z/p\Z$ and $A \subset G$ has density $\alpha$ with $\exp((\log p/\log \log p)^{1/3}) \leq |A| \leq p/3$.  Then
\begin{equation*}
\|1_A\|_{A(G)} = \Omega(\log \alpha^{-1} )^{1/3 - o(1)}.
\end{equation*}
\end{theorem}
The arguments behind these results are not restricted to indicator functions of sets and the results themselves have been extended by Gabdullin in \cite{gab::}; that paper also develops some higher dimensional analogues.

In $\Z/p\Z$ there are no non-trivial subgroups and so these three results can be combined to give the following.
\begin{theorem}[Green-Konyagin-Shkredov]
Suppose that $G=\Z/p\Z$ and $A \subset G$ has $\|1_A\|_{A(G)} \leq M$ for some $M \geq 1$.  Then there is some $z:\mathcal{W}(G) \rightarrow \Z$ such that
\begin{equation*}
1_A=\sum_{W \in {\mathcal{W}}(G)}{z(W)1_W} \text{ and } \|z\|_{\ell_1(\mathcal{W}(G))} \leq \exp(\exp(M^{3+o(1)})).
\end{equation*}
\end{theorem}
Note that this is already a strengthening of the main result of \cite{gresan::0} in the particular case of groups of prime order, and this has been further strengthened by Schoen in \cite{sch::9} who showed the above with a bound of the form $\exp(M^{16+o(1)})$ by combining Konyagin and Shkredov's work more effectively.

In fact Konyagin and Shkredov's results are much sharper if one takes $A$ to be sparse.  For example, they combine to give the following.
\begin{theorem}[Konyagin-Shkredov]
Suppose that $G=\Z/p\Z$ and $A \subset G$ has $\|1_A\|_{A(G)} \leq M$ for some $M \geq 1$ and $|A| \leq p^{9/10}$.  Then there is some $z:\mathcal{W}(G) \rightarrow \Z$ such that
\begin{equation*}
1_A=\sum_{W \in {\mathcal{W}}(G)}{z(W)1_W} \text{ and } \|z\|_{\ell_1(\mathcal{W}(G))} \leq \exp(M^{3+o(1)}).
\end{equation*}
\end{theorem}
This is stronger than our main theorem in this particular case of small sets in groups of prime order.

\subsection{Torsion-free groups}

For a non-vacuous discussion of torsion-free groups we need to have a definition of $A(G)$ for infinite groups.  This is virtually the same, but see the start of \S\ref{sec.con} for the formal details.  Konyagin \cite{kon::} and McGehee, Pigno and Smith \cite{mcgpigsmi::} resolved the Littlewood conjecture by proving the following in our language.
\begin{theorem}\label{thm.mpsk}
Suppose that $G=\Z$ and $f \in A(G)$ is integer-valued.  Then there is some $z:\mathcal{W}(G) \rightarrow \Z$ such that
\begin{equation*}
f=\sum_{W \in {\mathcal{W}}(G)}{z(W)1_W} \text{ and } \|z\|_{\ell_1(\mathcal{W}(G))} \leq \exp(O(\|f\|_{A(G)})).
\end{equation*}
\end{theorem}
In fact some work has been done on the constant behind the big-$O$ term.  Stegeman \cite{ste::1} and Yabuta \cite{yab::0} independently give a bound of the shape
\begin{equation*}
 \|z\|_{\ell_1(\mathcal{W}(G))} \leq\exp\left(\left(c\frac{\pi^3}{4}+o(1)\right)\|f\|_{A(G)}\right).
\end{equation*}
for some $c<1$.  It must be that $c\geq\pi^{-1}$ in view of the size of the Lebesgue constants (see \cite[(16.)]{fej::0}).

\section{Cohen's idempotent theorem}\label{sec.con}

In this section we extend our work to locally compact Abelian groups; suppose that $G$ is such.  Then we write $\wh{G}$ for the (locally compact Abelian group \cite[\S1.2.6, Theorem (d)]{rud::1}) of continuous homomorphisms $G \rightarrow S^1$.  We say $f$ is an element of $B(G)$ if there is a measure $\mu \in M(\wh{G})$ such that
\begin{equation*}
f(x)=\int{\gamma(x)d\mu(\gamma)} \text{ for all }x \in G,
\end{equation*}
and $f \in A(G)$ if there is a representation of the above form in which $\mu$ is absolutely continuous with respect to the Haar measure on $\wh{G}$.  We write $\|f\|_{B(G)}:=\|\mu\|$ which is well-defined since the choice of $\mu$, if it exists, is unique \cite[\S1.3.6]{rud::1}.  We also put $\|f\|_{A(G)}=\|f\|_{B(G)}$ if $f \in A(G)$ and
\begin{equation*}
\mathcal{W}(G):=\bigcup_{H \leq G \text{ open}}{G/H},
\end{equation*}
and note that if $G$ is finite these definitions agree with those in the introduction.

A \textbf{ring of sets on $G$} is a subset of $\mathcal{P}(G)$ including $G$, and closed under complements and finite intersections (and hence finite unions by de Morgan's laws).  $\mathcal{P}(G)$ is the standard example of a ring of sets on $G$.  Another easy example is $\mathcal{A}(G):=\{A \subset G: 1_A \in B(G)\}$:

A short calculation \cite[\S3.1.2]{rud::1} shows that if $W \in \mathcal{W}(G)$ then $W \in \mathcal{A}(G)$ and $\|1_W\|_{B(G)}=1$.  It follows from the triangle inequality for $\|\cdot \|_{B(G)}$ that if $A \in \mathcal{A}(G)$ then $\neg A \in \mathcal{A}(G)$ since $1_{\neg A} = 1_G -1_A$; and it follows from the sub-multiplicativity of $\|\cdot\|_{B(G)}$ that $A \cap B \in \mathcal{A}(G)$ if $A,B \in \mathcal{A}(G)$ since $1_{A\cap B}=1_A \cdot 1_B$.

The \textbf{coset ring of $G$} is the intersection of all rings of sets on $G$ containing $\mathcal{W}(G)$.  This is a ring, and by the above is contained in $\mathcal{A}(G)$.  Cohen's idempotent theorem is the following converse.
\begin{theorem}[{\cite[\S3.1.3]{rud::1}}]\label{thm.cidem}
Suppose that $A \in \mathcal{A}(G)$.  Then $A \in \mathcal{L}(G)$.
\end{theorem}
To give a quantitative version of this we need a more constructive view of $\mathcal{L}(G)$.  With an eye to our later results we take a slightly more complicated definition than one might at first choose.

Given $H \leq G$ and $\mathcal{S} \subset G/H$ we write $\mathcal{S}^*:=\mathcal{S} \cup \left\{\neg \bigcup{\mathcal{S}}\right\}$, that is the partition of $G$ into cells from $\mathcal{S}$ and an additional cell that is everything else.  We say that $A$ has a \textbf{$(k,s)$-representation} if there are open subgroups $H_1,\dots,H_k \leq G$, and sets $\mathcal{S}_1\subset G/H_1,\dots,\mathcal{S}_k\subset G/H_k$ of size at most $s$ such that $A$ is the (disjoint) union of some cells in the partition\footnote{Recall that if $\mathcal{P}$ and $\mathcal{Q}$ are partitions of the same set then $\mathcal{P} \wedge \mathcal{Q}:=\{P \cap Q: P \in \mathcal{P}, Q \in \mathcal{Q}\}$.} $\mathcal{S}_1^* \wedge \cdots \wedge \mathcal{S}_k^*$.

We write $\mathcal{W}_{k,s}(G)$ for the set of sets with $(k,s)$-representations.  It can be shown fairly directly that $\bigcup_{k}{\mathcal{W}_{k,s}(G)}=\mathcal{L}(G)$ for any $s \in \N$, but as this also follows from what we are about to show we omit the details.

The triangle inequality and sub-multiplicativity of $\|\cdot\|_{B(G)}$ gives that each cell in the partition has algebra norm at most $(s+1)^k$ and there are at most $(s+1)^k$ cells so
\begin{equation}\label{eqn.basiccalc}
\|1_A\|_{B(G)} \leq (s+1)^{2k} \text{ for all }A \in \mathcal{W}_{s,k}(G).
\end{equation}
We shall prove the following converse.
\begin{theorem}[Quantitative idempotent theorem]\label{thm.quantcoh}
Suppose that $\|1_A\|_{B(G)} \leq M$ and $\delta \in(0,1]$ is a parameter.  Then $A \in \mathcal{W}_{k,s}(G)$ where
\begin{equation*}
k \leq M(1+\delta) \text{ and }s \leq \exp(O(M^4\log^82M + M^2\log \delta^{-1}(\log 2\log2\delta^{-1}))).
\end{equation*}
\end{theorem}
We shall prove this after the proof of the next result.

The earlier work of this paper concerned integer-valued functions, not just $\{0,1\}$-valued functions, and we now turn to these.  We say that $f:G \rightarrow \C$ has an \textbf{$(l,L)$-representation} if there are open subgroups $H_1,\dots,H_l \leq G$ and functions $z_1:G/H_1 \rightarrow \Z, \dots, z_l:G/H_l \rightarrow \Z$ such that
\begin{equation}\label{eqn.lLrep}
f=\sum_{i=1}^l{\sum_{W \in G/H_i}{z_i(W)1_W}} \text{ and } \max_i{\|z_i\|_{\ell_1(G/H_i)}} \leq L.
\end{equation}
Note that in this case $f$ is necessarily integer-valued.  

By the triangle inequality and the aforementioned calculation \cite[\S3.1.2]{rud::1}, if $f$ has an $(l,L)$-representation then $\|f\|_{B(G)} \leq lL$.  We shall bootstrap our main result to give the following.
\begin{theorem}\label{thm.mgen}
Suppose that $G$ is a locally compact Abelian group and $f \in B(G)$ is integer-valued with $\|f\|_{B(G)} \leq M$ and $\delta \in (0,1]$ is a parameter.  Then $f$ has an $(M(1+\delta),L)$-representation where
\begin{equation*}
L \leq \exp(O(M^4\log^82M + M^2\log \delta^{-1}(\log 2\log2\delta^{-1}))).
\end{equation*}
\end{theorem}
\begin{proof}[Proof of Theorem \ref{thm.mgen}]
Our argument proceeds essentially as in \cite[Appendix A]{gresan::0}; recall that if $\Lambda \leq G$ then $\Lambda^\perp:=\{\gamma \in \wh{G}: \gamma(x)=1 \text{ for all }x \in \Lambda\}$, and $\mu$ is absolutely continuous w.r.t. $\nu$ if there is some $f \in L_1(\nu)$ such that $d\mu = fd\nu$.

We begin with a qualitative variant of our result, \cite[Theorem]{ameito::}.  This gives open subgroups $S_1,\dots,S_k\leq G$; mutually orthogonal measures $\mu_1,\dots,\mu_k \in M(\wh{G})$; natural numbers $R_i$; signs and $(\epsilon_{i,j})_{j=1}^{R_i}$, and elements $(x_{i,j})_{j=1}^{R_i}$ such that
\begin{equation}\label{eqn.up}
d\mu_i(\gamma) = \sum_{j=1}^{R_i}{\epsilon_{i,j}\gamma(x_{i,j})dm_i(\gamma)} \text{ for }1\leq i \leq k,
\end{equation}
where $m_i:=m_{S_i^\perp}$ is the Haar probability measure on the compact group $S_i^\perp$; and
\begin{equation}\label{eqn.dsty}
f(x)=\sum_{i=1}^k{\int{\gamma(x)d\mu_i(\gamma)}} \text{ for all }x \in G.
\end{equation}
Since the $\mu_i$ are mutually orthogonal we have
\begin{equation*}
\|f\|_{B(G)}=\sum_{i=1}^k{\|f_i\|_{B(G)}}.
\end{equation*}
In view of (\ref{eqn.up}) the functions $f_i$ are integer-valued.  The argument now proceeds as in the proof of \cite[Proposition A.1]{gresan::0}.
\end{proof}
If one wished to avoid appealing to Cohen's theorem in the proof above the key obstacle comes in \S\ref{sec.ac}.  The concept of arithmetic connectivity extends easily enough to locally compact Abelian groups (using, \emph{e.g.}, the definition of $B(G)$ developed by Eymard \cite[(2.14) Lemme]{eym::} for non-Abelian groups), but this does not lead to a statement about large energy directly because we do not yet have a natural measure with respect to which the support of $f$ is positive but finite.

\begin{proof}[Proof of Theorem \ref{thm.quantcoh}]
Apply Theorem \ref{thm.mgen} to get $k \leq M(1+\delta)$ open subgroups $H_1,\dots,H_k$ and functions $z_1:G/H_1 \rightarrow \Z, \dots, z_l:G/H_l \rightarrow \Z$ such that
\begin{equation*}
1_A=\sum_{i=1}^l{\sum_{W \in G/H_i}{z_i(W)1_W}}
\end{equation*}
and
\begin{equation*}
\max_i{\|z_i\|_{\ell_1(G/H_i)}} \leq \exp(O(M^4\log^82M + M^2\log \delta^{-1}(\log 2\log2\delta^{-1}))).
\end{equation*}
Let $\mathcal{S}_i:=\{W \in G/H_i: z_i(W) \neq 0\}$ for $1\leq i \leq k$ and note that $1_A$ is constant on cells of the partition $\mathcal{S}_1^*\wedge \cdots \wedge \mathcal{S}_k^*$, which gives the required result.
\end{proof}

Returning to Theorem \ref{thm.mgen}, taking $\delta =\frac{1}{2M}\left(\lfloor M+1\rfloor -M\right)\in (0,1]$ we have the following corollary.
\begin{corollary}\label{cor.imep}
Suppose that $G$ is a locally compact Abelian group and $f \in B(G)$ is integer-valued with $\|f\|_{B(G)} \leq M$.  Then $f$ has an $(M,O_M(1))$-representation.
\end{corollary}
This is best possible in the first parameter of the representation as can be seen by considering a disjoint union of cosets of subgroups $H_1,\dots,H_l \leq G$ where $|H_i+H_j : H_i\cap H_j|=\infty$ if $i \neq j$.  

It is important to note that the error term is \emph{not} monotonic in the $M$ parameter and this is necessarily the case: consider $A:=G \setminus \{0_G\}$ for $G$ a group whose order is a large prime.  Then $\|1_A\|_{B(G)} <2$ and so if we are to write $A$ as a sum of indicator function of cosets of at most $\|1_A\|_{B(G)}$ subgroups, then there can only be one subgroup and we can require arbitrarily many cosets of this as the prime $p$ increases.

Apart from Cohen's original proof \cite{coh::} of Theorem \ref{thm.cidem}, which is the proof on which Rudin's \cite[Chapter 3]{rud::1} is based, there are proofs of the idempotent theorem due to Amemiya and It{\^o} \cite{ameito::} (shortening Cohen's original argument), and Host \cite{hos::} also shortening Cohen's argument, but the main purpose of which is to beautifully extend it to non-Abelian groups.

As stated these results are trivial for finite groups and the arguments do not seem to immediately extend to give quantitative information.  Both Amemiya and It{\^o}'s and Host's are very soft; Cohen's less so.  That being said they do have non-trivial quantitative content in one respect and in particular they can all be used to prove the following theorem.
\begin{theorem}\label{thm.ai}
Suppose that $G$ is a locally compact Abelian group and $f \in B(G)$ is integer-valued.  Then there are integer-valued functions $f_1,\dots,f_l \in B(G)$ such that each $f_i$ has a $(1,O_f(1))$-representation,
\begin{equation}\label{eqn.sum}
f=\sum_{i=1}^l{f_i} \text{ and } \|f\|_{B(G)} = \sum_{i=1}^l{\|f_i\|_{B(G)}}.
\end{equation}
\end{theorem}
Here $O_f(1)$ is a finite constant depending on $f$.  This has the following corollary.
\begin{corollary}
Suppose that $G$ is a locally compact Abelian group and $f \in B(G)$ is integer-valued with $\|f\|_{B(G)} \leq M$.  Then $f$ has an $(M,O_f(1))$-representation.
\end{corollary}
This is slightly weaker than Corollary \ref{cor.imep} since there are multiple functions with the same algebra norm.

It is worth noting that one cannot guarantee equality in the right sum in (\ref{eqn.sum}) for finite groups unless $l=1$ -- the example following Corollary \ref{cor.imep} applies here too.  This means that we have to relax the requirement that the underlying measures -- that is the measures $\mu_i$ such that $f_i(x)=\int{\gamma(x)d\mu_i(\gamma)}$ are mutually orthogonal to simply a requirement that they are `quite' orthogonal.  In some respects this is what happens in our quantitative continuity argument in \S\ref{sec.qc}.

\section*{Acknowledgment}

My thanks to the referees for encouragement to motivate this topic and some very careful reading of the paper leading to numerous improvements.

\bibliographystyle{halpha}

\bibliography{references}

\end{document}